\documentclass[a4paper,12pt]{amsart}
\usepackage{amssymb}
\usepackage{latexsym}
\usepackage{amsfonts}
\usepackage{amsmath}
\usepackage{amsthm}
\usepackage{multicol}
\usepackage{constants}
\theoremstyle{definition}
\newtheorem{definition}{Definition}[section]
\newtheorem{pro}[definition]{Proposition}
\newtheorem{thm}[definition]{Theorem}
\newtheorem{corollary}[definition]{Corollary}
\newtheorem{lem}[definition]{Lemma}
\newconstantfamily{eps}{symbol=\epsilon}
\newconstantfamily{c}{symbol=c}

\newtheorem{remark}[definition]{Remark}

\def\e{\varepsilon}

\def\<{\mathop{<}}
\def\>{\mathop{>}}

\newcommand{\dist}{\mathrm{dist}\,}
\newcommand{\diam}{\mathrm{diam}\,}
\newcommand{\spt}{\mathrm{spt}\,}
\numberwithin{equation}{section}

\title[Mean curvature flow with transport term]{Existence and regularity of mean curvature flow with transport term
in higher dimensions}
\author[K. Takasao and Y. Tonegawa]{Keisuke Takasao and Yoshihiro Tonegawa}
\address{Department of Mathematics\\ Hokkaido University\\ Sapporo 060-0810 Japan} 
\email{takasao@math.sci.hokudai.ac.jp}
\address{Department of Mathematics\\ Hokkaido University\\ Sapporo 060-0810 Japan} 
\email{tonegawa@math.sci.hokudai.ac.jp}
\thanks{K. Takasao is supported by the JSPS
Research Fellowships for Young Scientists and
the JSPS Grant-in-Aid for the JSPS fellows 12J06075 and 
Y. Tonegawa is partially supported by JSPS Grant-in-aid for scientific research 21340033, 21224001, 23654057. }
\date{}
\begin{document}
\maketitle
\begin{abstract}
Given an initial $C^1$ hypersurface and a time-dependent vector field in a Sobolev space, 
we prove a time-global existence
of a family of hypersurfaces which start from the given hypersurface and which
move by the velocity equal to the mean curvature plus the given vector field.
We show that the hypersurfaces are $C^1$ for a short time and, even after some
singularities occur, almost everywhere $C^1$ away from the higher multiplicity region. 
\end{abstract}
\section{Introduction}
A family $\{M_t\}_{t\geq 0}$ of hypersurfaces in ${\mathbb R}^n$ is called mean curvature flow (hereafter abbreviated MCF) if 
the velocity vector $v$ of $M_t$ is equal to its mean curvature vector $h$ at each point and time, that is, 
\begin{equation}
v=h\ \ \mbox{on }M_t.
\label{mcf0}
\end{equation}
 As one of the fundamental geometric evolution problems,
the MCF has been studied by numerous researchers in the past few decades. 
One of many facets of investigations is the time-global existence 
question of such a family when given an initial hypersurface $M_0$. In general dimensions, there
exists a unique smooth family of MCF for finite time until singularities such as vanishing and pinch-off occur. 
Though the classical MCF ceases to exist at this point, it is well-known that a unique time-global solution $\{M_t\}_{t\geq 0}$ exists
in a weak viscosity sense \cite{giga1991,evans-spruck1991} despite the occurrence of singularities. 

In this paper, we are interested in an aspect of time-global existence theory for a related problem, and the question we ask is the following. 
Given an initial hypersurface $M_0$ and a vector field $u$, 
is there a family $\{M_t\}_{t\geq 0}$ of hypersurfaces whose velocity vector $v$ is equal to its mean curvature $h$ plus $u$? 
What is the minimum regularity assumption on $u$ for the existence and regularity of such a family?
To be more precise, since we would be interested in the normal velocity to see the motion, the requirement is
\begin{equation}
v=h+(u\cdot\nu)\nu \ \ \mbox{on }M_t
\label{mcf}
\end{equation}
where $\nu$ is the unit normal vector field of $M_t$ and $\ \cdot \ $ is the inner product in ${\mathbb R}^n$. 
Motivation to investigate \eqref{mcf} is more than just to see what happens when an extra lower order term is added. 
While the MCF is of premier importance, one 
wonders what is the limit of applicability of various analytic techniques developed for the MCF if one puts a wild perturbation.
In a reverse context, if one understands the limit of generality of the MCF, then some of the analytic techniques developed for 
more general settings may be useful for the MCF. In fact, our investigation on \eqref{mcf} has already led us to the development of a 
local regularity theory \cite{kasaitonegawa,tonegawa2012} which gives new insight to the MCF. 
Physically, 
one may regard \eqref{mcf} as a surface tension driven phase boundary motion with a given 
background transport effect such as
fluid flow or external force field. One can also find such motion law in a coupled system with the 
Navier-Stokes equation modeling a flow of dry foam (see, for example, \cite{kim} for the numerical simulation and
references therein). 

Though far from complete, in this paper we obtain
satisfactory time-global existence and regularity theorems if we assume that $M_0$ is $C^1$ and
$u$ satisfies
\begin{equation}\label{lowreg}
\big(\int_0^T\big(\int_{{\mathbb R}^n} |u(x,t)|^p+|\nabla u(x,t)|^p\, dx\big)^{\frac{q}{p}}\, dt\big)^{\frac{1}{q}}<\infty
\end{equation}
for all $T<\infty$, with $2<q<\infty$ and $\frac{nq}{2(q-1)}<p<\infty$ ($\frac43\leq p$ in addition if $n=2$).
Here $\nabla u=(\partial_{x_1}u,\cdots,\partial_{x_n}u)$ is the weak partial derivatives and $u,\nabla u$ are measurable
with the stated integrability. We prove that the hypersurfaces remain $C^1$ at least for a short time,
and it is a.e$.$ $C^1$ away from a region where $M_t$ develops higher multiplicities. With more regularity assumption
on $u$ such as H\"{o}lder continuity, we have $C^2$ instead of $C^1$ and \eqref{mcf}
is satisfied classically. For the precise statement of the regularity, see Theorem \ref{regreg}.

Here we briefly discuss our approach. If $u$ is regular enough with respect to $x$, for example 
Lipschitz continuous, the level set method approach works well with a good order preserving property
(see, for example, \cite{GGIS} and \cite[Sec. 4.8]{Giga}). Also for regular enough $u$, there are
a number of short time existence results which are often stated for the MCF but which can be extended
to include regular $u$:  (1) solving an evolution equation for the height
function from the reference initial manifold \cite{XYChen}, (2) solving equations for signed distance function
\cite{ES2} (and elaborated further in \cite{GG}), and (3) constructing an approximate solution by time-discrete
minimal movement \cite{almgren}, just to name a few examples. On the other hand, 
with irregular $u$, one can not expect the order preserving property in general and even the short time
existence of solution can be a serious issue.
Hence to characterize \eqref{mcf}, we take an approach pioneered by Brakke \cite{brakke} using the
notion of varifold from geometric measure theory. To construct a sequence of approximate 
solutions, we use the Allen-Cahn equation \cite{allen}
with an extra transport term coming from $u$, \eqref{fac}. Much of the 
analysis of the present paper concerns various $\e$-independent estimates of quantities associated with $\varphi_{\e}$.
We obtain a desired solution by taking a limit $\e\to 0$. Thus the interest of the present paper can 
be also the analysis of \eqref{fac} itself. Once we verify that the limit satisfies \eqref{mcf} in a weak sense of
varifold as in Brakke's formulation, we apply a local regularity theory developed in \cite{kasaitonegawa,tonegawa2012}
which is tailor-made for the present problem. To our knowledge, under the assumption \eqref{lowreg} of $u$, 
even the short time existence of $C^1$ solution seems new.  

As for the MCF in general, there are a number of books and papers some of which include up-to-date
research results on the subject and we mention \cite{Ambrosio1,Bellettini,Colding,Ecker,Giga,Mante,White4}.
Concerning a time-global existence for the MCF and the related problems, 
we mention \cite{almgren,brakke, giga1991,evans-spruck1991,ilmanen1994,luckhaus} and references therein.
While there are numerous works with varying generalities 
establishing the connection between the Allen-Cahn equation and the MCF (for example, \cite{bron1,chen,mot,ESS,Fife,rub}), 
analysis of the Allen-Cahn equation using geometric measure theory 
was pioneered by Ilmanen \cite{ilmanen1993} in which he proved that the limit surface measures are
rectifiable and satisfy \eqref{mcf0} in the sense of Brakke's formulation. The second author proved that the limit surface measures
are integral \cite{tonegawa2003}. There are a number of closely related works even if we restrict the scope within some
measure theoretic approach to the Allen-Cahn equation, and we further mention 
\cite{mugnai,sato,soner1,soner2} and refereces therein. 
The existence result of the present paper has been proved by Liu et al. \cite{tonegawa2010}
for $n=2,3$ and with more restrictive assumptions on $p$ and $q$. The limitation of the dimensions was due to the 
use of results by R\"{o}ger and Sch\"{a}tzle \cite{roger-schazle}, which gives a characterization of limit measures 
under an assumption of uniform $L^2$ bound of mean curvature-like quantity. In the present paper, we avoid using \cite{roger-schazle}, 
and we follow the line of proofs of \cite{ilmanen1993,tonegawa2003} combined with various estimates from \cite{tonegawa2010}.
This frees us from any dimensional restriction. As a special case, the first author investigated the graph-like problem of \eqref{mcf} with 
a better regularity assumption on $u$ and showed a unique short time existence \cite{takasao}. 

The paper is organized as follows. In Section 2 we set our notations and explain the main results. In Section 3 we briefly
discuss some heuristic aspects of the Allen-Cahn equation. Section 4 deals with the uniform upper density ratio bound 
and monotonicity formula, and this is the key to control the transport term subsequently. In Section 5, we show that
there exists a limit surface measure for all $t\geq 0$. Section 6 proves that the limit measure is rectifiable and this part owes
much to Ilmanen's work \cite{ilmanen1993}. In Section 7, we prove that the limit measure has integer density modulo
surface energy constant. There, the idea of proof goes back to \cite{tonegawa2000} and the parabolic version \cite{tonegawa2003}.
In Section 8 we prove the main results by combining all the results from previous four sections. We record our final remarks in the 
last Section 9. We intended the paper to be as self-contained as possible, only exception being the proof for regularity. There we cite
the main local regularity theorem which has a set of assumptions we need to check. 

\section{Preliminaries and Main results}
\subsection{Basic notation}
Let ${\mathbb N}$ be the set of natural numbers and ${\mathbb R}^+:=\{x\geq 0\}$. For $0<r<\infty$ and 
$a\in {\mathbb R}^k$ define
\begin{equation*}
B_r^k(a):=\{x\in {\mathbb R}^k\,:\, |x-a|<r\}.
\end{equation*}
We write $B_r^k:=B_r^k(0)$. When $k=n$, we omit writing $n$. We often identify ${\mathbb R}^{n-1}$
with ${\mathbb R}^{n-1}\times\{0\}\subset {\mathbb R}^n$. 
 On ${\mathbb R}^n$ we denote the Lebesgue measure by ${\mathcal L}^n$ and
 for $0\leq k\leq n$, the $k$-dimensional Hausdorff measure by ${\mathcal H}^{k}$. 
 Define $\omega_n:={\mathcal L}^n(B_1)$.
 Given a set $A\subset {\mathbb R}^n$ and a measure $\mu$, the restriction of $\mu$ to $A$ is denoted by $\mu\lfloor_A$.
 The characteristic function of $A$ is denoted by $\chi_A$. Symbol $\nabla$ always refers to a differentiation with 
 respect to the space variables. For a set of finite perimeter (see \cite{Giusti} for the definition) $A$, we denote
 the total variation measure of the distributional derivative $\nabla\chi_A$ by $\|\nabla\chi_A\|$. 
 
 Throughout the paper, we set $\Omega$ to be either ${\mathbb T}^n$, the $n$-dimensional unit torus, or
 ${\mathbb R}^n$. 
 For $\Omega={\mathbb T}^n$ we often regard $\Omega$ as the
 unit square $[0,1)\times\cdots\times [0,1)\subset {\mathbb R}^n$ where all the relevant quantities are extended periodically to the entire
 ${\mathbb R}^n$. Objects such as functions and sets in $\Omega$ 
 are understood implicitly in this manner. For any Radon measure
 $\mu$ on ${\mathbb R}^n$ and $\phi\in C_c({\mathbb R}^n)$ we often write $\mu(\phi)$ for $\int \phi\, d\mu$. We
 write ${\rm spt}\,\mu$ for the support of $\mu$. Thus $x\in {\rm spt}\,\mu$ if $\forall r>0$, $\mu(B_r(x))>0$. For
 $1\leq p\leq \infty$, we write $f\in L^p(\mu)$ if $f$ is $\mu$ measurable and
 $(\int |f|^p\, d\mu)^{1/p}<\infty$. We use the standard notation for Sobolev spaces such as $W^{1,p}(\Omega)$
 and $W^{1,p}_{loc}(\Omega)$ from \cite{GT}.
 
For $A,B\in{\rm Hom}({\mathbb R}^n;{\mathbb R}^n)$ which we identify with $n\times n$ matrices, we define
\[ A\cdot B:=\sum_{i,j}A_{ij}B_{ij} \qquad \text{and} \qquad |A|:=\sqrt{A\cdot A}. \]
$\|A\|$ denotes the operator norm. 
The identity of ${\rm Hom}({\mathbb R}^n;{\mathbb R}^n)$ is denoted by $I$. For $k\in {\mathbb N}$ with $k<n$, 
let ${\bf G}(n,k)$ be the space of $k$-dimensional subspaces of ${\mathbb R}^n$.
The orthogonal complement of 
$S\in {\bf G}(n,k)$ is denoted by $S^{\perp}\in {\bf G}(n,n-k)$. 
For $a\in \mathbb{R}^n$, $a\otimes a \in {\rm Hom}({\mathbb R}^n;{\mathbb R}^n)$ is the matrix with the entries $a_ia_j$ ($1\leq i,j\leq n$). 
For $S\in {\bf G}(n,k)$, we identify $S$ with the corresponding orthogonal projection of ${\mathbb R}^n$ 
onto $S$. In the case of $k=n-1$, we also identify $S\in {\bf G}(n,n-1)$ with the unit vector $\pm \nu\in {\mathbb S}^{n-1}$
which is perpendicular to $S$. Note that we may express the relation by $S=I-\nu\otimes\nu$. The
correspondence is a homeomorphism with respect to the naturally endowed topologies on ${\bf G}(n,n-1)$ and 
${\mathbb S}^{n-1}/\{\pm 1\}$. 
For $x,y\in {\mathbb R}^n$ and $t<s$ define
\begin{equation}
 \rho _{(y,s)} (x,t) := \frac{1}{(4\pi (s-t))^{\frac{n-1}{2}}} e^{-\frac{|x-y|^2}{4(s-t)}}, 
 \label{bhk}
 \end{equation}
which is the backward heat kernel with pole at $(y,s)$. 
\subsection{Varifolds}
\label{pnotation}
We recall some definitions from geometric measure theory and refer to \cite{allard, brakke, ilmanen1993} for more details. 
For any open set $U\subset {\mathbb R}^n$ let $G_{k}(U):=U\times {\bf G}(n,k)$. 
A general $k$-varifold in $U$ is a Radon measure on 
$G_k(U)$. We denote the set of all general $k$-varifolds in $U$ by $\mathbf{V}_k(U)$. For $V\in \mathbf{V}_k(U)$, let
$\|V\|$ be the weight measure of $V$, namely,
\begin{equation*}
\|V\|(\phi):=\int_{G_k(U)} \phi(x)\, dV(x,S),\,\,\, \forall \phi\in C_c(U).
\end{equation*}
We say $V\in \mathbf{V}_k(U)$ is rectifiable if there exist a ${\mathcal H}^k$ measurable countably $k$-rectifiable set 
$M\subset U$ and a locally ${\mathcal H}^k$ integrable function $\theta$ defined on $M$ such that
\begin{equation}
V(\phi)=\int_{M} \phi(x,{\rm Tan}_x M)\theta(x)\, d{\mathcal H}^k
\label{recvar}
\end{equation}
for $\phi\in C_c (G_k(U))$. Here ${\rm Tan}_x M$ is the approximate tangent space of $M$ at $x$ which
exists ${\mathcal H}^k$ a.e$.$ on $M$. Rectifiable $k$-varifold is uniquely determined by its
weight measure $\|V\|=\theta\,{\mathcal H}^{n-1}\lfloor_{M}$ through the formula \eqref{recvar}. 
For this reason, we naturally say a Radon measure $\mu$ on $U$ is rectifiable when one can associate
a rectifiable varifold $V$ such that $\|V\|=\mu$. 
If $\theta\in {\mathbb N}$, ${\mathcal H}^k$ a.e$.$ on $M$, we
say $V$ is integral. The set of all integral $k$-varifolds in $U$ is denoted by ${\bf IV}_k(U)$. If $\theta=1$,  ${\mathcal H}^k$ a.e$.$ on $M$, we say $V$ is a unit density $k$-varifold. 

For $V\in \mathbf{V}_k(U)$ let $\delta V$ be the first variation of $V$, namely, 
\begin{equation}
\delta V(g):= \int _{G_k(U)} \nabla g(x) \cdot S \, dV(x,S)
\label{firstvariation}
\end{equation}
for $g\in C_c^1(U\,;\,{\mathbb R}^n)$. If the total variation $\|\delta V\|$ of $\delta V$ is locally bounded and absolutely
continuous with respect to $\|V\|$, by the Radon-Nikodym theorem, we have a $\|V\|$ measurable vector
field $h(V,\cdot)$ with
\begin{equation}
\delta V(g)=-\int_U g(x)\cdot h(V,x)\, d\|V\|(x).
\label{fvfa}
\end{equation}
The vector field $h(V,\cdot)$ is called the generalized mean curvature vector of $V$. For
any $V\in {\bf IV}_k(U)$ with an integrable $h(V,\cdot)$, Brakke's perpendicularity theorem
\cite[Chapter 5]{brakke} says that we have
\begin{equation}
\int_U ({\rm Tan}_x M)^{\perp}(g(x))\cdot h(V,x)\, d\|V\|(x)= \int_U g(x)\cdot h(V,x)\, d\|V\|(x)
\label{hperp1}
\end{equation}
for all $g\in C_c(U;{\mathbb R}^n)$. Here, $M$ is related to $V$ as in \eqref{recvar}. In the case of $k=n-1$, note that
$({\rm Tan}_x M)^{\perp}=\nu(x)\otimes\nu(x)$ for $\|V\|$ a.e$.$ in $U$, where $\nu(x)$ is the
unit normal vector to ${\rm Tan}_x M$. With this notation, \eqref{hperp1}
may be written as
\begin{equation}
\int_U (g(x)\cdot \nu(x))(h(V,x)\cdot \nu(x))\, d\|V\|(x)=\int_U g(x)\cdot h(V,x)\, d\|V\|(x)
\label{hperp2}
\end{equation}
for $g\in C_c(U;{\mathbb R}^n)$. If $h(V,\cdot)\in L^2(\|V\|)$, by approximation, 
\eqref{hperp2} holds even for $g\in L^2(\|V\|)$. 

\subsection{Weak formulation of velocity}
Let $\{M_t\}_{t\geq 0}$ be a family of smooth hypersurfaces in $\Omega$ whose normal velocity is denoted by $v$.
To formulate the velocity in a weak sense, observe the following characterization of $v$:
a smooth normal vector field $\tilde v$ on $M_t$ is equal to $v$ if and only if 
\begin{equation}
\frac{d}{dt}\int_{M_t} \phi\, d{\mathcal H}^{n-1}\leq \int_{M_t} (\nabla\phi-h\phi)\cdot \tilde v+\partial_t \phi
\, d{\mathcal H}^{n-1}
\label{hperp3}
\end{equation}
holds for all $\phi\in C_c^1(\Omega\times[0,\infty);{\mathbb R}^+)$ and for all $t\geq 0$. Here 
$h$ is the classical mean curvature vector of $M_t$. 
To check this claim, after some calculation, one first sees that $v$ satisfies 
\eqref{hperp3} with equality. Conversely, if $\tilde v$ satisfies \eqref{hperp3}, and
already knowing that $v$ satisfies \eqref{hperp3} with equality, we obtain
\begin{equation*}
0\leq \int_{M_t}(\nabla\phi-h\phi)\cdot(\tilde v -v)\, d{\mathcal H}^{n-1}
\label{hperp4}
\end{equation*}
for $\phi\in C_c^1(\Omega;{\mathbb R}^+)$. For any $\hat x\in M_t$ and $\lambda>0$, let $\phi_{\lambda}
(y):=\lambda^{2-n}\phi(\lambda^{-1}(y-\hat x))$. Substitute $\phi_{\lambda}$ and let $\lambda\downarrow 0$.
Since $\lambda^{-1}(M_t-\hat x)\to {\rm Tan}_{\hat x} M_t$, we obtain
\begin{equation*}
0\leq \int_{{\rm Tan}_{\hat x} M_t} \nabla\phi \, d{\mathcal H}^{n-1} \cdot (\tilde v (\hat x)-v (\hat x)).
\label{hperp5}
\end{equation*}
The integration by parts shows $\int_{{\rm Tan}_{\hat x} M_t} \nabla\phi\, d{\mathcal H}^{n-1}\perp {\rm Tan}_{\hat x}
M_t$. On the other hand, 
one may choose this vector to be $-(\tilde v (\hat x)-v(\hat x))$, for example. Thus we have $\tilde v (\hat x)
=v(\hat x)$ and we complete the proof of the claim. The characterization \eqref{hperp3} motivates the following
definition.
\begin{definition} A family of varifolds $\{V_t\}_{t\geq 0}\subset {\bf V}_{n-1}(\Omega)$ is a generalized solution of
\eqref{mcf} if the following four conditions are satisfied. 
\begin{itemize}
\item[(a)] $V_t\in {\bf IV}_{n-1}(\Omega)$ for a.e$.$ $t\geq 0$.
\item[(b)] For all $T>0$,
\begin{equation}
\sup_{t\in [0,T]} \|V_t\|(\Omega)<\infty\ \ \mbox{ and }\sup_{t\in [0,T],\, B_r(x)\subset\Omega}\frac{\|V_t\|(B_r(x))}{\omega_{n-1} r^{n-1}}<\infty.
\label{exden}
\end{equation}
\item[(c)] For all $T>0$, 
\begin{equation}
\int_0^T dt \int_{\Omega} |h|^2+|u|^2 \, d\|V_t\|<\infty.
\label{uhl}
\end{equation}
\item[(d)] For all $\phi\in C^1_c(\Omega\times[0,\infty) ; {\mathbb R}^+)$ and $0\leq t_1<t_2<\infty$,
\begin{equation}
\|V_{t}\|(\phi(\cdot,t))\Big|_{t=t_1}^{t_2}\leq \int_{t_1}^{t_2}dt\int_{\Omega}
(\nabla\phi-h\phi)\cdot\{h+ (u\cdot\nu)\nu\}+\partial_t\phi\, d\|V_t\|
\label{mcfweak}
\end{equation}
holds, where we abbreviated $h(V_t,x)$ by $h$. 
\end{itemize}
\label{gsdef}
\end{definition} 
The condition (b) may appear out of place in the definition of velocity. In fact, if $u$ is $0$ or a bounded
function and if $\|V_0\|$ satisfies \eqref{exden}, one can derive \eqref{exden} as a 
consequence of \eqref{mcfweak} via Huisken's
monotonicity formula.
However, if $u$ is not bounded, it is not clear how
to obtain \eqref{exden} from \eqref{mcfweak}. The other important point is that, unless one
has \eqref{exden}, it is unclear how to make sense of \eqref{uhl} and \eqref{mcfweak}. 
The difficulty is, $u(\cdot,t)$ needs to be defined as a $\|V_t\|$ 
measurable function for a.e$.$ $t\geq 0$. In general, $u(\cdot,t)$ is assumed to be in some Sobolev
space on $\Omega$, and we need to define $\|V_t\|$ measurable $u(\cdot,t)$ as a trace function. 
If we have \eqref{exden}, we may define the trace using the following inequality.
\begin{thm}\label{thmmz}
For a Radon measure $\mu$ on $\mathbb{R}^n$ with $D:= \sup _{B_r(x)\subset{\mathbb R}^n}\frac{\mu (B_r(x))}{\omega _{n-1} r^{n-1}}$ and $1\leq p <n$,
\begin{equation}
\int_{{\mathbb R}^n}|\phi|^{\frac{p(n-1)}{n-p}}\, d\mu \leq c(n,p) D\Big( \int _{\mathbb{R}^n} |\nabla \phi |^p \, dx \Big)^{\frac{n-1}{n-p}}
\label{mzineq}
\end{equation}
holds for $\phi \in C^1 _c (\mathbb{R}^n)$. 
\end{thm}
See \cite{meyers1977} and \cite{ziemer1989} for the proof in the case of $p=1$. The above
inequality for $1< p<n$ may be derived by the H\"{o}lder and Sobolev inequalities. 

Suppose that we have \eqref{exden}. We only need to define $u$ as a function in $L_{loc}^2(\|V_t\|\times dt)$ to 
make sense of \eqref{uhl} and \eqref{mcfweak}. Since $W^{1,p'}_{loc}\subset W^{1,p}_{loc}$ if $p'>p$, we need
to consider only $1\leq p<n$. Using the H\"{o}lder inequality and \eqref{mzineq},
we obtain (with $D:= \sup _{B_r(x)\subset\Omega}\frac{\|V_t\| (B_r(x))}{\omega _{n-1} r^{n-1}}$)
\begin{equation}
\begin{split}
\int _{\Omega} |\phi| ^2 \, d\|V_t\| & \leq \big(\int_{\Omega} |\phi|^{\frac{p(n-1)}{n-p}}\, 
d\|V_t\|
\big)^{\frac{2(n-p)}{p(n-1)}} (\|V_t\|({\rm spt}\,\phi))^{\frac{pn+p-2n}{p(n-1)}} \\
&\leq 
 (c(n,p) D)^{\frac{2(n-p)}{p(n-1)}} \big(\int_{\Omega} |\nabla\phi|^p\, dx\big)^{\frac{2}{p}}(\|V_t\|({\rm spt}\,\phi))^{\frac{pn+p-2n}{p(n-1)}}.
\end{split}
\label{mz}
\end{equation}
for $\phi\in C^1_c(\Omega)$. Here, we also need to assume that 
\begin{equation}
p\geq \frac{2n}{n+1}
\label{newpcon}
\end{equation}
so that $\frac{p(n-1)}{n-p}\geq 2$.
Since we will assume \eqref{pqncon} in the next subsection, which implies $p>\frac{n}{2}$ in particular, 
\eqref{newpcon} will be relevant only for $n=2$ and we will assume $p\geq \frac43$ when $n=2$. 
With this restriction, we may define $u$ as an $L^2_{loc}(\|V_t\|\times dt)$ function on $\Omega\times[0,T]$ uniquely
as long as $u\in L^2_{loc}([0,\infty);( W_{loc}^{1,p} (\Omega))^n)$ 
by the standard density argument. 
The function $u$ in \eqref{uhl} and \eqref{mcfweak} is defined in this sense. 

\subsection{Main results}
First we present some existence result for \eqref{mcf} when given a vector field $u$ and an initial hypersurface
$M_0$. 
\begin{thm}\label{existence}
Suppose $n\geq 2$, 
\begin{equation}
2<q<\infty,\ \  \frac{nq}{2(q-1)}<p<\infty\ \ (\ \frac43\leq p\mbox{ in addition if $n=2$})
\label{pqncon}
\end{equation}
and $\Omega ={\mathbb R}^n\mbox{ or }\mathbb{T}^n$. Given any
\begin{equation}
u\in L^q_{loc}([0,\infty ) ;(W^{1,p} (\Omega))^n )
\label{intu}
\end{equation}
and a non-empty bounded domain $\Omega _0 \subset \Omega $ with
$C^1$ boundary $M_0=\partial \Omega_0$, there exist
\begin{enumerate}
\item a family of varifolds $\{V_t\}_{t\geq 0}\subset{\bf V}_{n-1}(\Omega)$ which is a generalized solution of \eqref{mcf} as in Definition \ref{gsdef}
with $\|V_0\|={\mathcal H}^{n-1}\lfloor_{M_0}$ and 
\item a function $\varphi \in BV _{loc} (\Omega \times [0,\infty)) \cap C^{\frac{1}{2}} _{loc} ([0,\infty);L^1 (\Omega))$ 
with the following properties. 
\begin{enumerate}
\item[(2a)] $\varphi (\cdot ,t)$ is a characteristic function for all $t\in [0,\infty)$,
\item[(2b)] $\|\nabla \varphi(\cdot,t)\|(\phi)\leq \|V_t\|(\phi)$ for all $t\in [0,\infty)$
and $\phi\in C_c(\Omega;{\mathbb R}^+)$, 
\item[(2c)] $\varphi (\cdot,0) = \chi _{\Omega _0}$ a.e. on $\Omega$,
\item[(2d)] writing $\|V_t\|=\theta {\mathcal H}^{n-1}\lfloor_{M_t}$
and $\|\nabla \varphi(\cdot,t)\|={\mathcal H}^{n-1}\lfloor_{\tilde M_t}$ for a.e$.$ $t>0$, we have
\begin{equation}
{\mathcal H}^{n-1}(\tilde M_t\setminus M_t)=0
\label{parity0}
\end{equation}
and
\begin{equation}
\theta(x,t)=\left\{ \begin{array}{ll} \mbox{even integer $\geq 2$} & \mbox{if } x\in M_t\setminus \tilde M_t, \\
\mbox{odd integer $\geq 1$} & \mbox{if } x\in \tilde M_t
\end{array}
\right.
\label{parity}
\end{equation}
for ${\mathcal H}^{n-1}$ a.e$.$ $x\in M_t$. 
\end{enumerate}
\item If $p<n$, then for any $T>0$, setting $s:=\frac{p(n-1)}{n-p}$, we have
\begin{equation}
\Big(\int_0^T \Big(\int_{\Omega} |u|^{s}\, d\|V_t\|\Big)^{\frac{q}{s}}\, dt\Big)^{\frac{1}{q}}<\infty.
\label{intu2}
\end{equation}
If $p=n$, then we have \eqref{intu2} locally for $U\subset\subset\Omega$ for any $2\leq s<\infty$
and if $p>$n, then we have \eqref{intu2} with $L^s$ norm replaced by $C^{1-\frac{n}{p}}$ norm on $\Omega$.
\item There exists $T_1>0$ such that $V_t$ has unit density for a.e. $t\in [0,T_1)$. In addition $\|\nabla \varphi(\cdot ,t) \| = \|V_t\| $ for a.e. $t\in[0,T_1)$.
\end{enumerate}
\end{thm}

The condition \eqref{pqncon} on $u$ is a dimensionally sharp condition in the following sense. 
Consider a natural parabolic change of
variables $\tilde x:=\frac{x}{\lambda}$ and $\tilde t:=\frac{t}{\lambda^2}$ with $\lambda>0$. Since $u$ is a velocity field,
it should behave just like $x/t$, thus it is natural to consider $\tilde u:=\lambda u$. Then we have
\begin{equation*}
\Big(\int_{0}^{\infty}\Big(\int_{{\mathbb R}^n}|\nabla u|^p\, dx\Big)^{\frac{q}{p}}\, dt\Big)^{\frac{1}{q}}=\lambda^{
\frac{n}{p}+\frac{2}{q}-2} \Big(\int_{0}^{\infty}
\Big(\int_{{\mathbb R}^n}|\nabla \tilde u|^p\, d\tilde x\Big)^{\frac{q}{p}}\, d\tilde t\Big)^{\frac{1}{q}}
\end{equation*}
and $\frac{n}{p}+\frac{2}{q}-2<0$ is equivalent to the second inequality in \eqref{pqncon}. 
This guarantees that $u$ 
locally behaves more like a perturbative term. In (3), if $p>n$, then the result follows from the standard Sobolev inequality on ${\mathbb R}^n$.

To understand what $V_t$ and $\varphi$ are, assume for a moment that
no singular behaviors occur and we have a smooth family $\{M_t\}_{t\geq 0}$ with the velocity given by \eqref{mcf}. 
Then we should have ${\rm spt}\, \|V_t\|=\partial \{\varphi(\cdot,t)=1\}=M_t$. Since \eqref{mcf} is stated in terms of
$V_t$, it may first appear that $\varphi$ is redundant. However, beside the fact that $\varphi$ is obtained naturally
from the approach of the present paper, it has a few important roles. First, $\varphi$ helps to
guarantee that $V_t$ is non-trivial. Since $\varphi(\cdot,t)$ is continuous in $L^1(\Omega)$ by (2), 
$\|\varphi(\cdot,t)\|_{L^1(\Omega)}$ cannot vanish instantaneously at some arbitrary time.  
As long as $\varphi(\cdot,t)$ is not identically zero or identically 1, $\|V_t\|$ is non-zero measure.
Note that, given arbitrary $t_0>0$, by re-defining
$V_t:= 0$ for all $t>t_0$, we obtain another generalized 
solution of \eqref{mcf} due to the inequality in \eqref{mcfweak}. Obviously, this is not a solution
we would like to obtain in the end. 
The second role of $\varphi$ is that it gives some restriction on the possible singularities of ${\rm spt}\,\|V_t\|$. For example,
consider in the $n=2$ case. One can see that a unit density $V_t$ cannot form a triple junction since 
$\partial \{\varphi(\cdot,t)=1\}$ cannot be a triple junction. Thus, having $\varphi$ as an auxiliary object may
be a useful tool to obtain some better regularity results. As for the actual occurrence of the 
higher multiplicities, Bronsard and Stoth \cite{bron2} showed that one can have solution with
$\theta\geq 2$ for a limit of the Allen-Cahn equation, thus we may indeed have such solution in general.

We next state the regularity property of ${\rm spt}\,\|V_t\|$, which is obtained as an application of 
\cite{kasaitonegawa,tonegawa2012}. To state the result, we recall some definitions from there.
\begin{definition} A point $x\in {\rm spt}\,\|V_t\|$ is said to be a $C^{1,\zeta}$ regular point if there exists
some open neighborhood $O$ in ${\mathbb R}^{n+1}$ containing $(x,t)$ such that $O\cap \cup_{s>0}
({\rm spt}\,\|V_s\|\times\{s\})$ is an embedded $n$-dimensional manifold with $C^{1,\zeta}$ regularity
in space and $C^{(1+\zeta)/2}$ regularity in time. Similarly, we define a $C^{2,\alpha}$ regular point
by replacing the respective regularities by $C^{2,\alpha}$ in space and $C^{1,\alpha/2}$ in time. 
\label{defregpt}
\end{definition}
\begin{thm}
\label{regreg}
Let $\{V_t\}_{t\geq 0}$ be as in Theorem \ref{existence}. 
\begin{itemize}
\item[(1)]
Suppose that there exist an open set $U\subset \Omega$ and an interval $(t_1,t_2)$ such that $V_t$ 
is unit density in $U$ for a.e$.$ $t\in (t_1,t_2)$. Then for a.e$.$ $t\in (t_1,t_2)$, there exists a closed set $G_t\subset U$ 
with ${\mathcal H}^{n-1}(G_t)=0$ such that $(U\cap {\rm spt}\, \|V_t\|)\setminus G_t$ is a set of $C^{1,\zeta}$ regular points
where $\zeta:= 2-\frac{n}{p}-\frac{2}{q}$ if $p<n$. If $p\geq n$,
one may take any $\zeta$ with $0<\zeta<1-\frac{2}{q}$. 
\item[(2)] There exists $T_2>0$ such that every point of ${\rm spt}\, \|V_t\|$ is a $C^{1,\zeta}$ regular point 
for all $t\in (0,T_2)$ (that is, $G_t=\emptyset$), where $\zeta$ is as in (1). 
\item[(3)] If $u$ is H\"{o}lder continuous with exponent $\alpha$ in the parabolic sense, i.e.,
\begin{equation*}
\sup_{\Omega\times[0,T]}|u|+\sup_{x,y\in \Omega, 0\leq t_1<t_2\leq T}\frac{|u(x,t_1)-u(y,t_2)|}{\max\{ |x-y|^{\alpha}, |t_1-t_2|^{\alpha/2}\}}<\infty
\ \ \ \mbox{ for all $0<T<\infty$},
\end{equation*}
then the same results for (1) and 
(2) hold true with $C^{1,\zeta}$ there replaced by 
$C^{2,\alpha}$ and \eqref{mcf} is satisfied pointwise. 
\item[(4)] We have $\lim_{t\downarrow 0} t^{-\frac12} {\rm dist}\, (M_0, {\rm spt}\, \|V_t\|)=0$ and
${\rm spt}\, \|V_t\|$ converges to $M_0$ in $C^1$ topology as $t\downarrow 0$. Namely, 
given $\e>0$ there exists a finite number of sets $\{U_i=x_i+O_i(B_r ^{n-1}\times (-r,r))\}_{i=1}^N$, where
$O_i$ is an orthogonal rotation and $x_i\in M_0$, such that $M_0\subset \cup_{i=1}^N U_i$, and $C^1$ norms of difference of graphs
representing $M_0$ and ${\rm spt}\, \|V_t\|$ 
over $x_i+O_i(B_r^{n-1})$ in $U_i$ are less than $\e$ for all sufficiently small $t>0$. 
\end{itemize} 
\end{thm}
The claim (1) says that wherever $V_t$ is unit density in some space-time neighborhood, 
${\rm spt}\,\|V_t\|$ is locally a hypersurface with regularity of $C^{1,\zeta}$ in space and
$C^{(1+\zeta)/2}$ in time, almost everywhere in space and time. We can guarantee by (2) that there is some
time interval $[0,T_2)$ such that ${\rm spt}\,\|V_t\|$ is a $C^{1,\zeta}$ hypersurface. We obtain a lower bound
on $T_2$ in terms of $M_0$ and the norm of $u$. On the other hand, $T_2$ may be much larger than
the lower bound and it is the time when a non-$C^{1,\zeta}$ regular point
occurs for the first time. In general, $T_2\leq T_1$ and it is plausible that some non-$C^{1,\zeta}$ regular point
first appears at $T_2$ but $V_t$ may remain unit density for some more time. The claim (4) shows that ${\rm spt}\,\|V_t\|$
has $C^1$ uniform regularity and convergence as $t\downarrow 0$. 
As for (3),
we first note that we can show the same existence results for H\"{o}lder continuous $u$ (and not in 
$L^q_{loc}([0,\infty);(W^{1,p}(\Omega))^n)$) as in 
Theorem \ref{existence}. In fact the proof is simpler if $u$ is bounded.
$C^{2,\alpha}$ regularity allows one to have pointwise mean curvature vector and velocity vector of ${\rm spt}\,
\|V_t\|$ and \eqref{mcf} is satisfied pointwise. At this point, we reach a well-defined PDE setting, and ${\rm spt}\,\|V_t\|$ 
is as regular as what the standard parabolic regularity theory shows depending on any additional regularity 
assumption imposed on $u$. 

\section{Allen-Cahn equation with transport term}
As stated in the introduction, the method of proof for the existence is to approximate \eqref{mcf} by the Allen-Cahn
equation with an extra transport term coming from $u$. 
Throughout the paper, we assume that a function
$W$ satisfies the following:
\begin{equation}
\mbox{$W:\mathbb{R}\to[0,\infty)$ is $C^3$ and $W(\pm 1)=W'(\pm 1)=0$. } 
\label{Was1}
\end{equation}
\begin{equation}
\mbox{ For some $\gamma\in(-1,1)$, $W'<0$ on $(\gamma,1)$ and $W'>0$ on $(-1,\gamma)$.} 
\label{Was2}
\end{equation}
\begin{equation}
\mbox{ For some $\alpha\in (0,1)$ and $\kappa >0$, $W''(x)\geq \kappa$ for all $1\geq |x|\geq \alpha$. }
\label{Was3}
\end{equation}
We also define a constant 
\begin{equation}
\sigma:=\int_{-1}^1 \sqrt{2W(s)}\, ds.
\label{Was4}
\end{equation}
Basically, above assumptions require $W$ to be W-shaped with non-degenerate two minima 
at $\pm 1$. Requiring \eqref{Was2} may appear non-essential, but it is used essentially in deriving 
an upper bound for $\xi_{\e}$ in Lemma \ref{betalem}. Any such $W$ satisfying above can be used.
The reader can take a concrete example such as $W(s)=(1-s^2)^2$ in the following. 

Given $u$ and $M_0$ as in Theorem \ref{existence}, the whole scheme of the present paper is to
approximate the motion law \eqref{mcf} by 
\begin{equation}
\partial_t \varphi_{\e} + u_{\e}\cdot \nabla\varphi_{\e}=\Delta \varphi_{\e}-\frac{W'(\varphi_{\e})}{\e^2},
\label{fac}
\end{equation}
where $\e>0$ is a small parameter tending to 0 and $u_{\e}$ is a smooth approximation of $u$. For 
readers who are not familiar with the Allen-Cahn equation, we give a quick heuristic argument. Assume
that $u$ is smooth and that we have a family of domains $\Omega_t$ with smooth boundaries $M_t=\partial \Omega_t$. 
Let $d(\cdot,t)$ be the signed distance function to $M_t$ so that $d(\cdot,t)>0$ inside
of $\Omega_t$. We let $\Psi\,:\, {\mathbb R}\to (-1,1)$ be an ODE solution of $\Psi''=W'(\Psi)$ with 
$\lim_{x\to\pm \infty}\Psi(x)=\pm 1$. Such solution exists and we may assume $\Psi(0)=0$.
If we postulate that $\varphi_{\e}(x,t)\approx \Psi(d(x,t)/\e)$ and $\varphi_{\e}$ satisfies \eqref{fac}, then we expect
that 
\begin{equation}
\Psi' \partial_t d+u_{\e}\cdot \Psi' \nabla d\approx \Psi' \Delta d+\e^{-1}(\Psi''|\nabla d|^2-W'(\Psi)).
\label{fac1}
\end{equation}
Since $d$ is a distance function, $|\nabla d|=1$, and the last two terms cancel each other. This leaves
\begin{equation}
\partial_t d+u_{\e}\cdot \nabla d\approx \Delta d.
\label{fac2}
\end{equation}
Due to the nature of distance function, evaluated on $M_t$, 
$\partial_t d$ is the outward velocity of $M_t$, $u_{\e}\cdot \nabla d$ is the inward normal component of $u_{\e}$ and
$\Delta d$ is the mean curvature of $M_t$. As $\e\to 0$, this approximation may be expected to get better, and
the relation \eqref{fac2} motivates that $\{\varphi_{\e}(\cdot,t)=0\}$ should
converge to $M_t$ which moves by \eqref{mcf}. This heuristic argument may be justified if we know in advance 
that there exists a smooth $M_t$ moving by \eqref{mcf}. Here, however, $u$ is not smooth and we aim to obtain 
a time-global existence result which necessitates a framework inclusive of singularities. This is the reason to 
use the language of varifold in this paper as was done first by Ilmanen \cite{ilmanen1993}. The basic approach is
to prove that $\varphi_{\e}$ satisfying \eqref{fac} has the property that
\begin{equation}
\mu^{\e}:= \Big(\frac{\e|\nabla\varphi_{\e}|^2}{2}+\frac{W(\varphi_{\e})}{\e}\Big)\, dx\approx \sigma N(x,t) {\mathcal H}^{n-1}\lfloor_{M_t} 
\label{fac3}
\end{equation}
when $\e$ is small and where $N(x,t)$ is some integer. 
At the same time we prove that the limiting measure of $\mu^{\e}$ satisfies \eqref{mcfweak}. 
The first key estimate to be established is the analogue of \eqref{exden} for $\varphi_{\e}$ which will be discussed in the next section.
\section{Density ratio upper bound and energy monotonicity formula}
\label{drub}
In this section, we prove the upper density ratio bound for diffused interface energy and energy monotonicity
formula which are crucial in the limiting process. 
Estimates in this section are similar to \cite[Section 3]{tonegawa2010} with some modifications.

\subsection{The upper density ratio bound}
We state the main theorem concerning the uniform density ratio upper bound independent of $\e$
of the Allen-Cahn equation with extra transport term. The proof takes the entire Section \ref{drub}. We establish the monotonicity
formula which is a perturbed version of Ilmanen's monotonicity formula for the Allen-Cahn equation 
(and Huisken's monotonicity formula for the
MCF \cite{huisken1990}) along the way. 
\begin{thm}\label{density}
Suppose $n\geq 2$, $\Omega={\mathbb T}^n$ or ${\mathbb R}^n$, $p,q$ satisfy \eqref{pqncon},
\begin{equation}
0<\beta<\frac12,
\label{betacon}
\end{equation}
$0<\e<1$ and $\varphi$ satisfies 
\begin{equation}
\partial_t\varphi+u\cdot\nabla\varphi=\Delta\varphi-\frac{W'(\varphi)}{\e^2}\hspace{1cm}
\mbox{on $\Omega\times [0,T]$},
\label{beq1}
\end{equation}
\begin{equation}
\varphi(x,0)=\varphi_0(x)\hspace{1cm}\mbox{on $\Omega$}.
\label{beq2}
\end{equation}
Assume $u\in C^{\infty}_c(\Omega\times[0,T])$, $\nabla^j\varphi,\, \partial_t\nabla^k\varphi\in C(\Omega\times[0,T])$ for $k\in \{0,1\}$
and $j\in \{0,1,2,3\}$.
Let $\mu_t^{\e}$ be a Radon measure on $\Omega$ defined by
\begin{equation}
\int_{\Omega}\phi(x)\, d\mu^{\e}_t(x):=\int_{\Omega}\phi(x)\left(\frac{\e|\nabla\varphi(x,t)|^2}{2}
+\frac{W(\varphi(x,t))}{\e}\right)\, dx
\label{defrad}
\end{equation}
for $\phi\in C_c(\Omega)$ and 
define 
\begin{equation}
D(t) :=\max\left\{1, \mu_t^{\e}(\Omega), \sup _{B_r(x)\subset\Omega} \frac{\mu^{\e}_t (B_r (x))}{\omega _{n-1} r^{n-1}}\right\}, \qquad t\in [0,T]. 
\label{defdens}
\end{equation}
Assume
\begin{equation}
\sup_{\Omega\times[0,T]}|\varphi|\leq 1,
\label{u1sup}
\end{equation}
\begin{equation}
\hspace{1cm}\sup_{\Omega}\e^i |\nabla^i\varphi_0|\leq \Cl[c]{c-0}\hspace{1cm}
\mbox{for $i\in \{1,2,3\}$},
\label{u1}
\end{equation}
\begin{equation}
\lim_{R\rightarrow\infty} R^k \|\varphi+1\|_{C^2(({\mathbb R}^n\setminus B_R)\times [0,T])}=0 
\mbox{ for any }k\in {\mathbb N}\mbox{ in case $\Omega={\mathbb R}^n$},
\label{extu1}
\end{equation}
\begin{equation}
\sup _{\Omega} \Big( \frac{\e |\nabla \varphi _0 |^2}{2} -\frac{W(\varphi _0 )}{\e}\Big)\leq \e ^{-\beta},
\label{ic}
\end{equation}
\begin{equation}
\sup_{\Omega \times [0,T]}|u|\leq  \e ^{-\beta}, \sup_{\Omega \times [0,T]}|\nabla u|\leq  \e ^{-(\beta+1)},
\label{u2}
\end{equation}
\begin{equation}
\|u \|_{L^q ([0,T];(W^{1,p}(\Omega))^n) }\leq \Cl[c]{c-1}
\label{u3}
\end{equation}
and
\begin{equation}
D(0)\leq D_0.
\label{dzerocon}
\end{equation}
Then there exist $D_1=D_1(\Cr{c-1},n,p,q,D_0,T)>0$ and $\Cl[eps]{eps-1}=\Cr{eps-1}(\Cr{c-1},
n,p,q,D_0,T,\Cr{c-0},\beta,W)>0$ such that 
\begin{equation}
\sup_{t\in [0,T]} D(t)\leq D_1
\label{denconcl}
\end{equation}
as long as $\e<\Cr{eps-1}$.
\end{thm}
\begin{remark}
If $u=0$, $\mu^{\e}_t(\Omega)$ is monotone
decreasing, thus it is straightforward to conclude that $\mu^{\e}_t(\Omega)$ is bounded uniformly independent of $\e$
if $\mu^{\e}_0(\Omega)$ is.
The uniform density ratio bound may be also obtained from Ilmanen's monotonicity formula.  
When $u\neq 0$, however, it is non-trivial even to conclude that the total energy $\mu^{\e}_t(\Omega)$ 
up to time $T$ has a uniform bound independent of $\e$. We will see that we need the density ratio bound to
estimate $\mu^{\e}_t(\Omega)$. 
\end{remark}
\subsection{Monotonicity formula}
In this subsection as a first step we obtain a modified monotonicity formula analogous to that of 
Ilmanen \cite{ilmanen1993}. 
It is still not a very useful formula due to the possible negative contribution coming from $\xi_\e$ defined below.
We will show that the negative contribution is small when $\e$ is small. 

To localize the computations, fix a radially symmetric cut-off function 
\begin{equation}
\eta(x) \in C_c ^\infty (B_{\frac{1}{2}} )\ \ \mbox{ with } \ \ \eta=1 \ \ \mbox{ on } \ \ B_{\frac{1}{4}}, \ \ \ 0\leq \eta \leq 1.
\label{etacut}
\end{equation}
Define
\begin{equation}
 \tilde \rho _{(y,s)}(x,t):=\rho _{(y,s)}(x,t) \eta (x-y)=\frac{1}{(4\pi (s-t))^{\frac{n-1}{2}}} e^{-\frac{|x-y|^2}{4(s-t)}}\eta (x-y), \qquad t<s, \ x,y\in \Omega 
 \label{defti}
 \end{equation}
and define 
\begin{equation}
 e_\e :=\frac{\e |\nabla \varphi|^2}{2}+ \frac{W(\varphi)}{\e} ,\qquad \xi_\e: =\frac{\e |\nabla \varphi|^2}{2}- \frac{W(\varphi)}{\e}. 
 \label{defti2}
 \end{equation}

\begin{pro}
Suppose that $\varphi$ satisfies \eqref{beq1}. With the notation of \eqref{defrad}, \eqref{defti}, \eqref{defti2}
and writing ${\tilde\rho}={\tilde\rho}_{(y,s)}(x,t)$, we have 
$\Cl[c]{c-2}$ depending only on $n$ such that 
\begin{equation}
\begin{split}
\frac{d}{dt} \int _{\Omega} \tilde \rho \, d\mu^{\e} _t (x) \leq & 
\frac12\int _\Omega \tilde \rho |u| ^2 \, d\mu^{\e} _t (x) 
+ \frac{1}{2(s-t)} \int _{\Omega} \xi _\e \tilde \rho \, dx  +\Cr{c-2} e ^{-\frac{1}{128(s-t)}}\mu^{\e} _t (B_{\frac12}(y))
\end{split}
\label{monoton}
\end{equation}
for $y\in \Omega$, $0<t<s<\infty$ and $t<T$.
\end{pro}
\begin{proof}
We define $L$ as follows and by \eqref{beq1}, 
\[ L:=\partial_t\varphi  + u\cdot \nabla \varphi = \Delta \varphi -\frac{W'(\varphi)}{\e ^2} .\] 
By integration by parts we have
\begin{equation}
\begin{split}
\frac{d}{dt} \int _{\Omega} e_{\e} \tilde\rho \, dx &
= \int _{\Omega}\{ e_\e \partial_t \tilde\rho -\e (L-u\cdot \nabla \varphi)(\nabla \tilde\rho \cdot \nabla \varphi +\tilde\rho L) \} \, dx \\
&= \int _{\Omega} \Big\{ e_{\e} \partial_t\tilde\rho  -\e \tilde\rho \Big(L+\frac{\nabla\tilde\rho \cdot \nabla \varphi  }{\tilde\rho } \Big)^2 +\e \Big( L\nabla \tilde\rho \cdot \nabla \varphi +\frac{(\nabla \tilde\rho \cdot \nabla \varphi)^2}{\tilde\rho} \Big)\\
&\qquad +\e \tilde\rho u \cdot \nabla \varphi \Big(L+\frac{\nabla\tilde\rho \cdot \nabla \varphi  }{\tilde\rho } \Big) \Big\} \, dx \\
&\leq  \int _{\Omega} \Big\{ e_\e \partial_t\tilde\rho +\e \Big( L\nabla \tilde\rho \cdot \nabla \varphi +\frac{(\nabla \tilde\rho \cdot \nabla \varphi)^2}{\tilde\rho} \Big)
+\frac{1}{4} \e \tilde\rho (u\cdot \nabla \varphi )^2 \Big\} \, dx .
\end{split}
\label{mono1}
\end{equation}
Moreover by integration by parts we obtain
\begin{equation}
\begin{split}
\int_{\Omega} \e L \nabla \tilde\rho \cdot\nabla \varphi \, dx
= \int  _{\Omega} -\e (\nabla \varphi\otimes \nabla \varphi) \cdot\nabla ^2 \tilde\rho + e_{\e} \Delta \tilde\rho \, dx.
\end{split}
\label{mono2}
\end{equation}
Substitution of \eqref{mono2} into \eqref{mono1} gives
\begin{equation}
\begin{split}
\frac{d}{dt}\int_{\Omega}e_{\e}\tilde\rho\, dx&\leq \int_{\Omega} (-\xi_{\e})(\partial_t\tilde \rho+\Delta\tilde\rho)
+\e |\nabla\varphi|^2\Big(\partial_t \tilde\rho +\Delta \tilde\rho \\ & -\frac{\nabla \varphi\otimes \nabla \varphi}{|\nabla \varphi |^2} 
\cdot\nabla ^2 \tilde\rho + \frac{(\nabla\tilde \rho \cdot \nabla \varphi )^2}{\tilde\rho |\nabla \varphi |^2} \Big)+\frac14 \e\tilde\rho
(u\cdot\nabla\varphi)^2\, dx.
\end{split}
\label{mono2.5}
\end{equation}
We remark that $\rho $ (without multiplication by $\eta$) satisfies the following:
\begin{equation}
\begin{split}
\partial_t \rho +\Delta \rho = -\frac{\rho}{2(s-t)} ,\qquad \partial_t \rho +\Delta \rho -\frac{\nabla \varphi\otimes \nabla \varphi}{|\nabla \varphi |^2} 
\cdot\nabla ^2 \rho + \frac{(\nabla \rho \cdot \nabla \varphi )^2}{\rho |\nabla \varphi |^2} =0.
\end{split}
\label{mono3}
\end{equation}
When one computes \eqref{mono3} with $\tilde \rho $ instead of $\rho$, we have additional terms
coming from 
differentiation of $\eta$. The integration of these terms can be bounded by 
$c \mu^{\e} _t (B_{1/2}(y)) e^{-\frac{1}{128(s-t)}}$ for $c=c(n)$ since $|\nabla^j\rho| \leq c(j,n) e^{-\frac{1}{128(s-t)}}$
for any $x,y\in\Omega$ with $|x-y|>\frac{1}{4}$ and $j=0,1$.
Thus, with an appropriate choice of $\Cr{c-2}$ depending only on $n$, 
we obtain \eqref{monoton}.
\end{proof}
\subsection{Some estimates on $\Omega\times [0,T]$}
\begin{lem}\label{est1}
Suppose that $\varphi$ satisfies \eqref{beq1}, \eqref{beq2}, \eqref{u1sup}, \eqref{u1} and \eqref{u2}. Then there exists $\Cl[c]{c-3}>0$ depending only on $n, \Cr{c-0}, W$ such that
\begin{equation}
\sup_{\Omega \times [0,T]} \e |\nabla \varphi|+\sup_{x,y\in \Omega,\ t\in [0,T]} \e^{\frac32} \frac{|\nabla\varphi(x,t)
-\nabla\varphi(y,t)|}{|x-y|^{\frac12}}\leq \Cr{c-3}. 
\label{estnablaphi}
\end{equation}
\end{lem}
\begin{proof}
Take any domain $B_{3\e}(x_0)\times[t_0,t_0+2\e^2]\subset \Omega\times [0,T]$. 
Define $\tilde \varphi  (x,t):=\varphi  (\e x+x_0,\e ^2 t+t_0)$ and $\tilde u  (x,t):=u  (\e x+x_0,\e ^2 t+t_0)$ for $(x,t)\in B_3\times[0,2]$. By \eqref{beq1} we have
\begin{equation}
\partial _t \tilde \varphi  + \e \tilde u \cdot \nabla \tilde \varphi  = \Delta \tilde \varphi  - W'(\tilde \varphi ).
\label{acepsilon}
\end{equation}
Using the estimate of \cite[p.342, Theorem 9.1]{ladyzhenskaja}, if $\partial_t v-\Delta v = f$ on $B_2\times [0,2]$ then we have
\begin{equation}
\|\partial_t v, \nabla ^2 v \|_{L^r(B_1\times [j,2])}\leq c(n,r) (\|f,\nabla v,v\|_{L^r (B_2 \times [0,2])}) +(1-j) \| v(\cdot,0) \|_{W^{2,r}(B_2)})
\label{lpestimate}
\end{equation}
for $j=0$ (up to $t=0$) or $j=1$ (interior estimate) and for $r\in (1,\infty)$. Let $\phi \in C_c ^1 (B_3)$ be a cut-off function and multiply $\phi ^2 \tilde \varphi$ to \eqref{acepsilon}, then by integration by parts, \eqref{u1sup}, \eqref{u1} and \eqref{u2}, 
we have
\begin{equation}
\int _0 ^2 \int _{B_2} |\nabla \tilde \varphi |^2 \,dxdt \leq c(W).
\label{wc2}
\end{equation}
Hence by \eqref{u1sup}, \eqref{u1}, \eqref{u2}, \eqref{lpestimate} ($r=2$) and \eqref{wc2} we obtain
\begin{equation*}
\int _0 ^2 \int _{B_1} |\tilde \varphi _t|^2 + |\nabla ^2 \tilde \varphi |^2 \,dxdt \leq c(n,\Cr{c-0},W).
\end{equation*}
By applying \eqref{lpestimate} to the equation
\[ \partial _t (\tilde \varphi _{x_i}) -\Delta \tilde \varphi_{x_i} =-\e \tilde u _{x_i} \cdot \nabla \tilde \varphi -\e \tilde u  \cdot \nabla \tilde \varphi _{x_i} - W''(\tilde \varphi )\tilde \varphi _{x_i}, \]
and using \eqref{u1sup}, \eqref{u1} and \eqref{u2} again, we obtain
\[ \int _0 ^2 \int _{B_1} |\nabla \tilde \varphi _t|^2 + |\nabla ^3 \tilde \varphi |^2 \,dxdt \leq c(n,\Cr{c-0},W). \]
Therefore we obtain the $W^{1,2}$ estimates of $\nabla \tilde \varphi $ on $B_1 \times [0,2]$, and by the Sobolev inequality we have
\[ \|\nabla \tilde \varphi\|_{L^{\frac{2(n+1)}{n-1}}(B_1 \times [0,2])}\leq c(n,\Cr{c-0},W). \]
We can use this estimate to \eqref{acepsilon} and \eqref{lpestimate} with $r=\frac{2(n+1)}{n-1}$. We repeat this argument until $r$
is large enough so that $ W^{1,r}\subset C^{\frac12}$ with appropriate modifications of the domain. Then we obtain the desired estimate
\[ \|\nabla \tilde \varphi \|_{C^{\frac12}(B_1\times [0,2])}\leq c(n,\Cr{c-0},W). \]
Since the domain was arbitrary, after returning to the original coordinate system, we obtain \eqref{estnablaphi}.
\end{proof}
\begin{lem}
There exists $\Cl[eps]{eps-2} =\Cr{eps-2} (n,W,\beta)>0$ such that, if $\e<\Cr{eps-2}$ and under the assumptions
of \eqref{betacon}-\eqref{beq2}, \eqref{u1sup}, \eqref{u1}, \eqref{ic} and \eqref{u2}, we have
\begin{equation}
\frac{\e |\nabla \varphi |^2 }{2}-\frac{W(\varphi )}{\e } \leq 10\e ^{-\beta} \qquad \text{on} \ \Omega \times [0,T].
\label{-beta}
\end{equation}
\label{betalem}
\end{lem}
\begin{proof}
Rescale the domain by $x\mapsto \frac{x}{\e}$ and $t\mapsto\frac{t}{\e^2}$. Under
the change of variables, we continue to use the same notations for $\varphi$ and $u$. Define 
\begin{equation}
\xi:=\frac{|\nabla \varphi |^2}{2} -W(\varphi)-G(\varphi),
\label{xi0}
\end{equation}
where $G$ will be chosen later. We compute $\partial_t \xi +\e u\cdot \nabla \xi -\Delta \xi $ and obtain 
\begin{equation}
\begin{split}
\partial_t \xi +\e u\cdot \nabla \xi -\Delta \xi
=& \nabla \varphi \cdot \nabla\partial_t \varphi  -(W'+G') \partial_t\varphi  +\e (u\otimes \nabla \varphi )\cdot \nabla ^2 \varphi -\e (W'+G')u\cdot \nabla \varphi \\
&-|\nabla ^2 \varphi | ^2 -\nabla \varphi \cdot \nabla (\Delta \varphi )
+(W' + G') \Delta \varphi +(W'' +G'') |\nabla \varphi |^2 .
\end{split}
\label{xi1}
\end{equation}
Here, we denoted and will denote $W'(\varphi)$ as $W'$,  $G(\varphi)$ as $G$ and so forth for simplicity. 
Differentiate \eqref{acepsilon} with respect to $x_j$, multiply $\varphi _{x_j}$ and sum over $j$ to obtain
\begin{equation}
\nabla \varphi \cdot \nabla \partial_t\varphi +\e \nabla u \cdot(\nabla \varphi \otimes \nabla \varphi )+\e (u\otimes \nabla \varphi) \cdot\nabla ^2 \varphi 
=\nabla \varphi \cdot \nabla (\Delta \varphi) -W'' |\nabla \varphi | ^2.
\label{acepsilon2}
\end{equation}
By \eqref{acepsilon}, \eqref{xi1} and \eqref{acepsilon2} we have
\begin{equation}
\partial_t \xi +\e u\cdot \nabla \xi -\Delta \xi
= W' (W'+G') -|\nabla ^2 \varphi |^2 -\e \nabla u \cdot( \nabla \varphi \otimes \nabla \varphi) +G'' |\nabla \varphi |^2.
\label{xi2}
\end{equation}
Differentiating \eqref{xi0} with respect to $x_j$ and by using the Cauchy-Schwarz inequality we have
\begin{equation}
\begin{split}
&\sum _{j=1} ^n \Big( \sum _{i=1} ^n \varphi _{x_i} \varphi _{x_i x_j} \Big) ^2
= \sum _{j=1} ^n (\xi _{x_j} +(W'+G')\varphi _{x_j} )^2 \\
&=|\nabla \xi | ^2  +2(W'+G') \nabla \xi \cdot \nabla \varphi + (W' +G')^2 |\nabla \varphi| ^2 \leq |\nabla \varphi |^2 |\nabla ^2 \varphi |^2. 
\end{split}
\label{xi3}
\end{equation}
On $\{|\nabla\varphi|>0\}$, divide \eqref{xi3} by $|\nabla \varphi |^2$ and substitute into \eqref{xi2} to obtain
\begin{equation}
\begin{split}
&\partial_t \xi +\e u \cdot \nabla \xi -\Delta \xi \\
\leq & W' (W'+G' ) -\frac{1}{|\nabla \varphi |^2} ( |\nabla \xi |^2 +2(W'+G') \nabla \xi \cdot \nabla \varphi + (W'+G')^2 |\nabla \varphi |^2 ) \\
 & -\e \nabla u \cdot( \nabla \varphi \otimes \nabla \varphi )+ G'' |\nabla \varphi|^2 \\
\leq & -(G') ^2 -W'G' -\frac{2(W'+G')}{|\nabla \varphi |^2} \nabla \xi \cdot \nabla \varphi -\e \nabla u \cdot( \nabla \varphi \otimes \nabla \varphi )+G''|\nabla \varphi|^2 . 
\end{split}
\label{xi4}
\end{equation}
By $|\nabla \varphi |^2 = 2(\xi+W+G)$ and \eqref{xi4} we have on $\{|\nabla\varphi|>0\}$
\begin{equation}
\begin{split}
\partial_t \xi +\e u\cdot \nabla \xi -\Delta \xi & \leq -(G')^2 -W'G' +2G'' (\xi+W+G) \\
&-\frac{2(W'+G')}{|\nabla \varphi |^2} \nabla \xi \cdot \nabla \varphi -\e \nabla u \cdot( \nabla \varphi \otimes \nabla \varphi).
\end{split}
\label{ineq-xi}
\end{equation}
Let $\phi(x,t)=\phi (x) \in C^\infty (B_{3\e ^{-1}})$ be such that
\begin{equation*}
\phi=\left\{
\begin{split}
&M:=\sup_{\mathbb{R}^n \times [0,\e ^{-2} T]} \Big( \frac{|\nabla \varphi |^2}{2} -W(\varphi) \Big) \qquad \text{on} \ B_{3\e ^{-1}}\setminus B_{2\e ^{-1}}, \\
&0 \qquad \text{on} \ B_{\e ^{-1}},
\end{split}
\right.
\end{equation*}
and
\[ 0\leq \phi \leq M, \ |\nabla \phi |\leq 2\e M , \ |\Delta \phi| \leq 2n\e ^2 M.  \]
Note that $M$ may be bounded depending only on $n, \Cr{c-0}, W$ by Lemma $\ref{est1}$. 
Note also that we may assume $M>0$ since $M\leq 0$ implies our conclusion \eqref{-beta} immediately. 
Let
\[ \tilde \xi :=\xi -\phi \ \text{and} \ G(\varphi):=\e ^{\frac{1}{2}} \Big(  1-\frac{1}{8} (\varphi-\gamma) ^2 \Big), \]
where $\gamma$ is as in \eqref{Was2}. 
To derive a contradiction, suppose that
\[ \sup _{B_{\e ^{-1}}\times [0, \e ^{-2} T]} \xi \geq \e ^{\frac{1}{2}} .\]
Since $\tilde \xi \leq 0$ on $ (B_{3\e ^{-1}} \setminus B_{2\e ^{-1}})\times [0, \e ^{-2} T]$, $\tilde \xi \leq \e ^{1-\beta} $ on $B_{3\e ^{-1}} \times \{ 0 \}$ by \eqref{ic} and 
\newline
$\sup_{B_{\e^{-1}} \times [0, \e ^{-2} T]} \tilde \xi \geq \e^{\frac{1}{2}}$, there exists some interior maximum point $(x_0,t_0)$ of $\tilde \xi $ where
\[  \partial_t \tilde \xi \geq 0, \ \nabla\tilde \xi =0, \ \Delta \tilde \xi \leq 0 \ \text{and} \ \tilde \xi \geq \e ^{\frac{1}{2}}\]
hold. By the definition of $\phi$ we have at the point $(x_0,t_0)$
\begin{equation}
\partial_t \xi \geq 0, \ |\nabla \xi |\leq 2\e M, \ \Delta \xi \leq 2n\e ^2 M \ \text{and} \ |\nabla \varphi |^2 \geq 2\e ^{\frac{1}{2}}.
\label{maxpoint}
\end{equation}
Substitute \eqref{maxpoint} into \eqref{ineq-xi}. Using $\e \nabla u \cdot( \nabla \varphi \otimes \nabla \varphi )\leq 2\e |\nabla u|(\xi +W+G)$ and \eqref{u2}, we have 
\begin{equation}
\begin{split}
0\leq 2n\e ^2 M -(G')^2 -W' G' + 2G''(\xi+W+G) +\frac{4(|W'| + |G'| )\e M}{(2\e ^{\frac{1}{2}})^{\frac{1}{2}}} \\
+ 2\e ^{1- \beta } (\xi +W+G)+2 \e ^{2-\beta}M.
\end{split}
\label{ineq-xi2}
\end{equation}
Since $\beta<\frac{1}{2}$ and $G''=-\e^{\frac12}/4$, for sufficient small $\e$ depending only on $\beta$
and $W$,
\begin{equation}
2G'' (\xi+W+G)
+ 2 \e ^{1-\beta} (\xi+W+G) \leq G'' (W+G).
\label{2G}
\end{equation}
If $|\varphi (x_0,t_0)|\leq \alpha$, then
\[ G''(\varphi(x_0,t_0)) W(\varphi(x_0,t_0))\leq -\frac{\e ^{\frac{1}{2}}}{4} \min_{|z| \leq \alpha} W(z), \]
which is a `big' negative number compared to the rest, and one can check that this and \eqref{2G} (as well as $W'G'\geq 0$
and $G>0$) 
lead to a contradiction in \eqref{ineq-xi2}. If $|\varphi (x_0,t_0)|\geq  \alpha$, then we would have `big' negative contributions coming from (all evaluated at $(x_0,t_0)$)
\[ (G') ^2 \geq \frac{\e(\alpha-|\gamma|)^2}{64} \ \text{and} \ -W'G'\leq -\frac{\e ^{\frac{1}{2}}(\alpha-|\gamma|)}{4} |W'|, \]
which again lead to a contradiction in \eqref{ineq-xi2} for sufficiently small $\e$. This shows that
\[ \sup _{B_{\e ^{-1}}\times [0, \e ^{-2} T]} \Big( \frac{|\nabla \varphi|^2}{2} -W(\varphi) \Big) \leq 2\e ^{\frac{1}{2}} ,\]
where $G\leq \e ^{\frac{1}{2}}$ is used. Now repeat the same argument, this time with $M$ replaced by $2\e ^{\frac{1}{2}}$ and $G$ replaced by $8\e ^{1-\beta} (1-\frac{1}{8} (\varphi-\gamma) ^2)$. If we assume
\[ \sup _{B_{\e ^{-1}}\times [0,\e ^{-2} T]} \xi \geq 2\e ^{1-\beta},\]
$\tilde \xi = \xi -\phi $ would attain some interior maximum in $B_{3\e ^{-1}}\times [0,\e ^{-2} T]$ by \eqref{ic} and by the subtraction of $\phi$. This time we would have $\partial_t \xi \geq  0, \ |\nabla \xi |\leq 4\e ^{\frac{3}{2}}, \ \Delta \xi \leq 4n\e ^{\frac{5}{2}}$ and $|\nabla \varphi| ^2 \geq 4\e ^{1-\beta}$. With this \eqref{ineq-xi2} is
\begin{equation*}
\begin{split}
0\leq 4n\e ^{\frac{5}{2}} -(G')^2 -W' G' + 2G''(\xi+W+G) +\frac{8(|W'| + |G'| )\e^{\frac{3}{2}} }{(4\e ^{1-\beta})^{\frac{1}{2}}} \\
+ 2\e ^{1- \beta } (\xi +W+G)+4 \e ^{\frac{5}{2}-\beta}.
\end{split}
\end{equation*}
Exactly the same type of argument as before shows that we have a contradiction, and since $G\leq 8  \e ^{1- \beta} $ and $\xi -G\leq 2\e ^{1-\beta}$, we have \eqref{-beta}.
\end{proof}

\begin{lem}
Let $\mu_s^{\varepsilon}$, $D(t)$ and ${\tilde\rho}_{(y,s)}$ be defined as in \eqref{defrad}, \eqref{defdens} and \eqref{defti}. 
Let $s,R,r$ be positive with $0\leq s-(\frac{R}{r})^2\leq T$ and $R\in (0,\frac{1}{2})$. Set $\tilde s= s-(\frac{R}{r})^2$. Then there exists $\Cl[c]{c-4}=\Cr{c-4}(n)\geq 1$ such that, for any $y\in \Omega$, we have
\begin{equation*}
\begin{split}
\int _{\Omega} \tilde \rho _{(y,s)}(x,\tilde s) \, d\mu^{\e} _{\tilde s} (x) \leq &
\Big( \frac{r}{\sqrt{4\pi}R}  \Big) ^{n-1} \Big\{ \mu^{\e} _{\tilde s}(B_R (y))+\mu^{\e} _{\tilde s}(B_{\frac12}(y)) \exp \Big( - \frac{r^2}{16 R^2 } \Big) \Big\} \\&+ \Cr{c-4} D(\tilde s)\exp \Big( - \frac{r^2}{8} \Big).
\end{split}
\end{equation*}
\label{sRr}
\end{lem}
\begin{proof}
First, on $B_R (y)$ we compute
\begin{equation*}
\begin{split}
\int _{B_R(y)} \tilde \rho _{(y,s)}(x,\tilde s) \, d\mu^{\e} _{\tilde s} 
\leq \Big( \frac{r}{\sqrt{4\pi}R}  \Big) ^{n-1}\int _{B_R(y)} e^{-\frac{r^2 |x-y|^2}{4R^2}} \, d\mu^{\e} _{\tilde s} \leq \Big( \frac{r}{\sqrt{4\pi}R}  \Big) ^{n-1}\mu^{\e}_{\tilde s}(B_R (y)).
\end{split}
\end{equation*}
On $\Omega \setminus B_R(y)$ we have
\begin{equation}
\begin{split}
&\Big( \frac{\sqrt{4\pi}R}{r}  \Big) ^{n-1} \int _{\Omega \setminus B_R(y)} \tilde \rho _{(y,s)}(x,\tilde s) \, d\mu^{\e} _{\tilde s} 
\leq \int_{B_{\frac{1}{2}} (y) \setminus B_R(y)} e^{-\frac{r^2 |x-y|^2}{4R^2}} \, d\mu^{\e} _{\tilde s} \\
\leq & \int _0 ^1 \mu^{\e} _{\tilde s} \Big(  \big( B_{\frac{1}{2}} (y) \setminus B_R(y)\big) \cap \{ x \ | \ e^{-\frac{r^2 |x-y|^2}{4R^2}}\geq \lambda \} \Big) \, d\lambda \\
\leq & \int _0 ^{\exp ( -\frac{r^2}{16R^2}) } \mu^{\e} _{\tilde s} ( B_{\frac{1}{2}} (y) \setminus B_R(y) ) \, d\lambda + \int _{\exp ( -\frac{r^2}{16R^2}) } ^{\exp ( -\frac{r^2}{4}) } \mu^{\e} _{\tilde s} ( B_{\frac{2R}{r} \sqrt {\log \lambda ^{-1}}} (y) ) \, d\lambda  \\
\leq & \mu^{\e} _{\tilde s}(B_{\frac12}(y)) e^{-\frac{r^2}{16R^2} } + D(\tilde s) \omega_{n-1} \Big( \frac{2R}{r} \Big)^{n-1} \int _{\frac{r^2}{4}} ^{\frac{r^2}{16R^2}}l^{\frac{n-1}{2}} e^{-l} \, dl \\
\leq & \mu^{\e} _{\tilde s}(B_{\frac12}(y)) e^{-\frac{r^2}{16R^2} }+ c(n) D(\tilde s)\Big( \frac{2R}{r} \Big)^{n-1} e^{-\frac{r^2}{8}}.
\end{split}
\label{estsRr}
\end{equation}
Here we used the fact that there exists $c=c(n)>0$ such that $l^{\frac{n-1}{2}} e^{-l}\leq c e^{-\frac{l}{2}}$ for any $l>0$.
\end{proof}
\subsection{Proof of Theorem $\ref{density}$}
In this subsection, we always work under the assumptions of Theorem \ref{density}. In particular, results from two preceding
subsections are available. Furthermore, from now on until Proposition \ref{time}, we assume 
\begin{equation}
D(t)\leq D_1
\label{cb2}
\end{equation}
holds for $t\in [0,T_1]$ and $T_1\leq T$. Here, $D_1\geq 2D_0$ is a constant depending only on $\Cr{c-1},n,p,q,T,D_0$,
and not on $\e$, and which will be determined after Proposition \ref{time}. We need to be careful about the
dependence of constants so that we do not end up a circular argument. Any constant depending on $D_1$ will be again
a constant depending on $\Cr{c-1},n,p,q,T,D_0$. 
Note that such $T_1>0$ exists because $D_1>D_0$ and by the continuity of $D(t)$ in time. Such continuity 
follows from that of $\varphi$ in the case of $\Omega={\mathbb T}^n$, and additionally from \eqref{extu1} in the case of
$\Omega={\mathbb R}^n$.  
$T_1$ may depend on $\e$ in general, but in the end, we prove that $T_1=T$ 
as long as $\e$ is sufficiently small.
First, under this assumption we have the following a-priori estimate:
\begin{lem}\label{erlem3}
There exists $\Cl[c]{c-n}$ depending only on $n,p,q$ such that for any $0\leq t_0<t_1$ we have
\begin{equation}
\begin{split}
\sup _{t\in [t_0,t_1]} \mu^{\e} _t (\Omega ) &+\frac12\int _{t_0} ^{t_1} \int _{\Omega} \e \Big( \Delta \varphi -\frac{W'(\varphi)}{\e ^2} \Big)^2 \,dxdt\\
&\leq \mu^{\e}_{t_0}(\Omega)+\Cr{c-n} (t_1-t_0)^{1-\frac{2}{q}} \|u\|^2_{L^q([t_0,t_1];(W^{1,p}(\Omega))^n)}\sup_{t\in [t_0,t_1]}
D(t).
\end{split}
\label{e0}
\end{equation}
In particular, there exists $E_0$ depending only on $\Cr{c-1},n,p,q,T,D_0$ such that
\begin{equation}
\sup_{t\in [0,T_1]} \mu_t^{\e}(\Omega)+\frac12\int_0^{T_1}\int_{\Omega}\e \Big( \Delta \varphi -\frac{W'(\varphi)}{\e ^2} \Big)^2 \,dxdt\leq E_0.
\label{e0sup}
\end{equation}
\end{lem}
\begin{proof}
By \eqref{beq1} we can compute
\begin{equation}
 \frac{d}{dt} \mu^{\e} _t (\Omega) \leq - \frac{1}{2} \int _\Omega \e \Big( \Delta \varphi -\frac{W'(\varphi)}{\e ^2} \Big)^2 \, dx +\e \int _\Omega (u\cdot \nabla \varphi )^2 \,dx. 
 \label{e01}
 \end{equation}
To estimate the last term of \eqref{e01}, we consider two cases $p<2$ and $p\geq 2$ separately. 
In addition we consider $\Omega={\mathbb T}^n,\,{\mathbb R}^n$ separately, and let us consider
${\mathbb T}^n$ first. Let $\{\psi_{\alpha}\}_{\alpha}$ be a partition of unity on 
$\Omega$ such that $\psi_{\alpha}\in C^{\infty}_c(\Omega)$,
${\rm diam}\,({\rm spt}\,\psi_{\alpha})\leq 1/2$
and $\|\psi_{\alpha}\|_{C^2}\leq c(n)$.  Consider $p<2$ case first. Just as in \eqref{mz}, by setting $s:=
\frac{p(n-1)}{n-p}\geq 2$, we have
\begin{equation}
\begin{split}
\e\int_{\Omega}(u\cdot \nabla\varphi)^2\, dx& \leq \big(\int_{\Omega} |u|^s\e|\nabla\varphi|^2\, dx\big)^{\frac{2}{s}}
(2\mu_t^{\e}(\Omega))^{1-\frac{2}{s}} \\
&\leq \big(\sum_{\alpha}c(n,p)\int_{\Omega} |\psi_{\alpha} u|^s\e |\nabla\varphi|^2\, dx\big)^{\frac{2}{s}}
(2D(t))^{1-\frac{2}{s}} \\
&\leq \big(\sum_{\alpha} c(n,p)D(t)\big( \int_{{\rm spt}\, \psi_{\alpha}}|u|^p+|\nabla u|^p\, dx\big)^{\frac{s}{p}}\big)^{\frac{2}{s}}
(2D(t))^{1-\frac{2}{s}}\\
&\leq c(n,p) D(t) \|u(\cdot,t)\|_{W^{1,p}(\Omega)}^2 
\end{split}
\label{e02}
\end{equation}
where each constants are different. We used the local finiteness of $\{\psi_{\alpha}\}_{\alpha}$ and 
$\sum_{\alpha} A_{\alpha}^{\frac{s}{p}}\leq (\sum_{\alpha} A_{\alpha})^{\frac{s}{p}}$ 
since $\frac{s}{p}\geq 1$. For $p\geq 2$, we have 
\begin{equation}
\begin{split}
\e\int_{\Omega}(u\cdot\nabla\varphi)^2\, dx&\leq \big(\int_{\Omega} |u|^p\e|\nabla\varphi|^2\, dx\big)^{\frac{2}{p}}
(2\mu_t^{\e}(\Omega))^{1-\frac{2}{p}}\\
&\leq \big(\sum_{\alpha} c(n,p)\int_{\Omega} |\psi_{\alpha} u|^p\, \e|\nabla\varphi|^2\, dx\big)^{\frac{2}{p}}
(2D(t))^{1-\frac{2}{p}} \\
&\leq \big( \sum_{\alpha} c(n,p) D(t) \int_{{\rm spt}\,\psi_{\alpha}} |u|^p+|u|^{p-1}|\nabla u|\, dx\big)^{\frac{2}{p}}
(2D(t))^{1-\frac{2}{p}} \\
&\leq  c(n,p)D(t)\|u(\cdot,t)\|_{W^{1,p}(\Omega)}^2.
\end{split}
\label{e03}
\end{equation} 
Here we used \eqref{mzineq} with $p=1$ there and $\phi=|\psi_{\alpha} u|^p$. Integration of \eqref{e0}
over $[t_0,t_1]$ using \eqref{e02} or \eqref{e03} gives \eqref{e0}. We define $E_0$ to 
be $D_0+\Cr{c-n} T^{1-\frac{2}{q}} \Cr{c-1}^2 D_1$. In case of $\Omega={\mathbb R}^n$, we do not need 
to take the partition of unity and the proof proceeds similarly.
\end{proof}
In the following we define $\beta'$ by
\begin{equation*}
\beta':=\frac{1+\beta}{2}.
\end{equation*}
In fact, any number $\beta'\in (\beta,1)$ can be used. To fix the idea, we specify such $\beta'$, and suppose that
$\beta'$ depends on $\beta$ for simplicity. 
\begin{lem}\label{lem-low}
There exist $\Cl[c]{c-5}>1$, $1>\Cl[c]{c-6}>0$ and $\Cl[eps]{eps-3}>0$ with $0<\Cr{eps-3} \leq \Cr{eps-2}$ depending only on $n,\Cr{c-0}, \Cr{c-1}, p, q, T, W, \beta$ and $D_0$ with the following property. Assume $\e \in (0,\Cr{eps-3})$ and $|\varphi (y,s)|\leq \alpha <1$ with $s\in (0,T_1]$. 
Here $\alpha$ is from \eqref{Was3}. Then for any $t\in [0,T_1 ]$ with $\max \{ 0,s-2\e ^{2\beta'} \} \leq t \leq s $ we have
\begin{equation}
\Cr{c-6} \leq \frac{1}{R^{n-1}} \mu^{\e} _t (B_R (y)),
\label{low}
\end{equation} 
where $R=\Cr{c-5} (s+\e ^2 -t)^{\frac{1}{2}}$.
\end{lem}
\begin{proof}
We will choose $\Cr{eps-3}<\Cr{eps-2}$ and assume for the moment that $\e<\Cr{eps-2}$. 
Set $\tilde \rho = \tilde \rho_{(y,s+\e ^2)} (x,t)$ in this proof. Assume $|\varphi (y,s)| \leq \alpha <1 $. We have
\[ \int_{\Omega} \tilde \rho \, d\mu^{\e} _s (x) = \int _{\e^{-1}\Omega} 
\frac{e^{-\frac{|\tilde x|^2}{4}}}{(\sqrt{4\pi} )^{n-1}} \eta (\e \tilde x) \Big(\frac{|\nabla \tilde \varphi |^2 }{2} +W(\tilde \varphi ) \Big) \, d\tilde x, \]
where $\tilde \varphi (\tilde x,s ) = \varphi (\e \tilde x +y ,s)$. 
By $|\tilde \varphi (0,s)| \leq \alpha <1$ and Lemma $\ref{est1}$ there exists $0<\Cl[c]{c-7}=\Cr{c-7}(n,\Cr{c-0},W)<1$ such that
\begin{equation}
5\Cr{c-7} \leq \int _{\Omega} \tilde \rho \, d\mu^{\e} _s (x).
\label{5c10}
\end{equation}
From \eqref{u2}, \eqref{monoton}, \eqref{-beta}, \eqref{e0sup} and $\e<\Cr{eps-2}$ we have for $\lambda \in [t,s )$
\begin{equation}
\frac{d}{d\lambda} \int _{\Omega} \tilde \rho \,d\mu^{\e} _\lambda
\leq \e ^{-2\beta} \int _\Omega \tilde \rho \, d\mu^{\e} _\lambda 
+ \frac{10\sqrt{\pi} \e ^{-\beta}}{\sqrt{s-\lambda}}
+ \Cr{c-2} e^{\frac{-1}{128(s+\e^2-\lambda)}}D_1 .
\label{monotonlambda}
\end{equation}
Here $\int_{\Omega} \tilde \rho \, dx\leq \sqrt{4\pi(s-t)}$ is used.
Multiply \eqref{monotonlambda} by $e^{ \e ^{-2 \beta}(s-\lambda)}$ and integrate over $[t,s]$. By $t\geq \max\{0,s-2\e ^{2\beta '}\}$ we have
\begin{equation}
e^{\e ^{-2\beta } (s-\lambda)} \int _\Omega \tilde \rho d\mu^{\e} _\lambda (x) \Big|_{\lambda=t}^s
\leq \e ^{\beta' -\beta} e^{2\e ^{2(\beta' -\beta)}}20\sqrt{2\pi} +2\Cr{c-2} D_1 e^{2\e ^{2(\beta' -\beta)}}e^{\frac{-1}{128(\e ^2+ 2\e ^{2\beta '})}}\e ^{2\beta'}.
\label{monotonlambda2}
\end{equation}
By \eqref{5c10} and \eqref{monotonlambda2} for sufficiently small $\e$ depending only
on $D_1,\beta,n$ and $\Cr{c-2}$ we have
\begin{equation}
2\Cr{c-7}\leq \int _\Omega \tilde \rho \, d\mu^{\e} _t (x).
\label{c10}
\end{equation}
Next we use Lemma \ref{sRr} with $ r:=\sqrt{ 8 \log (2\Cr{c-4} D_1 \Cr{c-7}^{-1})}$, where we may assume that $\Cr{c-4},D_1> 1$ and $\Cr{c-7}<1$. 
We chose this $r$ so that
\begin{equation}
\Cr{c-4} D_1 e^{-\frac{r^2}{8}} = \frac{\Cr{c-7}}{2}.
\label{c7}
\end{equation}
In Lemma $\ref{sRr}$, we replace $s$ and $s-(\frac{R}{r})^2$ by $s+\e ^2$ and $t$ respectively. Remark that $R:= r (s+\e ^2 -t )^{\frac{1}{2}}\leq r(\e ^2 +2 \e ^{2\beta '})^{\frac{1}{2}}$ since $s-t\leq 2\e ^{2\beta '}$. Hence we have $R< \frac{1}{2}$ by restricting $\e$ depending only on $\Cr{c-4},D_1$ and $\Cr{c-7}$. From \eqref{cb2}, \eqref{e0sup} and Lemma $\ref{sRr}$ we have
\begin{equation}
\int _\Omega \tilde \rho \,d\mu^{\e} _t (x) \leq (r/(\sqrt{4\pi} R))^{n-1} \{ \mu^{\e}_t (B_R (y)) + E_0 e^{-r^2 /(16R^2)} \} 
+\Cr{c-4} D_1 e^{-r^2/8}. 
\label{ad1}
\end{equation}
Note that $r/R\geq \e^{-\beta'}/\sqrt{3}$. By \eqref{ad1}, \eqref{c10} and \eqref{c7} for sufficiently small $\e$ we obtain
\[ \Cr{c-7} \leq (r/(\sqrt{4\pi} R))^{n-1} \mu^{\e} _t (B_R(y)).\]
Set $\Cr{c-5}:=r=\sqrt{ 8\log (2\Cr{c-4} D_1 \Cr{c-7}^{-1})}$ and $\Cr{c-6}=r^{1-n}(\sqrt{4\pi})^{n-1}\Cr{c-7}$ 
and we have the desired estimate \eqref{low}.
Note that the restriction on $\e$ depends on $\Cr{c-2}$, $\Cr{c-4}$, $D_1$, $\Cr{c-7}$. Examining the dependence, 
we may conclude the proof. 
\end{proof}

\begin{lem}
\label{disnic}
There exists $0<\Cl[eps]{eps-4}\leq \Cr{eps-3}$ and $\Cl[c]{c-8}$ depending only on $n,\Cr{c-0},\Cr{c-1},  p,  q, T, W, \beta$ and $D_0$ with the following property. For any $r\in (\e ^{\beta '} , \frac{1}{2})$ and $t\in [2\e ^{2\beta ' } ,T]\cap [0,T_1]$, we have
\begin{equation}
\int _{B_r (y)}\Big( \frac{\e |\nabla \varphi |^2}{2} -\frac{W(\varphi)}{\e}\Big) _{+} (x,t) \,
dx\leq \Cr{c-8} \e ^{\beta ' -\beta} r^{n-1} 
\label{beta'-beta}
\end{equation}
provided $\e \leq \Cr{eps-4}$.
\end{lem}
\begin{proof}
We only need to prove the claim when $T_1\geq 2\e ^{2\beta '}$ since the claim is vacuously true otherwise.
Let $y\in\Omega$, $r\in (\e^{\beta'},\frac12)$ and $t_\ast \in [2\e ^{2\beta'},T]\cap [0,T_1 ]$ be arbitrary and fixed. 
We define
\[ \tilde A:= \{ x\in B_{2r} (y) \, : \, \text{for some} \ \tilde t \ \text{with} \ t_\ast -\e ^{2\beta '} \leq \tilde t \leq t_\ast , \ |\varphi (x,\tilde t) |\leq \alpha \}, \]
\[ A:= \{ x \in B_{2r+2\Cr{c-5} \e ^{\beta '}}(y) \, : \, \dist (\tilde A,x)<2\Cr{c-5}\e ^{\beta'} \}. \]
By Vitali's covering theorem applied to $\mathcal{F} = \{ \bar B _{2\Cr{c-5} \e ^{\beta '}}(x) \, : \, x\in \tilde A \} \ (\text{note} \ A \subset \cup _{B\in \mathcal{F}} B )$, there exists a set of pairwise disjoint balls $\{ B_{2\Cr{c-5}\e ^{\beta '}} (x_i) \} _{i=1} ^{N}$ such that
\begin{equation}
x_i \in \tilde A \ \text{for each} \ i=1,\dots , N \ \ \text{and} \ \ A\subset \cup_{i=1} ^N \bar B _{10\Cr{c-5}\e ^{\beta '}} (x_i).
\label{tildeA}
\end{equation}
By the definition of $\tilde A$, for each $x_i$ there exists $ \tilde t_i$ such that
\begin{equation}
t_\ast -\e ^{2\beta'} \leq \tilde t_i \leq t_\ast , \ \ | \varphi (x_i,\tilde t_i)| \leq \alpha.
\label{tast}
\end{equation}
Define $\hat t:=t_\ast -2\e ^{2\beta '}$. Since $t_\ast\geq 2\e^{2\beta'}$, we have $\hat t\geq 0$. By \eqref{tast}, 
\begin{equation}
\e ^{2 \beta'} \leq \tilde t_i -\hat t\leq 2\e ^{2\beta '}
\label{bad1}
\end{equation}
and the assumption of Lemma $\ref{lem-low}$ is satisfied for $s=\tilde t_i , \ y= x_i , \ t=\hat t$ and $R_i:=\Cr{c-5} (\tilde t_i +\e ^2-\hat t)^{\frac{1}{2}}$ if $\e<\Cr{eps-3}$. Hence we may conclude that
\begin{equation}
\Cr{c-6}R_i^{n-1} \leq \mu^{\e} _{\hat t} (B_{R_i} (x_i))
\qquad \text{for} \ i=1,\dots,N.
\label{lowbdd}
\end{equation}
By \eqref{bad1}, we have 
$\Cr{c-5}(\e ^{2\beta'} +\e ^2)^{\frac{1}{2}} \leq R_i\leq   \Cr{c-5}(2\e ^{2\beta'} +\e ^2)^{\frac{1}{2}} \leq 2\Cr{c-5}\e^{\beta'}$, which shows 
\begin{equation}
\Cl[c]{c-9} \e^{\beta'(n-1)}\leq \mu^{\e}_{\hat t}(B_{2\Cr{c-5} \e^{\beta'}}(x_i))
\label{bad2}
\end{equation}
from \eqref{lowbdd} with $\Cr{c-9}:=\Cr{c-6} \Cr{c-5}^{n-1}$. 
 Since $\{B_{2\Cr{c-5} \e ^{\beta '}} (x_i ) \}_{i=1} ^N$ are pairwise disjoint and 
 $B_{2\Cr{c-5} \e ^{\beta '} } (x_i) \subset B_{2r+2\Cr{c-5} \e ^{\beta '}} (y)$, \eqref{bad2} gives 
\begin{equation}
N\Cr{c-9} \e ^{\beta ' (n-1)} \leq \mu^{\e} _{\hat t} (B_{2r+ 2\Cr{c-5} \e ^{\beta '}}(y)).
\label{lowbdd2}
\end{equation}
Hence the $n$-dimensional volume of $A$ is estimated by \eqref{tildeA} and \eqref{lowbdd2}
\begin{equation*}
\mathcal{L}^n (A) \leq N\omega _n (10\Cr{c-5}\e^{\beta '})^n \leq 
\frac{\omega _n (10\Cr{c-5})^n \e ^{\beta '}}{\Cr{c-9}} \mu^{\e} _{\hat t}( B_{2r+2\Cr{c-5}\e ^{\beta '}} (y)).
\end{equation*}
By \eqref{cb2} and $r\geq \e ^{\beta '}$,
\begin{equation}
\mathcal{L}^n (A) \leq \frac{\omega _n (10\Cr{c-5})^n \e ^{\beta '}}{\Cr{c-9}} D_1 \omega _{n-1} (2r +2\Cr{c-5}\e ^{\beta '}) ^{n-1} \leq \Cl[c]{c-10}\e ^{\beta '} r^{n-1} ,
\label{measA}
\end{equation}
where $\Cr{c-10}:= \omega _n \omega _{n-1} (10\Cr{c-5})^n(2+2\Cr{c-5})^{n-1} D_1 \Cr{c-9}^{-1}$. 
Hence by \eqref{-beta} and \eqref{measA}
\begin{equation}
\int _{A\cap B_r (y)} \Big( \frac{\e |\nabla \varphi |^2}{2}-\frac{W(\varphi)}{ \e} \Big)_{+} (x,t_\ast ) \, dx
\leq \mathcal{L}^n (A) 10 \e ^{-\beta}\leq 10 \Cr{c-10} \e ^{\beta ' -\beta}r^{n-1}.
\label{measA2}
\end{equation}
Next we estimate the surface energy on the complement of $A$ which decays very quickly. 
Define $\phi  \in \text{Lip} (B_{2r} (y)) $ such that
\begin{equation*}
\phi(x):=\left\{
\begin{split}
&1 \qquad \text{if} \ x\in B_r (y)\setminus A, \\
&0 \qquad \text{if} \ \dist (x,B_r (y)\setminus A) \geq \e ^{\beta '},
\end{split}
\right. \\
\end{equation*}
\[ |\nabla \phi |\leq 2\e ^{ -\beta '} \qquad \text{and} \qquad 0\leq \phi \leq 1.\]
By $r\geq \e ^{\beta '}$, $2\Cr{c-5}\e^{\beta'} >\e^{\beta'}$ and the definitions of $\tilde A$ and $\phi$,
we have $\spt \phi \cap \tilde A =\emptyset$, hence
\begin{equation}
|\varphi(x,s)|\geq \alpha , \qquad \text{for} \ x\in \spt \phi, \ s\in [t_\ast -\e ^{2\beta'},t_\ast ].
\label{phi>alpha}
\end{equation}
For each $j$ differentiate the equation \eqref{beq1} with respect to $x_j$, multiply $\phi^2 \frac{\partial \varphi }{\partial x_j}$, sum over $j$ and integrate to obtain
\begin{equation}
\label{name1}
\begin{split}
\frac{d}{dt} \int _{\Omega} \frac{1}{2} |\nabla \varphi |^2 \phi^2 \, dx
+ \int _{\Omega} (u\otimes \nabla \varphi \cdot \nabla ^2 \varphi + \nabla \varphi \otimes \nabla \varphi \cdot
\nabla u) \phi ^2 \, dx \\
=\int _\Omega \Big( \nabla \varphi \cdot \Delta \nabla \varphi -\frac{W''(\varphi)}{\e ^2}  |\nabla \varphi |^2 \Big)\phi ^2 \, dx.
\end{split}
\end{equation}
By integration by parts and the Cauchy-Schwarz inequality, \eqref{name1} gives
\begin{equation}
\label{name2}
\begin{split}
\frac{d}{dt} \int _{\Omega} \frac{1}{2} |\nabla \varphi |^2 \phi ^2 \, dx
 \leq & \frac{1}{2} \int _{\Omega} |u| ^2 |\nabla \varphi |^2 \phi  ^2 \, dx
 + \int _\Omega |\nabla \varphi |^2 |\nabla u| \phi ^2 \, dx \\
 &+ 4\int _\Omega |\nabla \phi| ^2 |\nabla \varphi |^2 \, dx 
 -\int _\Omega \frac{W'' (\varphi)}{\e ^2} |\nabla \varphi |^2 \phi  ^2 \, dx.
\end{split}
\end{equation}
By \eqref{phi>alpha}, $W''(\varphi)\geq \kappa$ on $\spt \phi $ for $t\in [t_\ast -\e ^{2\beta '} ,t_\ast]$. By \eqref{u2} and the definition of $\phi$, 
\eqref{name2} gives
\begin{equation}
\label{name3}
\begin{split}
&\frac{d}{dt} \int _{\Omega} \frac{1}{2} |\nabla \varphi |^2 \phi  ^2 \, dx \\
 \leq & \int _{\Omega} \Big( \frac{\e^{-2\beta}}{2} + \e ^{-1-\beta }\Big) |\nabla \varphi |^2 \phi  ^2 \, dx
 + 16  \e ^{-2\beta '} \int _{\spt \phi }|\nabla \varphi |^2 \, dx  -\frac{\kappa}{\e ^2}\int _\Omega  |\nabla \varphi |^2 \phi ^2 \, dx \\
 \leq & -\frac{\kappa}{2\e ^2} \int _\Omega |\nabla \varphi |^2 \phi  ^2 \, dx +16  \e ^{-2\beta '} \int _{\spt \phi } |\nabla \varphi| ^2 \,dx
\end{split}
\end{equation}
for small $\e$. By integrating \eqref{name3} over $[t_\ast-\e^{2\beta'},t_\ast]$, we obtain
\begin{equation}
\begin{split}
\int _{\Omega} \frac{1}{2} |\nabla \varphi |^2 \phi ^2 (x,t_\ast )\, dx
\leq & e^{-\kappa\e ^{2(\beta'-1)} } \int _{\Omega } \frac{1}{2} |\nabla \varphi |^2 \phi ^2 (x,t_\ast-
\e^{2\beta'}) \, dx \\
+ & \int _{t_\ast-\e^{2\beta'}} ^{t_\ast} e^{-\frac{\kappa}{\e ^{2}}(t_\ast-\lambda)} 16
 \e ^{-2 \beta '} \Big( \int _{\spt \phi } |\nabla \varphi |^2(x,\lambda) \, dx  \Big) \, d\lambda.  
\end{split}
\label{estcompA}
\end{equation}
Define
\[ M: = \sup _{\lambda \in [t_\ast -\e^{2\beta'},t_\ast ]} \int _{\spt \phi } \frac{1}{2} |\nabla \varphi |^2 (x,\lambda ) \, dx. \]
By \eqref{estcompA} we have
\begin{equation}
\int _\Omega \frac{1}{2} |\nabla \varphi |^2 \phi  ^2 (x,t_\ast) \, dx \leq 
(e^{-\kappa\e ^{2(\beta'-1)}} +32\kappa ^{-1} \e ^{2-2\beta '}  ) M.
\label{estcompA2}
\end{equation}
By $\spt \phi \subset B_{2r}(y)$ and \eqref{cb2}
\begin{equation}
\e M\leq  \omega_{n-1} D_1 (2r)^{n-1} .
\label{mj4}
\end{equation}
Since $B_r (y) \setminus A \subset \{ \phi =1 \} $, we have
\begin{equation}
\int _{B_r(y) \setminus A} \frac{\e}{2} |\nabla \varphi |^2 (x,t_\ast ) \, dx \leq \int _{\Omega} \frac{\e}{2} |\nabla \varphi |^2 (x,t_\ast ) \phi ^2 \, dx.
\label{mm1}
\end{equation}
Recall that $\beta'<1$. By \eqref{estcompA2}-\eqref{mm1}, we obtain for sufficiently small $\e$ (depending only on $\kappa$)
\begin{equation}
\int _{B_r(y) \setminus A} \frac{\e}{2} |\nabla \varphi |^2 (x,t_\ast ) \, dx \leq 33\kappa ^{-1}  \e^{2-2\beta'} D_1 \omega _{n-1} (2r)^{n-1}.
\label{estcompA3}
\end{equation}
By \eqref{measA2} and \eqref{estcompA3}, and since $\beta'-\beta=\frac{1-\beta}{2} < 2-2\beta'=1-\beta$, we obtain \eqref{beta'-beta}
with an appropriate choice of $\Cr{c-8}$.
\end{proof}
Later in Section \ref{muint}, we use the following estimate which follows from Lemma \ref{disnic}.
\begin{corollary}
For any $0<r<\frac12$, $\e\leq \Cr{eps-4}$ and $t\in [2\e^{2\beta'},T]\cap [0,T_1]$, we have
\begin{equation}
\int_0^r\frac{d\tau}{\tau^n} \int_{B_{\tau}(y)} \Big(\frac{\e|\nabla\varphi|^2}{2}
-\frac{W(\varphi)}{\e}\Big)_+(x,t)\, dx\leq \Cr{c-8}\e^{\beta'-\beta}|\log\e|
+10\omega_n \e^{\beta'-\beta}.
\label{atode1}
\end{equation}
\label{atode}
\end{corollary}
\begin{proof}
For the integration over the range $\tau\in (0,\e^{\beta'})$,
we simply use the estimate \eqref{-beta}. For the range $\tau\in (\e^{\beta'},r)$,
we use \eqref{beta'-beta}. 
\end{proof}
\begin{lem}
There exists a constant $\Cl[c]{c-11}$ depending only on $n, \Cr{c-0},\Cr{c-1}, p, q, T, D_0, W, \beta$ such that for $\e < 
\Cr{eps-4}$, $t\in [0,T_1 ]$ and $t<s$, we have
\begin{equation}
\int _0 ^t \Big\{ \frac{1}{2(s-\lambda )} \int _\Omega \Big( \frac{\e |\nabla \varphi| ^2}{2}-\frac{W(\varphi)}{\e} \Big)_{+} \tilde \rho _{(y,s)} (x,\lambda) \, dx \Big\} \, d\lambda \leq \Cr{c-11} \e ^{\beta' -\beta}|\log \e|.
\label{100}
\end{equation}
\label{discre1}
\end{lem}
\begin{proof}
If $t\leq 2\e ^{2\beta '}$ then by using \eqref{-beta} and $\int \rho \, dx =\sqrt{4\pi (s-\lambda)}$ we have
\begin{equation}
\begin{split}
\int _0 ^t \Big\{ \frac{1}{2(s-\lambda )} \int _\Omega \Big( \frac{\e |\nabla \varphi| ^2}{2}-\frac{W(\varphi)}{\e} \Big)_{+} \tilde \rho _{(y,s)} (x,\lambda) \, dx \Big\} \, d\lambda \leq \int _0 ^t \frac{10 \e ^{-\beta} \sqrt{\pi}}{\sqrt{s-\lambda} } \, d\lambda \\
\leq 20\sqrt{2\pi} \e ^{\beta '- \beta}.
\end{split}
\label{near0}
\end{equation}
By the similar argument, if $s>t\geq s-2\e ^{2\beta '}$ then we have
\begin{equation}
\begin{split}
\int _{s-2\e^{2\beta'}} ^t \Big\{ \frac{1}{2(s-\lambda )} \int _\Omega \Big( \frac{\e |\nabla \varphi| ^2}{2}-\frac{W(\varphi)}{\e} \Big)_{+} \tilde \rho _{(y,s)} (x,\lambda) \, dx \Big\} \, d\lambda \leq 20 \sqrt{2\pi} \e ^{\beta '-\beta}.
\end{split}
\label{nears}
\end{equation}
Hence we only need to estimate integral over $[2\e ^{2\beta'},t]$ with $t\leq s-2\e ^{2\beta'}$. First we estimate on $B_{\e ^{\beta '}} (y)$. We compute using \eqref{-beta} and 
$s-t\geq 2\e^{2\beta'}$ that
\begin{equation}
\begin{split}
\int _{2\e ^{2\beta '}} ^t \frac{1}{2(s-\lambda )} \int _{B_{\e ^{\beta '}}} \Big( \frac{\e |\nabla\varphi| ^2}{2}-\frac{W(\varphi)}{\e} \Big)_{+} \tilde \rho  \, dxd\lambda 
&\leq \int _{2\e ^{2\beta '}} ^t \frac{10 \e ^{-\beta}\e ^{n\beta '} \omega _n }{2(s-\lambda)^{\frac{n+1}{2}} (\sqrt{4\pi})^{n-1}} \,d\lambda \\
&\leq \frac{10 \e ^{\beta ' -\beta } \omega _n}{(\sqrt{8\pi})^{n-1}(n-1)}.
\end{split}
\label{middle1}
\end{equation}
On $\Omega \setminus B_{\e ^{\beta '} }(y)  $, by \eqref{beta'-beta}, $s-t \geq 2\e ^{2\beta '}$ and computations similar to \eqref{estsRr},
we have
\begin{equation}
\begin{split}
& \int _{2\e ^{2\beta '}} ^t \frac{1}{2(s-\lambda )} \int _{\Omega \setminus B_{\e ^{\beta '}}} \Big( \frac{\e |\nabla\varphi| ^2}{2}-\frac{W(\varphi)}{\e} \Big)_{+} \tilde \rho  \, dxd\lambda \\
& \leq \int _{2\e ^{2\beta '}} ^t \frac{d\lambda }{2(s-\lambda )^{\frac{n+1}{2} } (\sqrt{4\pi })^{n-1}} \int _0 ^1 \Big\{ \int _{B_{\frac{1}{2}} \cap \{ x \, :
 \, e^{-\frac{|x-y|^2 }{ 4(s-\lambda )}} \geq l \} \setminus B_{\e ^{\beta ' } } (y) } \Big( \frac{\e |\nabla\varphi| ^2}{2}-\frac{W(\varphi)}{\e} \Big)_{+} \Big\} \, dl \\
& \leq \Cr{c-8} c(n) \e^{\beta'-\beta} \int _{2\e ^{2\beta '}} ^t \frac{e^{-\frac{1}{16(s-\lambda )}} + (s-\lambda )^{\frac{n-1}{2}}}{(s-\lambda ) ^{\frac{n+1}{2}}} \, d\lambda \\
&\leq \Cr{c-8} c(n) \e ^{\beta'-\beta} (1+ \beta ' \log (1/\e)).
\end{split}
\label{middle2}
\end{equation}
By \eqref{near0}-\eqref{middle2} we obtain the desired estimate. 
\end{proof}
To utilize the formula \eqref{monoton}, we 
next obtain the estimate for $u$. 
\begin{lem}
\label{lem4-12}
There exists $\Cl[c]{c-12}$ depending only on $n,p$ and $q$ such that for any $t_0, t_1$ with $s>t_1>t_0 \geq 0$ we have
\begin{equation}
\int _{t_0} ^{t_1} \int _\Omega \tilde \rho _{(y,s)} |u|^2 \, d\mu^{\e} _t dt
\leq \Cr{c-12} (t_1 -t_0) ^{\hat p} \|u\|_{L^{q}([t_0,t_1];(W^{1,p}(B_{\frac12}(y)))^n)}^2\sup _{t\in [t_0,t_1]} D(t),
\label{c15}
\end{equation}
where (1) $0<\hat p=\frac{2pq -2p-nq}{pq} $ when $p<n$, (2) $\hat p<\frac{q-2}{q}$ 
may be taken arbitrarily close to $\frac{q-2}{q}$ when $p=n$
(and $\Cr{c-12}$ depends on $\hat p$), and (3) $\hat p=\frac{q-2}{q}$ when $p>n$. 
\end{lem}
\begin{proof}
Consider first $p<n$ case. By the H\"{o}lder inequality, for $l:=\frac{p(n-1)}{2(n-p)}$ (which is $\geq 1$
due to \eqref{newpcon}) we have 
\begin{equation}
\begin{split}
\int_{\Omega} \tilde \rho |u|^2 \, d\mu^{\e} _t 
&\leq \Big( \int _{\Omega} |\eta^{\frac12}u|^{2l} \rho \, d\mu^{\e} _t  \Big)^{\frac{1}{l}}
\Big( \int _{B_{\frac12}(y)} \rho \, d\mu^{\e} _t \Big)^{\frac{l-1}{l}} \leq (D(t))^{\frac{l-1}{l}} \Big( \int _{\Omega} |u\eta^{\frac12}|^{2l} \rho \, d\mu^{\e} _t \Big) ^{\frac{1}{l}}\\
&\leq (D(t))^{\frac{l-1}{l}} \left(\frac{1}{(4\pi (s-t))^{\frac{n-1}{2}}} \int _{\Omega} |u\eta^{\frac12}|^{2l} \, d\mu^{\e} _t \right) ^{\frac{1}{l}}.
\end{split}
\label{l}
\end{equation}
By \eqref{l} and \eqref{mzineq} we have
\begin{equation}
\begin{split}
\int_{\Omega} \tilde \rho |u|^2 \, d\mu^{\e} _t 
&\leq (D(t))^{\frac{l-1}{l}} \frac{1}{(4\pi (s-t)) ^{\frac{n-1}{2l}}}
\left( c(n,p) D(t) \Big( \int _{B_{\frac12}(y)} |u|^p+|\nabla u|^p \, dx \Big)^{\frac{n-1}{n-p}} \right)^{\frac{1}{l}} \\
& \leq \frac{\Cl[c]{c-13} D(t) }{(4\pi (s-t)) ^{\frac{n-p}{p}}} \| u \|^2 _{W^{1,p} (B_{\frac12}(y))} ,
\end{split}
\label{l2}
\end{equation}
where $\Cr{c-13}=\Cr{c-13}(n,p)$. Hence by the H\"{o}lder inequality and \eqref{l2} we obtain (with $\|u\|:= \|u\|_{L^{q}([t_0,t_1];(W^{1,p}(B_{\frac12}(y))^n)}$)
\begin{equation*}
\begin{split}
\int _{t_0} ^{t_1} \int_{\Omega} \tilde \rho |u|^2 \, d\mu^{\e} _t dt 
&\leq \Cr{c-13} \|u\|^2\sup _{t\in[t_0,t_1]}D(t) \Big( \int _{t_0}^{t_1} \frac{1}{(s-t)^{\frac{(n-p)q}{p(q-2)}}} \, dt \Big)^{\frac{q-2}{q}} \\
 &\leq \Cr{c-13}\|u\|^2
  \sup _{t\in[t_0,t_1]}D(t)  c(n,p,q)( (t_1-t_0)^{\frac{-(n-p)q}{p(q-2)} +1} )^{\frac{q-2}{q}} \\
&\leq c(n,p,q) \Cr{c-13} (t_1-t_0)^{\frac{2pq-2p-nq }{pq}}  \|u\|^2\sup _{t\in[t_0,t_1]}D(t).
\end{split}
\end{equation*}
We remark that $(s-t_0)^\iota - (s-t_1)^\iota\leq (t_1-t_0)^\iota$ for $\iota \in (0,1)$ and $ \frac{-(n-p)q}{p(q-2)} +1 \in (0,1)$.
By setting $\Cr{c-12}:=c(n,p,q)\Cr{c-13}$, we obtain the desired estimate when $p<n$. For $p=n$, since $W^{1,n}_{loc}
\subset W^{1,p'}_{loc}$ for $p'<n$, we repeat the same argument as above for $p$ close to $n$. Note that $\frac{2pq-2p-nq}{pq}
\uparrow\frac{q-2}{q}$ as $p\uparrow n$. This gives the estimate for $p=n$ case. For $p>n$, 
$\sup_{B_{\frac12}(y)}|\eta^{\frac12}u|\leq c(n,p)\|u\|_{W^{1,p}(B_{\frac12}(y))}$. Thus $\int \tilde\rho |u|^2\, d\mu_t^{\e}
\leq c(n,p)D(t) \|u\|_{W^{1,p}(B_{\frac12}(y))}^2$. This gives the desired estimate for $p>n$. 
\end{proof}
\begin{pro}\label{time}
There exist $\Cl[c]{c-15}>1$ depending only on $n$, $\Cl[c]{c-16}>0$ depending only on $n,\,p,\,q$ and 
$\Cl[eps]{eps-5}>0$ depending only on $n,p,q,\Cr{c-0},\Cr{c-1},D_0, T, W,\beta$ with the following property. For $t_0, t_1$ with $T_1 \geq t_1 >t_0 \geq 0$ and $t_1-t_0\leq 1$, suppose $D(t_1) =\Cr{c-15} D(t_0)$ and $\sup _{t\in [t_0 , t_1]} D(t) \leq \Cr{c-15}D(t_0)$. Then, if $\e <\Cr{eps-5}$, we have
\begin{equation}
(t_1 - t_0) ^{\hat p} \|u\|_{L^q([t_0,t_1];(W^{1,p}(\Omega))^n)}^2\geq \Cr{c-16},
\label{c15c17}
\end{equation}
where $\hat p$ is as in Lemma \ref{lem4-12}.
\end{pro}
\begin{proof} First, for any $s>t_0$, by direct computation and by the definition of $D(t_0)$, we have
\begin{equation}
\int_{\Omega}\tilde\rho_{(y,s)}\, d\mu_{t_0}^{\e}\leq D(t_0). 
\label{dext}
\end{equation}
Let $\Cr{c-15}>1$ be a constant defined by
\begin{equation}
\Cr{c-15}:=\max\{\frac{2\cdot4^{n-1}}{\omega_{n-1}},\frac{(2+\Cr{c-2}) (4\pi)^{\frac{n-1}{2}}}{\omega_{n-1}e^{-\frac14}}\}.
\label{defc15}
\end{equation}
By definition, $\Cr{c-15}$ depends only on $n$.
Suppose that $t_1$ satisfies the assumptions. Recalling the definition of $D(t_1)$, we 
have the following three possibilities, (a) $D(t_1)=\mu_{t_1}^{\e}(\Omega)$, (b) there exists 
$B_r(y)\subset\Omega$ such that $D(t_1)=\frac{1}{\omega _{n-1} r^{n-1}} \mu^{\e} _{t_1} (B_r (y))$ and 
$r\geq \frac14$, and (c) the same as (b) except that $r<\frac14$. 
For (b), we have the following 
\begin{equation*}
\frac{\omega_{n-1}}{4^{n-1}} D(t_1)\leq \omega_{n-1} r^{n-1}D(t_1)=\mu_{t_1}^{\e}(B_r(y))\leq \mu_{t_1}^{\e}(\Omega).
\end{equation*}
Since $\omega_{n-1}/4^{n-1}\leq 1$, either (a) or (b), we have
\begin{equation}
\frac{\omega_{n-1}}{4^{n-1}} D(t_1)\leq \mu_{t_1}^{\e}(\Omega).
\label{rsca}
\end{equation}
Then, by \eqref{e0}, we obtain with \eqref{rsca} (and writing $\|u\|:=\|u\|_{L^q([t_0,t_1];(W^{1,p}(\Omega))^n)}$)
\begin{equation}
\label{er2su}
\Cr{c-15}D(t_0)=D(t_1)\leq \frac{4^{n-1}}{\omega_{n-1}} \mu_{t_1}^{\e}(\Omega)\leq \frac{4^{n-1}}{\omega_{n-1}}\big(
D(t_0)+\Cr{c-n}(t_1-t_0)^{\frac{q-2}{q}} \|u\|^2
\Cr{c-15} D(t_0)\big).
\end{equation}
By \eqref{defc15}, $\frac{4^{n-1}}{\omega_{n-1}}\leq \frac{\Cr{c-15}}{2}$, thus \eqref{er2su} shows
\begin{equation}
\frac{1}{2\Cr{c-n}}\leq  (t_1-t_0)^{\frac{q-2}{q}} \|u\|^2.
\label{er2su2}
\end{equation}
This is the conclusion deduced from (a) and (b). 
Next consider the case (c). 
Let $s=t_1 +r^2$. By \eqref{monoton}, \eqref{100}, \eqref{c15} and \eqref{e0}, we have
\begin{equation}
\begin{split}
\int _{ \Omega } \tilde \rho _{(y,s)} \, d\mu^{\e} _{t_1} 
\leq \int _{\Omega} \tilde \rho _{(y,s)} \, d\mu^{\e} _{t_0}+\Cr{c-11} \e ^{\beta '-\beta}|\log\e|
 +\Cr{c-12}\Cr{c-15}  D(t_0) ( t_1 -t_0 )^{ \hat p} 
 \| u \|^2\\
+\Cr{c-2}(t_1-t_0)(D(t_0)+\Cr{c-n}\Cr{c-15}D(t_0)(t_1-t_0)^{\frac{q-2}{q}}
\|u\|^2).
\end{split}
\label{monotone2}
\end{equation}
We compute using $\eta=1$ on $B_{\frac14}(y)$ and $r\leq \frac14$ that
\begin{equation}
\begin{split}
\int _{\Omega} \tilde \rho _{(y,s)} \, d\mu^{\e} _{t_1}  & \geq \int_{B_{r}(y)}\rho_{(y,s)}\, d\mu_{t_1}^{\e} \geq
 \frac{e^{-\frac14}}{(4\pi )^{\frac{n-1}{2}} r^{n-1} } \mu^{\e} _{t_1} (B_r(y)) =\frac{\Cr{c-15}D(t_0)
 \omega_{n-1}e^{-\frac{1}{4}}}{(4\pi )^{\frac{n-1}{2}} } \\
&\geq (2+\Cr{c-2}) D(t_0),
\end{split}
\label{3D} 
\end{equation}
where $s=t_1+r^2$, the properties of $t_1$ and $\Cr{c-15}$ are used. By \eqref{dext}, \eqref{monotone2} and \eqref{3D} give (using also $t_1-t_0\leq 1$)
\begin{equation}
\begin{split}
D(t_0) &\leq \Cr{c-11} \e ^{\beta '-\beta}|\log\e|
 +\Cr{c-12}\Cr{c-15} D(t_0) ( t_1 -t_0 )^{\hat p} 
 \| u \|^2\\
&+\Cr{c-2} \Cr{c-n}\Cr{c-15}D(t_0)(t_1-t_0)^{2-\frac{2}{q}}
\|u\|^2.
\end{split}
\label{2D}
\end{equation}
Since $D(t_0)\geq 1$ by definition, we may restrict $\e$ depending on $\Cr{c-11}$ (see Lemma \ref{discre1})
so that $\Cr{c-11}\e^{\beta'-\beta}|\log \e|<1/2$, for example. 
Now, examining the dependence of constants, we obtain \eqref{c15c17} from \eqref{er2su2} and \eqref{2D} by choosing 
an appropriate $\Cr{c-16}>0$. Here we also use $\hat p< 2-\frac{2}{q}$ and $t_1-t_0\leq 1$. 
\end{proof}
{\it Proof of Theorem \ref{density}}.
We first choose $0<T_b\leq 1$ so that 
\begin{equation}
T_b^{\hat p}\Cr{c-1}^2\leq \Cr{c-16}
\label{deftb}
\end{equation}
holds. Due to the dependence of $\Cr{c-16}$, $T_b$ depends only on $n,p,q,\Cr{c-1}$. Then set 
\begin{equation}
D_1:=D_0 \Cr{c-15}^{[T/T_b]+1}(\geq 2D_0\mbox{ by \eqref{defc15}}),
\label{defcb}
\end{equation}
so that $D_1$ depends only on $n,p,q,\Cr{c-1},T,D_0$. Finally restrict $\e<
\Cr{eps-5}$ as in Proposition \ref{time}. Now we claim that 
\begin{equation}
D(t)\leq D_0 \Cr{c-15}^{[t/T_b]+1}
\label{defcb2}
\end{equation}
holds for all $t\in [0,T]$, thus proving $D(t)\leq D_1$ for all $t\in [0,T]$ and $T_1=T$. 
Suppose there exists $0<t\leq T$ such that \eqref{defcb2} fails. Then there must 
exist some $0<T_1<T$ such that $D(t)\leq D_0 \Cr{c-15}^{[t/T_b]+1}$ for all $t\in [0,T_1]$
and $D(T_1)=D_0 \Cr{c-15}^{[T_1/T_b]+1}$. Note that $D(t)\leq D_1$ for $t\in [0,T_1]$,
satisfying \eqref{cb2}. If $T_1<T_b$, we apply Proposition \ref{time} with $t_0=0$ and
$t_1=T_1$. We have $D(T_1)=\Cr{c-15}D_0$ and $\sup_{t\in [0,T_1]}D(t)\leq \Cr{c-15}D_0$.
Thus \eqref{c15c17} shows 
\begin{equation*}
T_1^{\hat p}\Cr{c-1}^2\geq \Cr{c-16},
\end{equation*}
but this contradicts $T_1<T_b$ and \eqref{deftb}. Thus, we have $T_1\geq T_b$. If
$T_1\in [T_b,2T_b)$, then $D(T_1)=D_0 \Cr{c-15}^2$. Thus there must exist $t_0\in [T_b,T_1)$
such that $D(t_0)=\Cr{c-15}D_0$ and $T_1-t_0<T_b$ (note that $D(t)\leq D_0 \Cr{c-15}$ for 
all $t\in [0,T_b)$). By Proposition \ref{time} with $t_1=T_1$, we have $(T_1-t_0)^{\hat p} \Cr{c-1}^2\geq\Cr{c-16}$, again contradicting $T_1-t_0<T_b$ and \eqref{deftb}.
Continuing this manner, we conclude that $T_1=T$, which is a contradiction. Thus we 
proved that \eqref{defcb2} holds for all $t\in [0,T]$. Also this concludes the proof of
Theorem \ref{density}.
\hfill{$\Box$}

Since we proved $T=T_1$, i.e., the assumption \eqref{cb2} is true for all $[0,T]$, 
all the estimates in this section hold with $T_1$ replaced by $T$. In particular, we have the
following monotonicity formula which follows from \eqref{monoton}, \eqref{c15} and \eqref{100}.
\begin{thm}
Under the same assumptions of Theorem \ref{density}, if $\e<\Cr{eps-1}$ and for
$s>t_1>t_0$, $t_0, t_1\in [0,T]$, and $y\in \Omega$ we have
\begin{equation}
\begin{split}
 \left.\int_{\Omega}\tilde\rho \, d\mu_t^{\e}\right|_{t=t_0}^{t_1} &+
\int_{t_0}^{t_1}\frac{dt}{2(s-t)}\int_{\Omega}|\xi_{\e}|\tilde\rho\, dx\leq 
\Cr{c-12}\Cr{c-1}^2 (t_1-t_0)^{\hat p} D_1 \\
&+\Cr{c-11}\e^{\beta'-\beta}|\log\e|+
\Cr{c-2}e^{-\frac{1}{128(s-t_0)}} (t_1-t_0)D_1,
\end{split}
\label{longeq}
\end{equation}
where $\tilde\rho=\tilde\rho_{(y,s)}(x,t)$ and $\xi_{\e}$ are defined as in \eqref{defti} and \eqref{defti2},
and $\hat p$ is as in Lemma \ref{lem4-12}.
\label{remono}
\end{thm}
The point of the right-hand side is that it is bounded independent of $\e$, 
and it can be made arbitrarily small when $\e\rightarrow 0$ and $t_0 \rightarrow t_1$.
\section{Existence of limit measures}
\label{secexm}
In this section we construct a sequence of approximate diffused interface solution 
for \eqref{mcf}, given any bounded hypersurface $M_0=\partial \Omega_0$ which is $C^1$,
and any vector field $u$ satisfying \eqref{intu}.
We then prove that
we may extract a subsequence which converges to a family of Radon measures $\{\mu_t\}_{t\geq 0}$.

We first construct a
convergent sequence of domains $\Omega_0^i$ with $C^{\infty}$ boundary $M_0^i$ which converges in
$C^1$ topology. This can be carried out by locally representing $M_0$ by a $C^1$ graph and by some
suitable mollification. Let $d_i$ be the signed distance function to $M_0^i$ 
which is positive inside of $\Omega_0^i$,
and which is smooth in some $r_i$-neighborhood of $M_0^i$. 
Let $h_i\in C^{\infty}({\mathbb R})$ be a monotone increasing function such that 
$h_i(s)=s$ for $0\leq s\leq r_i/3$, $h_i(s)=r_i/2$ for $s>2r_i/3$, $h_i'(s)\leq 1$  for $s>0$ and $h_i(s)=-h_i(-s)$ for
$s<0$. Then define $\tilde d_i(x):=h_i(d_i(x))$ for $x\in \Omega$. We next choose a 
sequence of $\e_i>0$ so that 
\begin{equation}
\lim_{i\rightarrow\infty} \sqrt{\e_i}/r_i=0.
\label{epfast}
\end{equation}
We define the initial data $(\varphi_{\e_i})$ differently depending on $\Omega={\mathbb T}^n$ or ${\mathbb R}^n$
as follows. 

For $\Omega={\mathbb T}^n$, we define
\begin{equation}
(\varphi_{\e_i})_0:=\Psi\big(\frac{\tilde d_i(x)}{\e_i}\big).
\label{vapini}
\end{equation}
Here and in the following, $\Psi$ is the solution for $\Psi''=W'(\Psi)$ (and $\Psi'=\sqrt{2W(\Psi)}$) with $\Psi(0)=0$. 
For $\Omega={\mathbb R}^n$, we will truncate the function to be $-1$ outside of a compact set as follows. 
Due to the definition, note
that for $x\in {\mathbb R}^n$ with ${\rm dist}(x,\Omega_0^i)\geq 2r_i/3$, we have ${\tilde d}_i(x)=-r_i/2$. Choose a sufficiently large $R>0$
such that 
\begin{equation}
\{x : {\rm dist}(x,\Omega_0^i)\leq 2r_i/3\}\subset B_R
\label{defR}
\end{equation}
for all $i$. Then we have ${\tilde d}_i(x)=-r_i/2$ 
on ${\mathbb R}^n\setminus B_R$. Let $g:{\mathbb R}^+\rightarrow [0,1]$ be a smooth decreasing function such that 
$g(r)=1$ for $0\leq r\leq R$, $g(r)=0$ for $R+1\leq r<\infty$ and $|g'|\leq 2$. Define
\begin{equation}
(\varphi_{\e_i})_0 (x):=g(|x|)\Psi\big(\frac{{\tilde d}_i (x)}{\e_i}\big)+ g(|x|)-1.
\label{vapini2}
\end{equation}
Then $(\varphi_{\e_i})_0(x)=\Psi\big(\frac{{\tilde d}_i(x)}{\e_i}\big)$ on $B_R$, and it smoothly changes from $\Psi(-r_i/2\e_i)$ 
to $-1$ as $|x|$ increases from $R$ to $R+1$. We may show from $\Psi'=\sqrt{2W(\Psi)}$ that $0<\Psi(-r_i/2\e_i)+1\leq 
c \exp(-c'r_i/\e_i)$ for some positive constants $c,c'$ depending only on $W$. Thus the difference between $(\varphi_{\e_i})_0$
and $-1$ is exponentially small on $B_{R+1}\setminus B_R$ by \eqref{epfast}, and $(\varphi_{\e_i})_0(x)=-1$ on ${\mathbb R}^n
\setminus B_{R+1}$. 

For both cases, one can check that \eqref{u1} is satisfied for $(\varphi_{\e_i})_0$ with 
some $i$-independent $\Cr{c-0}$, where we may need to take a smaller $\e_i$ 
depending on the growth of $C^3$ norm of the graph functions representing $M_0^i$. 
We fix $\beta$
\begin{equation}
\beta=\frac14,
\label{cb14}
\end{equation}
though any $0<\beta<1/2$ can be chosen. 
Using the
fact that $\Psi$ solves $\Psi'=\sqrt{2W(\Psi)}$ and $|\nabla \tilde d_i|\leq 1$, one can check that
\eqref{ic} is satisfied for all $i$. We may also assume that
\begin{equation}
\begin{split}
&\lim_{i\rightarrow\infty}\int_{\Omega} \Big| \frac{(\varphi_{\e_i})_0+1}{2}-\chi_{\Omega_0}\Big|\, dx=0,\\
&\lim_{i\rightarrow\infty}\big(\frac{\e_i |\nabla(\varphi_{\e_i})_0|^2}{2}+
\frac{W((\varphi_{\e_i})_0)}{\e_i}\big)\, dx=\sigma \|\nabla\chi_{\Omega_0}\|=\sigma {\mathcal H}^{n-1}
\lfloor_{M_0}
\end{split}
\label{coninit}
\end{equation}
where the second identity is in the sense of measure convergence. 
We may also assume, due to the assumption that $M_0$ is $C^1$, that we have some $D_0$
depending on $M_0 $ such that $D(0)$ as in \eqref{defdens} corresponding to $(\varphi_{\e_i})_0$ is
uniformly bounded by $D_0$ independent of $i$. 

We next let $T_i=i$ so that $\lim_{i\rightarrow\infty}T_i=\infty$, 
and let $\{u_{i}\}_{i=1}^{\infty}$ be a sequence of $C^{\infty}$ vector fields with compact support such that $
\|u_{i}-u\|_{L^q([0,T_i];(W^{1,p}(\Omega))^n)}\rightarrow 0$ as $i\rightarrow \infty$, 
which can be constructed by the standard
density argument.
Then for each $i$ we associate $j(i)$ so that \eqref{u2} is satisfied, i.e., 
\begin{equation}
\sup_{\Omega\times [0,T_i]}\{|u_{i}|,\, \e_{j(i)}|\nabla u_{i} |\} \leq \e_{j(i)}^{-\beta}
\label{supucon}
\end{equation}
for all $i$, and at the same time, $\e_{j(i)}<\Cr{eps-1}$ where $\Cr{eps-1}$
is determined by Theorem \ref{density} corresponding to $D_0$, 
$T=T_i$ and $\Cr{c-1}=\|u_{i}\|_{L^q([0,T_i];(W^{1,p}(\Omega))^n)}$.
We relabel $\e_{j(i)}$ as $\e_i$ and $u_{i}$ as $u_{\e_i}$. 

With these choices, for each $i\in {\mathbb N}$, 
we solve \eqref{beq1} and \eqref{beq2} on $\Omega\times[0,T_i]$ with initial data $(\varphi_{\e_i})_0$
and $u$ replaced by $u_{\e_i}$. For $\Omega={\mathbb T}^n$, the standard parabolic PDE
theory shows the existence of classical solution which we denote $\varphi_{\e_i}$. The maximum 
principle shows \eqref{u1sup}. Due to the 
choice of $\e_i$, for each fixed $T>0$, we have all the assumptions of Theorem \ref{density} 
satisfied on $[0,T]$ for all sufficiently large $i$, thus we have \eqref{denconcl}. 
The same can be said about Theorem \ref{remono}. For $\Omega={\mathbb R}^n$ and for each
fixed $i$, we construct the solution by domain approximation. Namely, for each $k\in {\mathbb N}$ 
with $k>3R$ (where $R$ is defined in \eqref{defR}), solve
\begin{equation}
\label{domainapp}
\left\{\begin{array}{ll} 
\partial_t\varphi+u_{\e_i}\cdot\nabla\varphi=\Delta\varphi-\frac{W'(\varphi)}{\e_i^2} & \mbox{on }B_k\times[0,T_i], \\
\varphi=(\varphi_{\e_i})_0 & \mbox{on }B_k\times\{0\}, \\
\varphi=-1 & \mbox{on }\partial B_k\times  [0,T_i].
\end{array}\right.
\end{equation}
By the standard parabolic existence theory, there exists a classical solution which we denote by $\varphi_{\e_i,k}$. 
By the maximum principle, we have $-1\leq \varphi_{\e_i,k}<1$. We claim that 
\begin{equation}
\label{domainapp2}
\varphi_{\e_i,k}(x,t)<\Psi\big(\frac{3R+t\|u_{\e_i}\|_{L^{\infty}}-|x|}{\e_i}\big)=:\psi_{\e_i}(x,t)
\end{equation}
for all $k$ by the maximum principle. To see this, on $\partial B_k\times [0,T_i]$, 
we have $\varphi_{\e_i,k}(x,t)=-1<\psi_{\e_i}(x,t)$ by \eqref{domainapp} and \eqref{domainapp2}. 
On $B_k\times \{0\}$ where $\varphi_{\e_i,k}=(\varphi_{\e_i})_0$, we may check $\psi_{\e_i}>(\varphi_{\e_i})_0$ as follows.
When $|x|\geq R+1$, $\psi_{\e_i}(x,0)>-1=(\varphi_{\e_i})_0(x)$,
and when $R\leq |x|\leq R+1$, $(\varphi_{\e_i})_0(x)\approx -1<\Psi(0)<\psi_{\e_i}(x,0)$. When $|x|< R$,  
\begin{equation*}
(\varphi_{\e_i})_0(x)\leq \Psi\big(\frac{{\tilde d}_i(x)}{\e_i}\big)< \Psi\big(\frac{2R}{\e_i}\big)
\leq \Psi\big(\frac{3R-|x|}{\e_i}\big)= \psi_{\e_i}(x,0)
\end{equation*}
since $|{\tilde d}_i(x)|\leq |d_i(x)|< 2R$ from $M_0^i\subset B_R$. 
$\psi_{\e_i}$ is a super-solution since, for $|x|\neq 0$, 
\begin{equation*}
\partial_t\psi_{\e_i}+u_{\e_i}\cdot\nabla\psi_{\e_i}-\Delta\psi_{\e_i}+\frac{W'(\psi_{\e_i})}{\e_i^2}
=\frac{\Psi'(\psi_{\e_i})}{\e_i}\big(\|u_{\e_i}\|_{L^{\infty}} +\frac{n-1}{|x|}-\frac{x}{|x|}\cdot u_{\e_i}\big)>0.
\end{equation*}
We note that $\varphi_{\e_i,k}$ cannot touch $\psi_{\e_i}$ 
from below at $|x|=0$. Thus we may prove \eqref{domainapp2} by the standard argument of the maximum 
principle. Now let $k\rightarrow\infty$ and we may prove that $\varphi_{\e_i,k}$ converge to a solution $\varphi_{\e_i}$
of \eqref{beq1} on ${\mathbb R}^n\times[0,T_i]$ satisfying $-1\leq \varphi_{\e_i}\leq \psi_{\e_i}$. Hence, we have
\eqref{u1sup}. Due to \eqref{domainapp2},
for each fixed $i$, we have the exponential approach of $\varphi_{\e_i}$ to $-1$ as $|x|\rightarrow \infty$, which is 
\eqref{extu1}. Thus, in the case of $\Omega={\mathbb R}^n$, we have all the assumptions of Theorem \ref{density} 
satisfied and we may obtain the desired conclusion. 

We next prove that there exists a family of Radon measures $\{\mu _t \}_{t\geq 0}$ such that,
after choosing a subsequence, $\mu _t ^{\e_{i_j}} \to \mu _t $ as $j\rightarrow\infty$ for all $t\geq 0$.
\begin{pro}\label{propmonotone}
Corresponding to $T>0$ and $\phi \in C_c^2 (\Omega ; \mathbb{R}^+)$, 
there exists $\Cl[c]{c-17} >0$ depending only on $n,p,q,T,D_0,\Cr{c-1}$ and $\|\phi\|_{C^2(\Omega)}$ 
such that, for all $i$ with $i>T$ and $\mu^{\e_i}_t$ 
constructed as above, the function
\begin{equation}
\mu_t ^{\e_i } (\phi) -\Cr{c-17} \Big( \int _0 ^t \|u_{\e_i} (\cdot , s) \|^2 _{W^{1,p} (\Omega)} \, ds +t\Big)
\label{monotonefunc}
\end{equation}
of $t$ is monotone decreasing on $[0,T]$.
\end{pro}
\begin{proof}
By \eqref{beq1} and integration by parts we have
\begin{equation}
\begin{split}
\frac{d}{dt}\mu _t ^{\e_i} (\phi)
= \int_\Omega &-\e_i \phi \Big( \Delta \varphi_{\e_i} 
-\frac{W'(\varphi_{\e_i} )}{\e_i ^2} \Big) ^2 -\e_i \nabla 
\phi \cdot \nabla \varphi_{\e_i} \Big( \Delta \varphi_{\e_i}
 -\frac{W'(\varphi_{\e_i} )}{\e_i ^2} \Big) \\
& +\e_i \phi \Big( \Delta \varphi_{\e_i} -\frac{W'(\varphi_{\e_i} )}{\e_i ^2} \Big)
u_{\e_i}\cdot \nabla \varphi_{\e_i} + \e_i (\nabla \varphi_{\e_i} \cdot \nabla \phi )
 (u_{\e_i}\cdot \nabla \varphi_{\e_i} ) \, dx. 
\end{split}
\label{derivativemu}
\end{equation}
By the Cauchy-Schwarz inequality and estimating as in the proof of Lemma \ref{erlem3}, we have
\begin{equation}
\begin{split}
\frac{d}{dt} \mu_t^{\e_i}(\phi)
& \leq \int_\Omega \e_i|\nabla \varphi_{\e_i}|^2 \frac{|\nabla\phi|^2}{\phi}
+\e_i \phi|u_{\e_i}|^2|\nabla\varphi_{\e_i}|^2\, dx \\
& \leq 4(\sup \|\nabla^2 \phi\| )D(t)
 + \sup|\phi|c(n,p)D(t) \| u_{\e_i} (\cdot , t) \|^2 _{W^{1,p} (\Omega)}.
\end{split}
\label{monotoneineq}
\end{equation}
Thus with a suitable constant independent of $i$ and Theorem \ref{density}, we have \eqref{monotonefunc}.
\end{proof}
\begin{pro}[See \cite{ilmanen1993,tonegawa2010}]
There exist a family of Radon measures $\{\mu _t\}_{t\geq 0}$ and a subsequence (denoted by the
same index)
such that for all $t\geq 0$,
\[ \lim_{i\rightarrow\infty}\mu _t ^{\e _{i}}=\mu _t \qquad \text{as Radon measures.} \]
\label{seqconv}
\end{pro}
\begin{proof}
Fix $T>0$ and $\phi \in C_c^2 (\Omega ; \mathbb{R}^+)$. 
By the Cauchy-Schwarz inequality and $q>2$,
\begin{equation*}
\int _{t_1} ^{t_2} \| u_{\e_i}(\cdot,s)\|^2_{W^{1,p}(\Omega)} \, ds
\leq (t_2 -t_1 )^{\frac{q-2}{q}} \|u_{\e_i} \|_{L^q ([0,T]; (W^{1,p} (\Omega) ) ^n ) } ^2 
\end{equation*}
for $0\leq t_1<t_2\leq T$.
Hence the last term of \eqref{monotonefunc} is uniformly bounded in H\"{o}lder continuous norm with exponent 
$\frac{q-2}{q}$. Thus by the Ascoli-Arzel\`{a} compactness theorem, there exists a subsequence which 
converges uniformly on $[0,T]$. By the monotone decreasing property due to Proposition $\ref{propmonotone}$,
we can choose a subsequence such that $\mu ^{\e _i} _{t} (\phi)$ converges on a co-countable set 
$B(\phi)\subset [0,T]$. Choose a countable set $\{\phi _k \}_{k=1} ^\infty \subset C_c^2 (\Omega ; \mathbb{R} ^+)$ 
which is dense in $C_c(\Omega ;\mathbb{R}^+)$. By the similar argument we can choose a subsequence such that 
$\mu _t ^{\e _i} (\phi _k )$ converges on a co-countable set $B=\cap _{k=1} ^\infty B(\phi _k)$. For any 
$k\geq 1$ we define $\mu _t (\phi _k)= \lim _{i\to \infty} \mu _t ^{\e _i} (\phi _k)$ for $t\in B$. 
Then we may define $\mu _t (\phi)= \lim _{i\to \infty} \mu _t ^{\e _i} (\phi)$ for any 
$\phi \in C_c(\Omega;\mathbb{R}^+)$ and for any $t\in B$ since $\{\phi _k \}_{k=1} ^\infty$ is dense in 
$C_c(\Omega;\mathbb{R}^+)$ and the measures
are uniformly bounded. Since $[0,T]\setminus B$ is countable, we can choose a subsequence so that 
$\mu _t ^{\e _i} (\phi _k)$ converges on $[0,T]\setminus B$ for any $k$. Thus we have the limit 
$\mu _t (\phi)$ for all $\phi \in C_c(\Omega; \mathbb{R}^+) $ and for all $t\in [0,T]$. Now by letting 
$T\rightarrow \infty$ and by diagonal argument,
we may choose a subsequence so that $\mu_t^{\e_i}(\phi)$ converges for all $t\geq 0$ and $\phi\in C_c(\Omega;
\mathbb{R}^+)$. 
\end{proof}
We also denote, after choosing a further subsequence, 
\begin{definition}
Let $\mu$ be a measure on $\Omega\times[0,\infty)$ such that $d\mu=\lim_{j\rightarrow\infty}d\mu_{t}^{\e_j}dt$
locally as measures.
\end{definition}
Since $\sup_{t\in [0,T]}\mu_t^{\e_j}(\Omega)$ is bounded uniformly in $j$ for all $T$, the dominated 
convergence theorem shows $d\mu=d\mu_t\, dt$. On the other hand, note that $\spt \mu$ may not be the same
as $\cup_{t\geq 0}\spt\mu_t\times\{t\}$. In the following section we also use the following notation. 
\begin{definition}
Define $(\spt\mu)_t\subset\Omega$ as 
$(\spt\mu)_t:=\{x\in \Omega\,:\, (x,t)\in \spt\mu\}$.
\end{definition}
\label{seqex}
We have the following inclusion.
\begin{lem}
For all $t>0$, 
\begin{equation}
{\rm spt}\, \mu_t\subset ({\rm spt}\, \mu)_t.
\label{seqex2}
\end{equation}
\label{seqex3}
\end{lem}
\begin{proof}
Suppose $x\in {\rm spt}\, \mu_{t_0}$ and assume for a contradiction that $(x,t_0)\notin {\rm spt}\, \mu$. Then there exists $r>0$ 
such that $\mu(B_r(x)\times(t_0-r^2,t_0+r^2))=0$. Take $\phi\in C^2_c(B_r(x);{\mathbb R}^+)$ with $\phi=1$ on $B_{r/2}(x)$. 
Since $x\in {\rm spt}\, \mu_{t_0}$, we have $\mu_{t_0}(\phi)>0$. By Proposition \ref{propmonotone} and \ref{seqconv}, 
$\mu_t(\phi)-\Cr{c-17}(\int_0^t \|u(\cdot,s)\|_{W^{1,p}}^2\, ds+t)$ is monotone decreasing. Thus one sees that for all sufficiently
small $h>0$, we have $\mu_{t_0-h}(\phi)\geq \mu_{t_0}(\phi)-o(1)\geq \mu_{t_0}(\phi)/2$ where $o(1)\rightarrow 0$ as $h\rightarrow 0$. 
Since $d\mu=d\mu_t dt$, this contradicts $(x,t_0)\notin {\rm spt}\, \mu$. 
\end{proof}
\section{Rectifiability of limit measures}
Throughout this section, let $\varphi_{\e_i}$, $\mu_t^{\e_i}$, $u_{\e_i}$, $\mu_t$ and
$\mu$ be as in 
Section \ref{seqex} and let ${\tilde\rho}_{(y,s)}$, $e_{\e_i}$ and $\xi_{\e_i}$ be
as in \eqref{defti} and \eqref{defti2}. We fix arbitrary $T>0$ and let $\Cr{c-1}$
be as in \eqref{u3} with this $T$. Note that all the estimates in the previous
two sections hold in $[0,T]$ for all sufficiently large $i$ (such that $T_i>T$). For simplicity we often drop $i$ from these quantities. 
In this section we prove that for a.e$.$ $t\geq 0$, there exists a countably $(n-1)$-rectifiable
set $M_t$ such that $\mu_t=\theta(x,t) {\mathcal H}^{n-1}\lfloor_{M_t}$, where $\theta$ is
a non-negative ${\mathcal H}^{n-1}$ measurable function. The important ingredient for
the proof is the vanishing of the discrepancy measure defined below. As stated in the introduction, the content of
this section is based on \cite{ilmanen1993} with some modifications coming from the transport term.
First we note
\begin{lem}
Let $\varphi_{\e_i}$ and $\mu_t^{\e_i}$ be the sequences constructed 
in Section \ref{seqex}. Then there exist a subsequence
(denoted by the same index) and a Radom measure $|\xi|$ such that
\begin{equation}
\lim_{i\rightarrow\infty}\int_{t_0}^{t_1}\int_{\Omega}|\xi_{\e_i}|\phi\, dxdt=\int_{t_0}^{t_1}
\int_{\Omega}|\xi|\phi\, dxdt
\label{xiconv}
\end{equation}
for all $0\leq t_0<t_1<\infty$ and $\phi\in C_c(\Omega\times [0,\infty))$. 
\end{lem}
Due to the uniform estimate $\sup_{i\in{\mathbb N}}\sup_{t\in [0,T]}
\mu_t^{\e_i}(\Omega)$ for any fixed $T$, the existence of such
subsequence follows from the weak compactness of measures. Since $|\xi|$
measures the difference between the two terms in $\mu_t^{\e_i}$ in the limit,
we may call $|\xi|$ as a discrepancy measure. Unlike $\mu_t^{\e_i}$, which
converges to $\mu_t$ for all $t\geq 0$, note that we do not claim any convergence of 
$|\xi_{\e_i}(\cdot, t)|\, dx$ in general.
Instead, we will prove
\begin{thm}
$|\xi|=0$ on $\Omega\times
[0,\infty)$.
\label{xiva}
\end{thm}
\subsection{Forward density lower bound}
\begin{lem}\label{lem-eta}
There exist $1>\gamma_1,\,\eta_1>0$ depending only on  
$n,\Cr{c-0},\Cr{c-1},p,q,T,W,D_0$ and $1>\eta_2>0$ depending only on $n,\Cr{c-0},W$
with the following property. Given $0\leq t<s<T/2$ with $s-t\leq \eta_1$, 
set $r:=\sqrt{2(s-t)}$ and $t':=s+r^2/2$. 
If $x\in \Omega$ satisfies
\begin{equation}
\int_{\Omega} {\tilde\rho} _{(y,s)}(x,t) \, d\mu _s (y)<\eta_2,
\label{eta}
\end{equation}
then $(B_{\gamma_1 r}(x)\times\{t'\})\cap \spt \mu=\emptyset$. 
\end{lem}
\begin{remark}
Note that $t<s<t'<T$ with $s=\frac{t'+t}{2}$. The Lemma says that, unless
there is at least a certain amount of measure, there would be no measure later in the neighborhood. 
The monotonicity formula \eqref{longeq} plays a crucial role for such conclusion. 
\end{remark}
\begin{proof}
Assume for a contradiction that $(x',t')\in \spt\mu$ for some $x'\in B_{\gamma_1 r}(x)$
under the assumption of \eqref{eta}, where $\gamma_1$ will be chosen
later. Then there is a sequence $\{(x_j,t_j)\}_{j=1}^{\infty}$ and 
$\{\e_{i(j)}\}_{j=1}^{\infty}$ such that $\lim_{j\rightarrow\infty}(x_j,t_j)=(x',t')$ and
$|\varphi_{\e_{i(j)}}(x_j,t_j)|< \alpha$ for all $j$. We relegate its proof to Lemma \ref{lemad1}.
We re-index $i(j)$ as $j$. Then just as in the proof of \eqref{5c10}, there exists $\eta_2=\eta_2(n,\Cr{c-0},W)>0$ such that
\begin{equation}
3\eta_{2}\leq \int_{B_{\e_j}(x_j)} \frac{W(\varphi_{\e_j}(y,t_j))}{\e_{j}}
\tilde{\rho}_{(x_j, t_j+\e_j^2)}(y,t_j)\, dy\leq \int_{\Omega}\tilde{\rho}_{(x_j,t_j+\e_j^2)}
(y,t_j)\, d\mu_{t_j}^{\e_j}(y).
\label{etax1}
\end{equation}
We use Theorem \ref{remono}. By restricting $t'-s\leq \eta_1$ small so that
\begin{equation*}
\Cr{c-12}\Cr{c-1}^2(t_j-s)^{\hat p}D_1+\Cr{c-2} e^{-\frac{1}{128(t_j+\e_j^2-s)}}
(t_j-s)D_1<\eta_2
\end{equation*}
in \eqref{longeq} for all sufficiently large $j$, we obtain
\begin{equation}
\int_{\Omega}\tilde{\rho}_{(x_j,t_j+\e_j^2)}(y,t_j)\, d\mu_{t_j}^{\e_j}(y)
\leq \int_{\Omega}\tilde{\rho}_{(x_j,t_j+\e_j^2)}(y,s)\, d\mu_s^{\e_j}(y)
+\Cr{c-11}\e_j^{\beta'-\beta}|\log\e_j|+\eta_2.
\label{etax1.5}
\end{equation}
Letting $j\rightarrow \infty$, we obtain by \eqref{etax1} and \eqref{etax1.5}
\begin{equation}
2\eta_2\leq \int_{\Omega} \tilde{\rho}_{(x',t')}(y,s)\,d \mu_{s}(y).
\label{etax2}
\end{equation}
We next want to change the center of the kernel from $x'$ to $x$. 
Fix $0<\delta<1/2$ so that $2\delta D_1<\eta_2$. Corresponding to $\delta$,
a direct computation shows that we may choose $\gamma_1>0$ so that
\begin{equation}
\int_{\Omega} \tilde{\rho}_{(x',t')}(y,s)\, d\mu_s(y)\leq \delta D_1+(1+\delta)
\int_{\Omega}\tilde{\rho}_{(x,t')}(y,s)\, d\mu_s(y)
\label{etax3}
\end{equation}
if $|x-x'|\leq \gamma_1 r$. By the choice of $\delta$, \eqref{etax2} and \eqref{etax3} show
\begin{equation}
\eta_2\leq \int_{\Omega}\tilde{\rho}_{(x,t')}(y,s)\, d\mu_s(y).
\label{etax4}
\end{equation}
Finally, since $t'-s=s-t$, we have $\tilde{\rho}_{(x,t')}(y,s)=\tilde{\rho}_{(y,s)}(x,t)$.
This is a contradiction to \eqref{eta}. Thus we proved $(x',t')\notin\spt\mu$. 
\end{proof}
\begin{lem}
Assume $(x',t')\in \spt \mu$.
Then there are sequences $\{(x_j,t_j)\}_{j=1}^{\infty}$ and 
$\{\e_{i(j)}\}_{j=1}^{\infty}$ such that $\lim_{j\rightarrow\infty}(x_j,t_j)=(x',t')$ and
$|\varphi_{\e_{i(j)}}(x_j,t_j)|< \alpha$ for all $j$.
\label{lemad1}
\end{lem}
\begin{proof}
If the claim were not true, there would be $0<r_0<1/2$
such that 
\begin{equation}
\inf_{B_{r_0}(x')\times[t'-r_0^2,t'+r_0^2]}|\varphi_{\e_i}|\geq \alpha
\label{exue1}
\end{equation} for all 
sufficiently large $i$. Let $\phi\in C^{2}_c(B_{r_0}(x'))$ be a function 
such that $|\nabla\phi|\leq 2/r_0$, $0\leq \phi\leq 1$ on $B_{r_0}(x')$ and $\phi=1$ on $B_{r_0/3}(x')$. Then the same computations
following \eqref{phi>alpha} using \eqref{exue1} show
\begin{equation*}
\frac{d}{dt}\int_{\Omega}\frac{1}{2} |\nabla\varphi_{\e_i}|^2\phi^2\, dx\leq
-\frac{\kappa}{2\e_i^2}\int_{\Omega}|\nabla\varphi_{\e_i}|^2\phi^2\, dx+16 r_0^{-2}
\int_{\spt \phi}|\nabla\varphi_{\e_i}|^2\, dx
\label{exue2}
\end{equation*}
for $t\in [t'-r_0^2, t'+r_0^2]$. Writing $M_i:=\sup_{\lambda\in [t'-r_0^2,t'+r_0^2]}\int_{\spt \phi}
\frac12 |\nabla\varphi_{\e_i}(x,\lambda)|^2\, dx$, and proceeding similarly as in \eqref{estcompA2}, we obtain
\begin{equation}
\int_{\Omega}\frac12 |\nabla\varphi_{\e_i}(\cdot,\lambda)|^2\phi^2\,dx
\leq (e^{-\frac{\kappa}{\e_i^2}(\lambda-t'+r_0^2)}+ \frac{32\e_i^2}{r_0^2 \kappa})M_i
\label{exue3}
\end{equation}
for $\lambda\in [t'-r_0^2,t'+r_0^2]$. Since $\e_i M_i$ is uniformly bounded, 
we see from \eqref{exue3} that
\begin{equation}
\lim_{i\to\infty}\sup_{\lambda\in[t'-\frac{r_0^2}{2},t'+r_0^2]}\int_{\Omega}\frac{\e_i}{2}
|\nabla \varphi_{\e_i}(\cdot,\lambda)|^2
\phi^2\, dx=0.
\label{exue4}
\end{equation}
Next, due to \eqref{exue1} and the continuity of $\varphi_{\e_i}$,
 we may assume $1\geq \varphi_{\e_i}\geq \alpha$ on $B_{r_0}(x')\times
[t'-r_0^2,t'+r_0^2]$ without loss of generality. Otherwise, we have $-1\leq \varphi_{\e_i}\leq
-\alpha$ and we may argue similarly. In the following, we use 
\begin{equation}
W'(s)(s-1)\geq (s-1)^2 \kappa
\geq c(W)W(s)
\label{exue5}
\end{equation} 
for some $c(W)>0$ if $s\in [\alpha,1]$. Multiply the equation \eqref{beq1} by $(\varphi_{\e_i} -1)\phi^2$ and
integrate over $Q:=\Omega\times[t'-r_0^2,t'+r_0^2]$. By integration by parts, the Cauchy-Schwarz inequality,
$|\varphi_{\e_i}-1|\leq 1$ and \eqref{exue5}, one obtains
\begin{equation}
c(W)\int_{Q}\phi^2\frac{W(\varphi_{\e_i})}{\e_i^2}\, dxdt
\leq \frac12 \int_{\Omega} \phi^2\, dx+\int_{Q}2|\nabla\phi|^2+\frac12 |u_{\e_i}|^2\phi^2\, dxdt.
\label{exue6}
\end{equation}
Since the right-hand side of \eqref{exue6} is uniformly bounded, we obtain
\begin{equation}
\lim_{i\to\infty} \int_{Q}\phi^2\frac{W(\varphi_{\e_i})}{\e_i}\, dxdt=0.
\label{exue7}
\end{equation}
The estimates \eqref{exue4} and \eqref{exue7} show that
\begin{equation}
\lim_{i\to\infty}\int_{t'-r_0^2/2}^{t'+r_0^2} \mu_t^{\e_i}(\phi^2)\, dt=0.
\label{exue8}
\end{equation}
By Fatou's lemma, Proposition
\ref{seqconv} and \eqref{exue8}, we have
\begin{equation}
\int_{t'-r_0^2/2}^{t'+r_0^2} \mu_t(\phi^2)\, dt=0.
\label{exue9}
\end{equation}
This proves that $(x',t')\notin \spt \mu$. 
\end{proof}
\begin{corollary}\label{measspt} Let $U \subset \Omega$ be open.  For $0< t\leq T$, there
exists $\Cl[c]{c-18}$ depending only on $n,\Cr{c-0},\Cr{c-1},p,q,T,W,D_0$ with the property that
\begin{equation}
\mathcal{H}^{n-1} ((\spt \mu )_t \cap U ) \leq \Cr{c-18} \liminf _{r \rightarrow 0} \mu _{t-r^2} (U)
\label{sptineq}
\end{equation}
and 
\begin{equation}
\mathcal{H}^{n-1} ({\rm spt}\,\mu_t \cap U ) \leq \Cr{c-18} \liminf _{r \rightarrow 0} \mu _{t-r^2} (U).
\label{sptineq2}
\end{equation}
\end{corollary}
\begin{proof}
We only need to prove the result for every compact set $K\subset U$. Set $X_t = (\spt \mu)_t\cap K$. For any $(x,t)\in X_t$,
by the same argument leading to \eqref{etax2}, we have
\begin{equation}
2\eta_2 \leq  \int_{\Omega} \tilde{\rho} _{(x,t)} (y,t-r^2) \, d\mu _{t-r^2}(y)
\label{sptex0}
\end{equation}
for sufficiently small $r>0$. For $0<L<1/(2r)$, using the upper density ratio bound, we have
\begin{equation}
\int_{\Omega\setminus B_{rL}(x)} \tilde{\rho} _{(x,t)} (y,t-r^2) \, d\mu _{t-r^2}(y)\leq D_1\omega_{n-1} (\pi)^{-\frac{n-1}{2}}
\int_{L^2/4}^{\infty}s^{\frac{n-1}{2}}e^{-s}\, ds.
\label{sptex1}
\end{equation}
Thus by choosing sufficiently large $L$ depending only on $n, D_1$ and $\eta_2$, \eqref{sptex0} and \eqref{sptex1} show
\begin{equation}
\eta_2\leq \int_{B_{rL}(x)}\tilde{\rho}_{(x,t)}(y,t-r^2)\, d\mu_{t-r^2}(y).
\label{sptex2}
\end{equation}
Since $\tilde{\rho}_{(x,t)}(\cdot,t-r^2)\leq (4\pi)^{-(n-1)/2}r^{-(n-1)}$, from 
\eqref{sptex2} we obtain
\begin{equation}
(4\pi)^{\frac{n-1}{2}} r^{n-1} \eta_2\leq  \mu_{t-r^2}(B_{rL}(x)).
\label{sptex3}
\end{equation}
Let $\mathcal{B} =\{ \bar{B}_{rL} (x)\subset U \, | \, x\in X_t\}$ which is the covering of $X_t $
by closed balls centered at $x\in X_t$. By the Besicovitch covering theorem, there exist a finite sub-collection $\mathcal{B}_1,\dots , \mathcal{B}_{B(n)}$ such that each $\mathcal{B}_i$ is disjoint set of closed balls and 
\begin{equation}
X_t \subset \cup _{i=1} ^{B(n)} \cup _{\bar{B}_{rL} (x_j) \in \mathcal{B}_i} \bar{B}_{rL} (x_j).
\label{spt2}
\end{equation}
Let ${\mathcal H}^{n-1}_{\delta}$ be defined as in \cite{simon}, so that ${\mathcal H}^{n-1}=\lim_{\delta
\downarrow 0}{\mathcal H}^{n-1}_{\delta}$. 
By the definition, \eqref{sptex3} and \eqref{spt2} we obtain
\begin{equation*}
\begin{split}
\mathcal{H}^{n-1}_{2rL}(X_t ) & \leq \sum _{i=1}^{B(n)} \sum _{\bar{B}_{rL} (x_j) \in \mathcal{B}_i}\omega _{n-1} (rL)^{n-1}
\leq  \sum _{i=1}^{B(n)} \frac{\omega_{n-1} L^{n-1}}{(4\pi)^{\frac{n-1}{2}} \eta_2 } 
\sum _{\bar{B}_{rL} (x_j) \in \mathcal{B}_i} \mu _{t-  r^2} (B_{rL} (x_j)) \\
& \leq \sum _{i=1}^{B(n)} \frac{\omega_{n-1}L^{n-1}}{(4\pi)^{\frac{n-1}{2}}\eta_2} \mu _{t-r^2} (U)
= \frac{\omega_{n-1}L^{n-1}B(n)}{(4\pi)^{\frac{n-1}{2}}\eta_2} \mu _{t-r^2} (U).
\end{split}
\end{equation*}
By setting $\Cr{c-18}$ to be the constant above and letting $r\downarrow 0$, we obtain \eqref{sptineq}.
The second inequality \eqref{sptineq2} follows immediately from \eqref{sptineq} and Lemma \ref{seqex3}.
\end{proof}
\begin{lem}\label{lowerbound}
For $1\leq T<\infty$, let $\eta_2$ be as in Lemma \ref{lem-eta} corresponding to $T$. 
Define
\[Z_T :=\{ (x,t)\in \spt \mu \ : \ 0\leq t\leq T/2,\, \limsup _{s \downarrow t}\int_{\Omega} \tilde{\rho} _{(y,s)}(x,t) \, d\mu _s (y)\leq   \eta_2/2 \}. \]
Then we have $\mu(Z_T)=0$. 
\end{lem}
\begin{proof}
For $0<\tau\leq \eta_1$, where $\eta_1$ is as in Lemma \ref{lem-eta}, define
\begin{equation*}
Z^{\tau} :=\{ (x,t)\in \spt \mu \,: 0\leq t< T/2,\, \int_{\Omega}
\tilde{\rho} _{(y,s)} (x,t) \, d\mu _s (y) < \eta_2, \,\,\forall s\in (t,t+\tau ] \}.
\label{lbd1}
\end{equation*}
Note that $Z_T \subset \cup _{m=1}^{\infty}Z^{\tau_m}$ for some $\{\tau_m\}_{m=1}^{\infty}$ with $\lim_{m\rightarrow\infty}
\tau_m=0$.  Hence we only need to prove $\mu(Z^{\tau})=0$.
In the following we fix $0<\tau\leq \eta_1$.
For $0\leq t\leq T/2$ and $x\in \Omega$, set
\begin{equation}
P_{\tau}(x,t):=\{(x',t')\,:\, \tau> |t-t'|>\gamma_1^{-2}|x-x'|^2\},
\label{lbd2}
\end{equation}
where $\gamma_1$ is as in Lemma \ref{lem-eta}. For $(x,t)\in Z^{\tau}$, 
we use Lemma \ref{lem-eta} to prove 
\begin{equation}
P_{\tau}(x,t)\cap Z^{\tau}=\emptyset.
\label{lbd3}
\end{equation}
Suppose for a contradiction that $(x',t')\in P_{\tau}(x,t)\cap Z^{\tau}$.
Suppose first that $t'>t$. Set $r:=\sqrt{t'-t}$ and $s:=(t'+t)/2$ so that $t'=s+r^2/2$. 
Note that we have $|x-x'|<\gamma_1 r$ by $(x',t')\in P_{\tau}(x,t)$. 
Since $s-t<\tau\leq \eta_1$, we may apply Lemma \ref{lem-eta} to conclude that 
$(x,t)\in Z^{\tau}$ implies $(x',t')\notin \spt \mu$, and in particular, $(x',t')\notin Z^{\tau}$, which is a contradiction.
Next suppose that $t'<t$. We change the role of $(x,t)$ and $(x',t')$ in the previous case, and conclude that
$(x',t')\in Z^{\tau}$ implies $(x,t)\notin Z^{\tau}$, which is again a contradiction. This proves \eqref{lbd3}.
Next, for $(x_0,t_0)\in \Omega\times[\tau/2,T/2]$, define
\begin{equation}
Z^{\tau ,x_0,t_0}=Z^{\tau}\cap (B_{\frac12} (x_0) \times (t_0-\tau/2, t_0+\tau/2)).
\label{lbd4}
\end{equation}
Then $Z^{\tau}$ can be covered by at most a countable union of $Z^{\tau,x_j,t_j}$ with a suitable choice of $\{(x_j,t_j)\}$. 
Thus we only need to prove $\mu (Z^{\tau,x_0,t_0})=0$. With 
arbitrary $0<r\leq\gamma_1\sqrt{\tau}$, 
consider a family of closed balls $\{\bar{B}_r(x)\}_{(x,t)\in Z^{\tau,x_0,t_0}}$ and apply the 
Besicovitch covering theorem. Then we have a finite subfamily
$\bar{B}_r(x_1),\,\cdots,\bar{B}_r(x_N)$ with $(x_j,t_j)\in Z^{\tau,x_0,t_0}$ ($j=1,\cdots,N$) and
\begin{equation}
\{x\in B_{\frac12}(x_0)\,: \, (x,t)\in Z^{\tau,x_0,t_0}\} \subset \cup_{j=1}^N \bar{B}_r(x_j),\,\,\, N r^n \leq 2 B(n) (1/2)^n.
\label{lbd5}
\end{equation}
Note that for each $j=1,\cdots, N$, by \eqref{lbd3} and \eqref{lbd4}, we have
\begin{equation}
Z^{\tau,x_0,t_0}\cap (\bar{B}_r(x_j)\times(t_0-\tau/2,t_0+\tau/2))
\subset (\bar{B}_r(x_j)\times(t_0-\tau/2,t_0+\tau/2))\setminus 
P_{\tau}(x_j,t_j).
\label{lbd6}
\end{equation}
The inclusions \eqref{lbd5} and \eqref{lbd6} shows 
\begin{equation}
Z^{\tau,x_0,t_0}\subset \cup_{j=1}^N(\bar{B}_r(x_j)\times(t_0-\tau/2,t_0+\tau/2))\setminus 
P_{\tau}(x_j,t_j).
\label{lbd7}
\end{equation}
Since $\bar{B}_r(x_j)\times(t_0-\tau/2,t_0+\tau/2)\setminus P_{\tau}(x_j,t_j)\subset \bar{B}_r(x_j)\times
[t_j-\gamma_1^{-2}r^2,t_j+\gamma_1^{-2}r^2]$, from \eqref{lbd7} we obtain
\begin{equation}
Z^{\tau,x_0,t_0}\subset\cup_{j=1}^N  \bar{B}_r(x_j)\times
[t_j-\gamma_1^{-2}r^2,t_j+\gamma_1^{-2}r^2].
\label{lbd8}
\end{equation}
Since $d\mu=d\mu_t dt$, \eqref{lbd8}, \eqref{denconcl} and \eqref{lbd5} show
\begin{equation}
\begin{split}
\mu(Z^{\tau,x_0,t_0})&\leq \sum_{j=1}^N \int_{t_j-\gamma_1^{-2}r^2}^{t_j+\gamma_1^{-2}r^2}\mu_t(\bar{B}_r(x_j))\, dt
\leq 2\omega_{n-1} D_1 r^{n+1}\gamma_1^{-2}N\\ & \leq 2^{2-n} \omega_{n-1}B(n) D_1 r\gamma_1^{-2}.
\end{split}
\label{lbd9}
\end{equation}
Since $0<r\leq \gamma_1\sqrt{\tau}$ is arbitrary, \eqref{lbd9} shows $\mu(Z^{\tau,x_0,t_0})=0$. This concludes the proof.
\end{proof}
\subsection{Vanishing of $\xi$}
First we remark the following
\begin{lem}\label{lem-xi}
For $1\leq T<\infty$ there exists $\Cl[c]{c-19}$ depending only on $n,\Cr{c-0},\Cr{c-1},p,q,T,W,D_0$ with the following 
property. For any $(y,s)\in \Omega\times(0,T)$, we have
\begin{equation}
\int_{\Omega\times(0,s)} \frac{\tilde{\rho} _{(y,s)}(x,t)}{s-t}\, d|\xi|(x,t)\leq \Cr{c-19}.
\label{absolute}
\end{equation}
\end{lem}
\begin{proof}
In \eqref{longeq}, set $t_0=0$ and $t_1=s-\epsilon$ for $0<\epsilon<s$. We simply let $\e_i\rightarrow 0$
and we set the supremum of the right-hand side of \eqref{longeq} (with no $\e$ term) plus 
$D_0$ (coming from the left-hand side) to be $\Cr{c-19}$. 
Then letting $\epsilon\rightarrow 0$, we obtain \eqref{absolute}.
\end{proof}
We are ready to prove Theorem \ref{xiva}.
\begin{proof}
We integrate \eqref{absolute} with respect to $d\mu_s ds$
over $\Omega\times(0,T)$ and use  Fubini's theorem to obtain 
\begin{equation}
\int _{\Omega \times (0,T) }  \big(\int _{\Omega\times(t,T)} \frac{\tilde{\rho} _{(y,s)} (x,t)}{s-t} \, d\mu _s (y) ds\big)d|\xi|(x,t) 
\leq \Cr{c-19}D_1 T.
\label{abso1}
\end{equation}
The finiteness of \eqref{abso1} shows 
\begin{equation}
\int _{\Omega\times(t,T)} \frac{\tilde{\rho} _{(y,s)} (x,t)}{s-t}\, d\mu _s (y) ds <\infty
\label{vanish1}
\end{equation}
for $|\xi|$ a.e$.$ $(x,t)\in \Omega \times (0,T)$.
Next, we claim that, whenever \eqref{vanish1} holds at $(x,t)$, we have
\begin{equation}
\lim _{s\downarrow t} \int _{\Omega} \tilde{\rho} _{(y,s)} (x,t) \, d\mu _s (y) =0.
\label{vanish4}
\end{equation}
We use the monotonicity formula \eqref{longeq} for the proof. 
Set $\lambda  := \log (s-t)$ and
\begin{equation*}
h(s):= \int _{\Omega} \tilde{\rho} _{(y,s)} (x,t) \, d\mu _s (y). 
\end{equation*}
After the change of variable, \eqref{vanish1} is equivalent to
\begin{equation}
\int _{-\infty} ^{\log(T-t)} h(t+e^\lambda)\, d\lambda <\infty.
\label{vanish5}
\end{equation}
We fix $\theta \in (0,1]$ in the following. 
Corresponding to this $\theta$, by \eqref{vanish5}, there exists a decreasing sequence $\{ \lambda _i \}_{i=1}^\infty$ such that
\begin{equation}
\lambda _i \downarrow  -\infty, \qquad \lambda_i -\lambda_{i+1}\leq \theta, \qquad h(t+e^{\lambda _i}) \leq \theta.
\label{vanish6}
\end{equation}
For arbitrary $\lambda \in (-\infty , \lambda_1 )$, choose $i$ such that $\lambda \in [\lambda_i , \lambda _{i-1} )$. Then by \eqref{longeq} (with $\e\rightarrow 0$) applied with $t_0=t+e^{\lambda_i} < t_1=t+e^{\lambda}$, we have
\begin{equation}
\begin{split}
h(t+e^\lambda )&= \int_{\Omega} \tilde{\rho} _{(y,t+e^\lambda )} (x,t) \, d\mu _{t+e^\lambda} (y)
= \int_{\Omega} \tilde{\rho} _{(y,t+2 e^\lambda )} (x,t+e^\lambda ) \, d\mu _{t+e ^\lambda} (y)\\
& \leq \int_{\Omega} \tilde{\rho} _{(y,t+2 e^\lambda )} (x,t+e^{\lambda _i} ) \, d\mu _{t+e ^{\lambda_i} } (y) +o(1)
\end{split}
\label{vanish7}
\end{equation}
where $\lim_{\theta\rightarrow 0}o(1)=0$. On the other hand, by \eqref{vanish6} we have
\begin{equation}
 \int_{\Omega} \tilde{\rho} _{(y,t+e^{\lambda_i})} (x,t) \, d\mu _{t+e^{\lambda_i}} (y)=  h(t+e^{\lambda _i}) \leq \theta.\label{vanish8}
\end{equation}
By direct calculation,
\begin{equation}
\int_{\Omega}\tilde{\rho}_{(y,t+2e^{\lambda})}(x,t+e^{\lambda_i})\, d\mu_{t+e^{\lambda_i}}(y)
\leq o(1)+\int_{B_{M\sqrt{2e^{\lambda}-e^{\lambda_i}}}(y)} \tilde{\rho}_{(y,t+2e^{\lambda})}(x,t+e^{\lambda_i})\, 
d\mu_{t+e^{\lambda_i}}(y)
\label{vanish8.5}
\end{equation}
where $\lim_{M\rightarrow \infty}o(1)=0$ and the convergence does not depend on $\theta$. For any fixed $M$, 
we have
\begin{equation}
\begin{split}
\sup_{x\in B_{M\sqrt{2e^{\lambda}-e^{\lambda_i}}}(y)} \tilde{\rho}_{(y,t+2e^{\lambda})}(x,t+e^{\lambda_i})/
\tilde{\rho}_{(y,t+e^{\lambda_i})}(x,t) & \leq \exp\big(M^2(e^{\lambda-\lambda_i}-1)/2\big)\\
& \leq 1+o(1)
\end{split}
\label{vanish8.7}
\end{equation}
where $\lim_{\theta\rightarrow 0}o(1)=0$. The inequalities \eqref{vanish7}-\eqref{vanish8.7} show that
$h(t+e^{\lambda})$ is made arbitrarily small for all $\lambda<\lambda_1$ and prove \eqref{vanish4}.
Finally define $a(x,t):=\limsup_{s\downarrow t}\int_{\Omega}\tilde{\rho}_{(y,s)}(x,t)\, d\mu_s (y)$ and
note that $\Omega\times(0,T)$ may be split into two disjoint sets 
\begin{equation*}
A\cup B:=\{(x,t)\,:\, a(x,t)=0\}
\cup \{(x,t)\,:\, a(x,t)>0\}.
\label{vanish8.8}
\end{equation*}
The claim \eqref{vanish4} proved $|\xi|(B)=0$. 
On the other hand, by Lemma $\ref{lowerbound}$ we have $\mu(A)=0$.
Since $|\xi|\leq \mu$ by definition, this proves $|\xi|(\Omega\times(0,T))=0$. Since $T>0$ is arbitrary,
we have $|\xi|(\Omega\times (0,\infty))=0$.
\end{proof}

\subsection{Associated varifolds and rectifiability theorem}
We have so far obtained $\mu_t$ as a limit of Radon measures $\{\mu_{t}^{\e_i}\}_{i=1}^{\infty}$. 
To prove the rectifiability of $\mu_t$ for a.e$.$ $t\geq 0$, we now consider a sequence of 
varifolds which are naturally associated with $\{\mu_t^{\e_i}\}_{i=1}^{\infty}$.
\begin{definition}
For $\varphi_{\e_i}(\cdot,t)$, we define $V_t^{\e_i}\in {\bf V}_{n-1}(\Omega)$ as follows.
For $\phi\in C_c(G_{n-1}(\Omega))$, 
\begin{equation}
V_t^{\e_i}(\phi):=\int_{\Omega\cap \{|\nabla \varphi_{\e_i}(x,t)|\neq 0\}} 
\phi\Big(x,I-\frac{\nabla\varphi_{\e_i}(x,t)}{|\nabla\varphi_{\e_i}(x,t)|}\otimes
\frac{\nabla\varphi_{\e_i}(x,t)}{|\nabla\varphi_{\e_i}(\cdot,t)|}\Big)\, d\mu_t^{\e_i}(x).
\label{assvar}
\end{equation}
\end{definition}
\begin{lem}
For $g=(g_1,\cdots,g_n)\in C_c^1(\Omega;{\mathbb R}^n)$, we have
\begin{equation}
\begin{split}
\delta& V_t^{\e_i}(g)= \int_{\Omega}(g\cdot\nabla\varphi_{\e_i})
\big(\e_i\Delta\varphi_{\e_i}-\frac{W'(\varphi_{\e_i})}{\e_i}\big)\, dx \\
&+\int_{\Omega\cap\{|\nabla\varphi_{\e_i}|\neq 0\}} \nabla g\cdot\big(
\frac{\nabla\varphi_{\e_i}}{|\nabla\varphi_{\e_i}|}\otimes
\frac{\nabla\varphi_{\e_i}}{|\nabla\varphi_{\e_i}|}\big)\xi_{\e_i}\,dx
-\int_{\Omega\cap \{|\nabla\varphi_{\e_i}|=0\}}\frac{W(\varphi_{\e_i})}{\e_i}\,I\cdot\nabla g\, dx.
\end{split}
\label{assvar0}
\end{equation}
\end{lem}
\begin{proof} We omit $i$ in the following. 
The first variation of $V_t^{\e}$ with respect to $g$ is 
\begin{equation}
\begin{split}
\delta V_t^{\e}(g)&=\int_{G_{n-1}(\Omega)} \nabla g(x)\cdot S\, dV_t^{\e}(x,S)\\ 
&=\int_{\Omega\cap\{|\nabla\varphi_{\e}|\neq 0\}}\nabla g
\cdot \big(I-\frac{\nabla\varphi_{\e}}{|\nabla\varphi_{\e}|}\otimes
\frac{\nabla\varphi_{\e}}{|\nabla\varphi_{\e}|}\big)\big(\frac{\e}{2}|\nabla
\varphi_{\e}|^2+\frac{W}{\e}\big)\, dx.
\end{split}
\label{assvar1}
\end{equation}
By repeated integration by parts, we have
\begin{equation}
\begin{split}
\int_{\Omega\cap\{|\nabla\varphi_{\e}|\neq 0\}} \nabla g\cdot I\,
\frac{\e}{2}|\nabla
\varphi_{\e}|^2\, dx&=-\e\int_{\Omega}\sum_{j,l=1}^n g_j(\varphi_{\e})_{x_j x_l}
(\varphi_{\e})_{x_l}\, dx \\ &=\e\int_{\Omega} \nabla g\cdot( \nabla\varphi_{\e}\otimes
\nabla\varphi_{\e})+(g\cdot\nabla\varphi_{\e})\Delta\varphi_{\e}\, dx.
\end{split}
\label{assvar2}
\end{equation}
Also by integration by parts, 
\begin{equation}
\int_{\Omega\cap\{|\nabla\varphi_{\e}|\neq 0\}} \nabla g\cdot I\,
\frac{W}{\e}\, dx=-\int_{\Omega\cap\{|\nabla\varphi_{\e}|= 0\}} \nabla g\cdot I\,
\frac{W}{\e}\, dx-\int_{\Omega}(g\cdot\nabla \varphi_{\e})\frac{W'}{\e}\, dx.
\label{assvar3}
\end{equation}
Now substituting \eqref{assvar2} and \eqref{assvar3} into \eqref{assvar1}, we obtain \eqref{assvar0}.
\end{proof}
\begin{pro}
For a.e$.$ $t\geq 0$, $\mu_t$ is rectifiable, and any convergent subsequence
$\{V_t^{\e_{i_j}}\}_{j=1}^{\infty}$ with 
\begin{equation}
\liminf_{j\rightarrow\infty}
\int_{\Omega} \e_{i_j}\big(\Delta\varphi_{\e_{i_j}}(x,t)-\frac{W'(\varphi_{\e_{i_j}}
(x,t))}{\e_{i_j}^2}\big)^2\, dx<\infty
\label{assvar5.1}
\end{equation}
converges to the unique
varifold associated with $\mu_t$.
\label{rec1}
\end{pro}
\begin{proof}
By Theorem \ref{xiva} and by the dominated convergence theorem, we have 
\begin{equation}
\lim_{i\rightarrow \infty} \int_{\Omega}|\xi_{\e_{i}}(\cdot,t)|\, dx=0.
\label{assvar6}
\end{equation}
for full sequence for a.e$.$ $t\geq 0$. 
By Lemma \ref{erlem3}, we see that
\begin{equation*}
\int_0^T\int_{\Omega}
\e_i \big(\Delta \varphi_{\e_i}-\frac{W'}{\e^2_i}\big)^2\, dxdt\leq 2E_0.
\end{equation*} 
Thus, by Fatou's lemma, we have
\begin{equation}
\liminf_{i\rightarrow\infty}\int_{\Omega}
\e_i \big(\Delta \varphi_{\e_i}(x,t)-\frac{W'(\varphi_{\e_i}(x,t))}{\e^2_i}\big)^2\, dx<\infty
\label{assvar5}
\end{equation}
for a.e$.$ $t\geq 0$. 
Suppose $t\geq 0$ satisfies both \eqref{assvar6} and \eqref{assvar5}. 
Since $\|V_t^{\e_i}\|(\Omega)
=\mu_t^{\e_i}(\Omega)$ is uniformly bounded in $i$, by the weak compactness
theorem for measures, there exists a convergent subsequence $\{V_t^{\e_{i_j}}\}_{j=1}^{\infty}$
which  satisfies \eqref{assvar5.1} and which converges to a varifold $V_t$.  
Due to Proposition \ref{seqconv} and \eqref{assvar6}, we have
\begin{equation}
\|V_t\|=\mu_t.
\label{assvar7.5}
\end{equation}
Next, a standard measure theoretic argument
(see for example \cite[3.2(2)]{simon}) shows
\begin{equation}
\mu_t(\{x\in {\rm spt}\, \mu_t\,:\, \limsup_{r\downarrow 0}\frac{\mu_t(B_r(x))}{\omega_{n-1}r^{n-1}}\leq s\})
\leq 2^{n-1}s{\mathcal H}^{n-1}({\rm spt}\, \mu_t)
\label{assvar8}
\end{equation}
for any $s> 0$. By \eqref{sptineq2}, ${\mathcal H}^{n-1}({\rm spt}\, \mu_t)<\infty$, thus \eqref{assvar8} shows
\begin{equation}
\mu_t(\{x\in {\rm spt}\, \mu_t\,:\, \lim_{r\downarrow 0}r^{1-n}\mu_t(B_r(x))=0\})=0.
\label{assvar9}
\end{equation}
The two equalities \eqref{assvar7.5} and \eqref{assvar9} show that
\begin{equation}
V_t=V_t\lfloor_{\{x\in\Omega\, :\, \limsup_{r
\downarrow 0}r^{1-n} \|V_t\|(B_r(x))>0\}\times {\bf G}(n,n-1)}.
\label{assvar10}
\end{equation}
Next we use \eqref{assvar0}. For any fixed $g\in C^1_c(\Omega;{\mathbb R}^n)$, 
\eqref{assvar6} shows
that the limits of the last two terms of
\eqref{assvar0} are both 0. Thus we have
\begin{equation}
\lim_{j\rightarrow\infty} |\delta V_t^{\e_{i_j}}(g)|\leq \liminf_{j\rightarrow\infty} \big(\int_{\Omega}\e_{i_j}
|\nabla \varphi_{\e_{i_j}}|^2\, dx\big)^{1/2}\big(\int_{\Omega}
\e_{i_j} \big(\Delta\varphi_{\e_{i_j}}-\frac{W'}{\e_{i_j}^2}\big)^2\, dx\big)^{1/2}
\label{assvar7}
\end{equation}
for $g$ with $\sup\, |g|\leq 1$.
Since the right-hand side of \eqref{assvar7} does not depend on $g$ and since
$\delta V_t^{\e_{i_j}}(g)\rightarrow \delta V_t(g)$, we have
\begin{equation*}
\sup_{g\in C^1_c(\Omega;{\mathbb R}^n),\ \sup|g|\leq 1} |\delta V_t(g)|<\infty
\end{equation*}
which shows that the total variation $\|\delta V_t\|$ is a Radon measure.
Allard's rectifiability theorem \cite{allard} shows that the right-hand side of \eqref{assvar10} is rectifiable, and hence so is
$V_t$. Once we know that $V_t$ is rectifiable, $V_t$ is determined uniquely by $\|V_t\|=\mu_t$. In particular,
this shows that $\mu_t$ is rectifiable. 
The argument up to this point is valid for any convergent subsequence with \eqref{assvar5.1} and \eqref{assvar6}.
On the other hand, note that $\mu_t$ does not depend on the choice of subsequence 
$\{V_t^{\e_{i_j}}\}_{j=1}^{\infty}$. Since $\mu_t$  determines $V_t$ uniquely, any 
converging subsequence of $\{V_t^{\e_i}\}_{i=1}^{\infty}$ with \eqref{assvar5.1} and \eqref{assvar6} has the same limit $V_t$. 
This completes the proof.
\end{proof}
\section{Integrality of limit measures}
\label{muint}
In this section we prove that the density function of $\mu_t$ is integer-valued $\mu_t$ a.e$.$
modulo division by $\sigma$. 
\subsection{Separating sheets}
We prove in this subsection that, if a set of appropriate quantities are controlled, then we have a lower bound
on a measure in terms of a sum of densities of vertically aligned points. As the name of the present subsection
indicates, what one carries out in essence is to decompose the domain horizontally so that each separated domain
contains approximately one sheet of  diffused interface. 
The original idea comes from \cite{allard} and 
it has been first used in the context of the diffused interface problem in \cite{tonegawa2000}. 
\begin{lem}
Suppose 
\begin{enumerate}
\item $N\in {\mathbb N}$, $Y$ is a finite subset of $\mathbb{R}^n , 0<R<\infty , 1<M<\infty, 0<a<\infty, 0<\e <1, 0<\varrho<\infty,0<E_0 <\infty $ and $-\infty \leq l_1 <l_2 \leq \infty$.
\item $Y$ has no more than $N+1$ elements, and $Y\subset \{(0,\cdots,0,x_n)\,:\, l_1+a<x_n<l_2-a\}$.
Moreover $|x-z|>3a$ for $x,z \in Y $ with $ x\not = z$.
\item $(M+1)\diam Y<R$, and put $\tilde R := M \diam Y$.
\item We have $\varphi\in C^2(\{ y\in \mathbb{R}^n \, : \, \dist (y,Y)<R \})$.
\item For all $x=(0,\dots ,0,x_n)\in Y$,
\begin{equation}
\int _a ^R \frac{d\tau}{\tau ^n}\int _{B_\tau (x) \cap \{ y_n =l_j \}} |e_\e (y_n -x_n)-\e \varphi _{x_n} (y-x) \cdot \nabla \varphi | \, d\mathcal{H} ^{n-1} (y) \leq \varrho
\label{pstack1}
\end{equation}
for $j=1,2$, where $e_{\e}$ is defined as in \eqref{defti2}. 
\item For all $x\in Y$ and $a\leq r \leq R$,
\begin{equation}
\int _{B_r (x)} |\xi _\e|
+(1-(\nu _n)^2) \e |\nabla \varphi|^2 +\e|\nabla\varphi|\big|\Delta\varphi-\frac{W'(\varphi)}{\e^2}
\big|\, dy \leq \varrho r^{n-1}, 
\label{pstack2}
\end{equation}
where $\xi_{\e}$ is defined as in \eqref{defti2} and $\nu=(\nu_1,\cdots,\nu_n)=\frac{\nabla \varphi}{|\nabla\varphi|}$. 
\item For all $x\in Y$,
\begin{equation}
\int_a^R \frac{d\tau}{\tau^n}\int_{B_{\tau}(x)}(\xi_{\e})_+\, dy\leq \varrho.
\label{pstack3}
\end{equation}
\item For all $x\in Y$ and $a\leq  r\leq R$, 
\begin{equation}
\int _{B_r (x)} \e |\nabla \varphi|^2 \, dy\leq E_0 r^{n-1}.
\label{pstack4}
\end{equation}
\end{enumerate}
Then we have the following:
\begin{enumerate}
\item[(A)]
With $S:=\{x\,:\, l_1<x_n<l_2\}$ and
for all $x\in Y$ and $a\leq r<R$,
\begin{equation}
\frac{1}{r^{n-1}}\int_{B_r(x)\cap S} e_{\e}\leq \frac{1}{R^{n-1}}\int_{B_R(x)\cap S} e_{\e}+\varrho(3+R).
\label{log}
\end{equation}
\item[(B)] There exists $l_3 \in (l_1,l_2)$ such that $|x_n -l_3|\geq a$ and
\begin{equation}
\begin{split}
\int _a ^{\tilde R} \frac{d\tau}{\tau ^n} \int _{B_\tau (x)\cap \{ y_n =l_3 \} } 
| e_\e (y_n -x_n )-\e \varphi _{x_n} (y-x) \cdot \nabla \varphi | \, d\mathcal{H}^{n-1} (y) \\
\leq 3(N+1)NM (\varrho + E_0 ^{\frac{1}{2}} \varrho ^{\frac{1}{2}}) 
\end{split}
\label{stack1}
\end{equation}
for any $x=(0,\cdot,0,x_n)\in Y$.
\item[(C)] Put 
\[ Y_1 :=Y \cap \{ x \, : \, l_1 < x_n < l_3 \}, \qquad Y_2 :=Y \cap \{ x \, : \, l_3 < x_n < l_2 \}, \]
\[ S_0 :=\{ x \, : \, l_1 < x_n < l_2 \ \text{and} \ \dist (Y,x) <R \}, \]
\[ S_1 :=\{ x \, : \, l_1 < x_n < l_3 \ \text{and} \ \dist (Y_1,x) <\tilde R \}, \]
\[ S_2 :=\{ x \, : \, l_3 < x_n < l_2 \ \text{and} \ \dist (Y_2,x) < \tilde R \}. \]
Then $Y_1$ and $Y_2$ are non-empty,
\begin{equation}
\diam Y_j\leq \frac{N-1}{N}\diam Y\,\,\,\mbox{for $j=1,2$}
\label{stack1.5}
\end{equation}
and
\begin{equation}
\frac{1}{\tilde R ^{n-1}} \Big( \int_{S_1}e_\e  + \int_{S_2}e_\e \Big) 
\leq  \Big( 1+\frac{1}{M} \Big)^{n-1} \Big\{ \frac{1}{R^{n-1}} \int _{S_0} e_\e +\varrho(3+R) \Big\}.
\label{stack2}
\end{equation}
\end{enumerate}
\label{stack}
\end{lem}
\begin{proof}
For any $x\in Y$, after a parallel translation, assume without loss
of generality that $x=0$ for the proof of $\text{(A)}$. Let $\zeta _1 (y)$ be a smooth approximation of the characteristic function $\chi_{B_r (0)} $, where $a\leq r<R$. Let $\zeta _2 (y)$ be a smooth approximation to the characteristic function of $S$ which depends only on $y_n$. Let us denote
\begin{equation}
h_{\e}:=\Delta\varphi -\frac{W'(\varphi)}{\e^2}. 
\label{stack3}
\end{equation}
Multiply \eqref{stack3} by $(y\cdot \nabla \varphi ) \zeta _1 \zeta _2$. After integration by parts twice 
(as in the computation for \eqref{assvar0}) and letting $\zeta _1 \to \chi _{B_r (0)}$, we obtain
\begin{equation}
\begin{split}
\frac{d}{dr} \Big\{ \frac{1}{r^{n-1}}\int _{B_r}  e_\e \zeta _2 \Big\} +\frac{1}{r^n} \int_{B_r} (\xi _\e + \e 
h_{\e}(y \cdot \nabla \varphi ))\zeta _2
-\frac{ \e }{r^{n+1}}\int _{\partial B_r} (y \cdot \nabla \varphi )^2 \zeta_2 \\
-\frac{1}{r^n}\int _{B_r} \{ e_\e y_n -\e \varphi _{x_n} (y \cdot \nabla \varphi ) \}\zeta' _2 
 =0.
\end{split}
\label{s1}
\end{equation}
We estimate the integral over $[r,R]$ ($r\geq a$) of the second term in \eqref{s1} first. We let $\zeta_2\rightarrow \chi_{S}$ and compute
\begin{equation}
\begin{split}
\int _r ^R \frac{d\tau}{\tau ^n} \int _{B_\tau \cap S} (\xi _\e +\e h_{\e}(y\cdot \nabla \varphi )) 
& \leq  \int _r ^R \frac{d\tau}{\tau ^n}\Big( \int _{B_\tau }(\xi _\e )_{+} \Big)
+ \int _r ^R \frac{d\tau}{\tau ^{n-1}} \Big( \int _{B_\tau} \e |h_{\e}|  |\nabla \varphi | \Big)  \\
&\leq (1+R)\varrho
\end{split}
\label{s2}
\end{equation}
where \eqref{pstack2} and \eqref{pstack3} are used.
From \eqref{s1}, \eqref{s2} and \eqref{pstack1}, we obtain \eqref{log}, proving $\text{(A)}$. 
Next, choose $\tilde y,\tilde z \in Y$ such that $\tilde z_n -\tilde y_n \geq \frac{\diam Y}{N} $ and $Y\cap \{ x \, : \, \tilde y_n < x_n < \tilde z_n \}=\emptyset$. Let $\tilde l_1 = \tilde y_n +\frac{\tilde z_n -\tilde y_n}{3}$ and $\tilde l_2 = \tilde z_n -\frac{\tilde z_n -\tilde y_n}{3}$. To choose an appropriate $l_3\in (\tilde l_1, \tilde l_2)$ which satisfies \eqref{stack1}, we first observe, for $x\in Y$ and $y\in B_r (x)$,
\begin{equation}
\begin{split}
I:=&|e_\e (y_n -x_n ) -\e \varphi _{x_n} (y-x) \cdot \nabla \varphi | \\
=&|(-\xi _\e ) (y_n-x_n) +\e |\nabla \varphi |^2 ((y_n-x_n)-\nu _n(y-x)\cdot \nu ) | \\
\leq & |\xi _\e |r + \e |\nabla \varphi |^2 r\Big( 1-(\nu_n)^2 + \sqrt{1-(\nu_n)^2} \Big).
\end{split}
\label{llog1}
\end{equation}
Thus by Fubini's theorem, \eqref{llog1}, \eqref{pstack2} and \eqref{pstack4} we obtain
\begin{equation}
 \int _{\tilde l_1} ^{\tilde l_2} \, dl \int _a^{\tilde R} \frac{d\tau}{\tau ^n}
\int _{B_\tau (x)\cap \{ y_n =l \}} I \, d\mathcal{H}^{n-1} =
\int _{a}^{\tilde R} \frac{d\tau}{\tau ^n} \int _{B_\tau (x) \cap \{ \tilde l_1 <y_n< \tilde l _2 \}} I \, dy \leq \tilde R (\varrho + E_0^{\frac{1}{2}} \varrho^{\frac{1}{2}}). 
\label{llog2}
\end{equation}
The inequality \eqref{llog2} is satisfied for each $x\in Y$, hence 
we guarantee that there exists $l_3 \in (\tilde l_1 , \tilde l_2)$ such that
\[ \int _a ^{\tilde R} \frac{d\tau}{\tau ^n} \int _{B_\tau (x)\cap \{ y_n =l_3 \}} I \, d\mathcal{H}^{n-1}(y) \leq \frac{(N+1)\tilde R (\varrho + E_0^{\frac{1}{2}} \varrho^{\frac{1}{2}})}{\tilde l_2 - \tilde l_1} \]
for each $x\in Y$. Since $\tilde l_2 - \tilde l _1 \geq \frac{\diam Y}{3N}$, we have $\frac{\tilde R}{\tilde l_2 - \tilde l_1} \leq 3MN$, and we obtain $\text{(B)}$.
We have $S_1 \cup S_2\subset B_{(\tilde R +\diam Y)} (x) \cap S$ for $x\in Y$ and $S_1\cap S_2=\emptyset$. 
Thus, using also (3) and \eqref{log} with $r=\tilde R+\diam Y<R$, we have
\begin{equation}
\begin{split}
\frac{1}{\tilde R ^{n-1}} \Big( \int _{S_1}  e_\e + \int _{S_2} e_\e \Big) &\leq \frac{1}{\tilde R ^{n-1}} \int _{ B_{(\tilde R +\diam Y)} (x) \cap S}  e_\e \\ 
& \leq  \Big( 1+ \frac{1}{M} \Big)^{n-1} \Big\{ \frac{1}{R^{n-1}}\int _{B_R (x) \cap S} e_\e +\varrho(3+R) \Big\}.
\end{split}
\label{llog3}
\end{equation}
Since $B_R (x) \cap S \subset S_0$, we obtain \eqref{stack2}. One can check that
$\tilde{z}_n-\tilde{y}_n\geq 
\frac{\diam Y}{N}$ implies \eqref{stack1.5}. This proves $\text{(C)}$. 
\end{proof}
\begin{pro}\label{a}
Corresponding to $0<R<\infty , \ 0<E_0<\infty, \ 0<s<1$ and $N\in {\mathbb N}$, there exists $0<\varrho<1$ 
with the following property:
Assume $Y\subset \mathbb{R}^n$ has no more than $N+1$ elements and $Y\subset\{(0,\cdots,0,x_n)\,:\, x_n\in {\mathbb R}\}$. 
For some $0<a<R$ and for all $y,\, z\in Y$ with $y\neq z$, we have $|y-z|>3a$ and $\diam Y \leq \varrho R$.
In addition we assume (4), (6), (7), (8) of Lemma \ref{stack}.
Then we have
\begin{equation}
\sum _{x\in Y} \frac{1}{a^{n-1}} \int _{B_a (x)} e_\e \leq s + \frac{1+s}{R^{n-1}} \int _{ \{ x \, : \, \dist (Y,x)<R \} } e_\e.
\label{nstack1}
\end{equation}
\end{pro}
\begin{proof} Denote the number of elements in $Y$ by $\# Y$.
If $\# Y=1$, 
the proof leading to the conclusion $\text{(A)}$ of Lemma \ref{stack} (with $l_1=-\infty$ and $l_2=+\infty$)
gives \eqref{nstack1} if $\varrho(1+R)<s$. Note that $M$ is irrelevant in this case since $\diam Y=0$.
If $1<\# Y\leq N+1$, 
we use Lemma \ref{stack} inductively.  First, we choose $M>1$ depending only on $s,\, n,\,N$ so that 
\begin{equation}
\Big(1+\frac{1}{M}\Big)^{(n-1)N}<1+s\,\,\,\mbox{and }\,\, \frac{N-1}{N}<\frac{M}{M+1}.
\label{nstack2}
\end{equation}
Suppose $(M+1)\diam Y<R$. Then all the assumptions of Lemma \ref{stack} are satisfied, and
we obtain $Y_1$ and $Y_2$ with the estimates. We apply Lemma \ref{stack} again to both $Y_1$
and $Y_2$ with $R$ there replaced by ${\tilde R}=M\diam Y$. Due to \eqref{stack1.5}
and \eqref{nstack2}, we have the
assumption (3) satisfied: 
\begin{equation*}
(M+1)\diam Y_j<M\diam Y
\end{equation*}
for $j=1,2$. We have \eqref{pstack1} with the right-hand side given by the
right-hand side of \eqref{stack1}. For each $j=1,2$, if $\# Y_j=1$, 
then we obtain \eqref{log} with $r=a$. Otherwise, we separate $Y_j$ into two non-empty 
sets. Each time, all the assumptions of Lemma \ref{stack} are satisfied. Thus, after
$(\# Y -1)$-times, we separate ${\mathbb R}^n$ into $\# Y$ disjoint horizontal stacks, each
having one element of $Y$. With \eqref{nstack2}, \eqref{stack2} and \eqref{log}, we may 
choose a sufficiently small $\varrho$ depending only on $s,\, n,\, N, \, R,\, E_0$ so that 
\eqref{nstack1} holds.
\end{proof}
\subsection{The $\e$-scale estimate}
Next proposition is almost identical to \cite{tonegawa2000} and \cite{tonegawa2003}.
It shows that the energy behaves more or less like a 1-D simple ODE solution if 
certain quantities are controlled.
\begin{pro}\label{L}
Given $0<s, \, b,\, \beta<1$, and $1<c<\infty$, there exist $0<\varrho,\, \Cl[eps]{eps-55}<1$ and 
$1<L<\infty$ (which also depend on $n$ and $W$) with the following property:

Assume $0<\e< \Cr{eps-55}$, $\varphi\in C^2(B_{4\e L})$
and
\begin{equation}
 \sup_{B_{4\e L}} \e |\nabla \varphi|\leq c,\,\,\sup_{x,y\in B_{4\e L}}
\e^{\frac32}\frac{ |\nabla\varphi(x)-\nabla\varphi(y)|}{|x-y|^{\frac12}}\leq c, \,\, |\varphi (0)|<1-b,
\label{clo1}
\end{equation}
\begin{equation}
\int _{B_{4\e L} } ( |\xi _\e|
+(1-(\nu _n)^2) \e |\nabla \varphi|^2 )\,dx \leq \varrho (4\e L)^{n-1}
\label{assumptionL}
\end{equation}
and 
\begin{equation}
\sup_{B_{4\e L}} (\xi _\e)_{+} \leq \e^{-\beta},
\label{clo2}
\end{equation}
where $\nu$ and $\xi_{\e}$ are as in \eqref{pstack2} and \eqref{defti2}.
Then for $J:= B_{3\e L}\cap \{x=(0,\cdots,0,x_n)\}$,
\begin{equation}
\inf_{x\in J}\partial_{x_n}\varphi (x)>0, \ \ (\mbox{or } \sup_{x\in J} \partial_{x_n}\varphi(x)<0),
\, \mbox{ and } [-1+b,1-b]\subset \varphi(J).
\label{clo3}
\end{equation}
We also have
\begin{equation}
\Big|\sigma- \frac{1}{\omega _{n-1} (L\e)^{n-1}} \int _{B_{\e L} } e_\e\Big| \leq s.
\label{estL}
\end{equation}
\end{pro}
\begin{proof}
Rescale the domain by $x\mapsto \frac{x}{\e}$. 
The rescaled function defined on $B_{4L}$ is denoted by $\tilde \varphi$. 
Let $\Psi:\mathbb{R} \to (-1,1)$ be the unique solution of the ODE
\begin{eqnarray}
\left\{ \begin{array}{ll}
\Psi'(t)=\sqrt{2W(\Psi(t))} & \text{for} \ t\in \mathbb{R}, \\
\Psi(0)=\tilde \varphi (0). & \\
\end{array} \right.
\label{sclo4}
\end{eqnarray} 
We have
\begin{equation}
 \int _{\mathbb{R}} \frac{1}{2} |\Psi '(t)|^2 \, dt =\int_{\mathbb{R}}W(\Psi(t))\, dt
 =\int _\mathbb{R} \sqrt{\frac{W(\Psi(t))}{2} } \Psi'(t) \, dt = \int _{-1} ^1 \sqrt{\frac{W(s)}{2} } \, ds
  =\frac{\sigma}{2} . 
 \label{sclo1}
 \end{equation}
Define $\hat \Psi (x)=\hat \Psi(x_1,x_2,\dots , x_{n}):= \Psi(x_n)$ for $x\in \mathbb{R}^n$. 
Using \eqref{sclo1}, it is not difficult to check that $\lim_{L\rightarrow\infty}\frac{1}{\omega_{n-1}L^{n-1}}
\int_{B_L}\left( \frac{|\nabla \hat \Psi|^2}{2}+W(\hat \Psi) \right)=\sigma$. Thus
depending only on $n, s, b, W$, we may choose a sufficiently large $L>0$ such that
\begin{equation}
\Big|\sigma- \frac{1}{\omega_{n-1} L^{n-1}} \int _{B_L } \left( \frac{|\nabla \hat \Psi|^2}{2}+W(\hat \Psi) \right) \Big|\leq \frac{s}{2}
\label{ineqL}
\end{equation}  
whenever $|\hat \Psi(0)|=|\tilde\varphi(0)|\leq 1-b$. After fixing such $L$, we next observe that, for a constant $\tilde c =\tilde c(W)$, 
\begin{equation}
\frac{|\nabla \tilde \varphi |^2}{2}-\tilde c (1\pm \tilde\varphi)^2\leq \frac{|\nabla \tilde \varphi |^2}{2}-W(\tilde \varphi) = \e (\xi_{\e})_+ \leq  \e ^{1-\beta} \qquad \text{on} \ B_{4L}
\label{estdiscrepancy}
\end{equation}
by \eqref{clo2}. 
Some simple ODE argument combined with \eqref{estdiscrepancy} shows that there exist $0<\tilde b<b$ and 
$0<\Cr{eps-55}<1$ depending only on $b, \beta, L, W$ such that, whenever $|\tilde\varphi (0)| \leq 1-b$ 
and $\e<\Cr{eps-55}$, we have $|\tilde \varphi |\leq 1-\tilde b$ on $B_{4L}$. 

Next, we define $z:B_{4L}\to \mathbb{R}$ by $z(x) = \Psi^{-1} (\tilde \varphi (x))$, where $\Psi^{-1}$ is the inverse function of $\Psi$. By $\Psi' >0$ and $|\tilde \varphi|\leq 1-\tilde b$, $\Psi^{-1}$ and $z$ are well-defined and 
\begin{equation}
\Psi'(z(x))\geq \min _{|\tilde \varphi|\leq 1-\tilde b} \sqrt{2W(\tilde \varphi)}
\label{sclo3}
\end{equation}
for $x\in B_{4L}$. By \eqref{clo1}, we have
$\|\tilde\varphi\|_{C^{1,\frac12}(B_{4L})}\leq 2c$. Since $\|\Psi^{-1}\|_{C^2(\{t\,:\, |t|\leq 1-\tilde b\})}$ is bounded
depending only on $b, \beta, L, W$ due to \eqref{sclo3}, we have 
\begin{equation}
\|z\|_{C^{1,\frac12}(B_{4L})}\leq C(b,\beta,L,W,c).
\label{sclo7}
\end{equation}
We next note that $\tilde\varphi=\Psi \circ z$ and \eqref{sclo4} give
\begin{equation}
\frac{|\nabla \tilde \varphi |^2 }{2}-W(\tilde \varphi ) =\frac{1}{2} (\Psi' (z))^2 (|\nabla z|^2 -1), \ \ |\nabla \tilde \varphi |^2 (1-(\nu _n )^2 ) = (\Psi ' (z))^2 (|\nabla z|^2 -(\partial _{x_n} z)^2 ) .
\label{sclo2}
\end{equation}
After rescaling \eqref{assumptionL} and using \eqref{sclo3} and \eqref{sclo2}, we obtain
\begin{equation}
\int_{B_{4L}}(||\nabla z|^2-1|+|\nabla z|^2-(\partial_{x_n} z)^2)\, dx
\leq \max_{|t|\leq 1-\tilde b}W(t)^{-1} \varrho (4L)^{n-1}.
\label{sclo5}
\end{equation}
For a non-negative function $f\in C^{\frac12}(B_{4L})$, 
suppose $\max_{\bar B_{3L}} f=f(\hat x)>0$ for $\hat x\in \bar B_{3L}$. Then it is 
easy to check that $f(x)\geq f(\hat x)/2$ as long as 
$|x-\hat x|\leq (f(\hat x))^2/(2\|f\|_{C^{\frac12}(B_{4L})})^2=:r$.
Then we have
\begin{equation*}
f(\hat x)\leq \frac{1}{\omega_n r^n}\int_{B_r(\hat x)}2f\, dx
\leq \frac{2\cdot 4^n \|f\|_{C^{\frac12}(B_{4L})}^{2n}}{\omega_n
(f(\hat x))^{2n}} \int_{B_{4L}}f\, dx
\end{equation*}
and thus we obtain
\begin{equation}
(\max_{\bar B_{3L}}f)^{2n+1}\leq 2\cdot 4^n \|f\|^{2n}_{C^{\frac12}
(B_{4L})} \int_{B_{4L}}f\, dx.
\label{sclo6}
\end{equation}  
By \eqref{sclo7}, \eqref{sclo5} and \eqref{sclo6}, we have
\begin{equation}
\max_{\bar B_{3L}}(||\nabla z|^2-1|+|\nabla z|^2-(\partial_{x_n}z)^2)\leq C(b,\beta,L,W,c)\varrho^{\frac{1}{2n+1}}.
\label{sclo8}
\end{equation}
Since $\Psi(0)=\tilde\varphi(0)=\Psi(z(0))$, we have $z(0)=0$. Note that \eqref{sclo8} for sufficiently small $\varrho$ shows
that $\nabla z\approx (0,\cdots,0,\pm 1)$ uniformly on $B_{3L}$. This shows that $z\approx x_n$ or $-x_n$ in 
$C^1(B_{3L})$ when $\varrho$ is small, and in particular, we have \eqref{clo3}. For the former case, we have $\tilde\varphi(x)=\Psi(z(x))\approx \Psi(x_n)
=\hat \Psi(x)$, and \eqref{ineqL} gives \eqref{estL} for sufficiently small $\varrho$ with the right dependence. In the case
of $-x_n$, we simply note that changing $\hat \Psi$ to $\Psi(-x_n)$ does not affect the proof. 
\end{proof}

\subsection{Estimate on $\{ |\varphi_{\e}| \geq 1-b \}$}
We need to show some uniform smallness of energy on $\{|\varphi_{\e}|\geq 1-b\}$ for the
final step of this section. 
\begin{lem}\label{alpha}
Suppose $\varphi_{\e}$ and $u_{\e}$
are the solutions for \eqref{beq1} constructed in 
Section \ref{secexm}. Given $0<\delta<T$,
there exist $\Cl[c]{c-20}$ and $\Cl[eps]{eps-6}$ depending only on $n, \Cr{c-0}, W$ with the following property. Suppose for $(x_0,t_0)\in \Omega \times (\delta,T)$ and $0<\lambda\leq 2/3$, 
\begin{equation}
\varphi_{\e}  (x_0,t_0)<1-\e ^\lambda \ \ \ (\mbox{or }\ \varphi_{\e} (x_0,t_0)>-1+\e ^\lambda),
\label{rrest2}
\end{equation}
where $\lambda$ additionally satisfies 
\begin{equation}
1\leq \tilde r := \Cr{c-20} \lambda |\log \e|\leq  \e ^{-1}\min\{\sqrt{\delta/2},1/2\}.
\label{rrest1}
\end{equation}
Then
\begin{equation*}
\inf_{B_{\e \tilde r}(x_0)\times (t_0 -\e ^2 \tilde r ^2,t_0)}\varphi_{\e}  <\alpha \qquad \left( \text{resp.} \ \sup_{B_{\e \tilde r}(x_0)\times (t_0 -\e ^2 \tilde r ^2,t_0)}\varphi_{\e}  >-\alpha \right)
\end{equation*}
if $\e \in (0,\Cr{eps-6})$.
\end{lem}
\begin{proof}
First note that $B_{\e \tilde r}(x_0)\times (t_0-\e^2 {\tilde r}^2,t_0)
\subset \Omega\times (0,T)$ due to \eqref{rrest1}.
Rescale the domain by $x \mapsto \frac{x-x_0}{\e}$ and $t\mapsto \frac{t-t_0}{\e ^2}$,
so that we are concerned with the domain $B_{\tilde r}\times (-{\tilde r}^2,0)$. Let $\tilde \varphi _{\e} (x,t):=\varphi_{\e}  (\e x+x_0,\e ^2 t+t_0)$ and $\tilde u_{\e}  (x,t):=u_{\e}  (\e x+x_0,\e ^2 t+t_0)$. 
As a comparison function, we need a function $\psi$ with the following property
\begin{equation}
\left\{
\begin{split}
& \partial _t \psi = \Delta \psi -\frac{\kappa}{2}\psi \qquad \text{on} \ {\mathbb R}^n \times 
(-\infty ,0), \\
&\psi (x,t)\geq e^{\frac{|x|+|t|}{\Cr{c-20}}} \qquad \text{on} \ \mathbb{R}^n \times (-\infty ,0)\setminus B_1 ^{n+1}(0,0), \\
&\psi(0,0)=1,
\end{split}
\right.
\label{psi}
\end{equation}
for some $\Cr{c-20}>0$. To find such a function, solve $\Delta \tilde\psi=\kappa\tilde\psi / 4$ with
$\tilde\psi (0)=1$ on ${\mathbb R}^n$ among radially symmetric functions. One can show 
that $\tilde\psi$ grows exponentially as $|x|\rightarrow\infty$ and $\tilde\psi$ achieves its minimum
at the origin, thus $\tilde\psi\geq 1$ on ${\mathbb R}^n$ in particular. Then set $\psi(x,t):= 
e^{-\kappa t/4} \tilde\psi (x)$. With a suitably large $\Cr{c-20}$ depending only on $n$ and $\kappa$,
this $\psi$ satisfies \eqref{psi}. Next 
set $\tilde r := \Cr{c-20} \lambda |\log \e|$. We choose such $\tilde r$ so that
\begin{equation}
1-\e ^\lambda e^{\frac{\tilde r}{\Cr{c-20}}}=0.
\label{rrest3}
\end{equation}
Under the assumption of \eqref{rrest2} which is equivalent to 
\begin{equation}
{\tilde \varphi}_{\e}(0,0)<1-\e^{\lambda}, 
\label{rrest6}
\end{equation}
for a
contradiction, assume 
\begin{equation}
\inf _{B_{\tilde r}\times (-\tilde r^2,0)}\tilde \varphi_{\e}  \geq \alpha.
\label{rrest4}
\end{equation} Define $\phi_{\e} :=1-\e ^\lambda \psi$. By \eqref{psi} we have $\partial _t \phi_{\e} = \Delta \phi_{\e} +\frac{\kappa}{2}(1-\phi_{\e}) $ on $\mathbb{R}^n \times (-\infty ,0)$. Furthermore, on the parabolic boundary of $B_{\tilde r}\times
(-{\tilde r}^2,0)$, $\psi\geq e^{\frac{\tilde r}{\Cr{c-20}}}$ by $\tilde r\geq 1$ and \eqref{psi}, hence 
\begin{equation}
\phi_{\e} \leq  1-\e ^\lambda e^{\frac{\tilde r}{\Cr{c-20}}}=0<\alpha \leq \tilde{\varphi}_{\e}
\label{rrest5}
\end{equation}
where \eqref{rrest3} and \eqref{rrest4} are used. On the other hand $\phi_{\varepsilon} (0,0)=
1-\e^{\lambda}\psi(0,0)=1-\e ^\lambda >\tilde \varphi_{\e}  (0,0)$ by \eqref{psi} and \eqref{rrest6}.
Hence a positive maximum value of $\phi_{\e}-{\tilde\varphi}_{\e}$ is achieved at a parabolic interior point 
$(x',t') \in B_{\tilde r}\times (-\tilde r^2,0]$. We have $\partial _t (\phi_{\e}-{\tilde\varphi}_{\e}) -\Delta (
\phi_{\e}-{\tilde\varphi}_{\e}) \geq 0$ at $(x',t')$ and $\phi_{\e}(x',t')>{\tilde \varphi}_{\e}  (x',t')$. 
The latter inequality combined with \eqref{rrest4} and \eqref{Was3} implies $W'({\tilde\varphi}_{\e})<W'(\phi_{\e})$.
By substituting the equations satisfied by $\phi_{\e}$ and ${\tilde \varphi}_{\e}$ into the former
inequality, we obtain
\begin{equation*}
\begin{split}
0&\leq \frac{\kappa}{2}(1-\phi_{\e})+\e {\tilde u}_{\e}\cdot \nabla{\tilde\varphi}_{\e}+W'({\tilde\varphi}_{\e}) 
 <  \frac{\kappa}{2}(1-\phi_{\e})+\e^{\frac34}\|\nabla{\tilde\varphi}_{\e}\|_{L^{\infty}}+W'(\phi_{\e}) \\
&\leq -\frac{\kappa}{2}(1-\phi_{\e})+\e^{\frac34}\|\nabla{\tilde\varphi}_{\e}\|_{L^{\infty}}
 \leq -\frac{\kappa}{2} \e^{\lambda}+\e^{\frac34}\|\nabla{\tilde\varphi}_{\e}\|_{L^{\infty}},
\end{split}
\end{equation*}
where $W'(\phi_{\e})\leq -\kappa(1-\phi_{\e})$ follows from \eqref{rrest4} and \eqref{Was3} and
$|{\tilde u}_{\e}|\leq \e^{-\beta}=\e^{-\frac14}$ by \eqref{supucon} and \eqref{cb14}.
We also used $\psi\geq {\tilde\psi}\geq 1$ in the last inequality. 
Since $\|\nabla{\tilde\varphi}_{\e}\|_{L^{\infty}}$ is bounded uniformly in $\e$ (see Lemma \ref{est1})
and $\lambda\leq 2/3<3/4$,
for sufficiently small $\e$, this is a contradiction. The other case may be proved similarly.
\end{proof}
\begin{lem}\label{alphad}
Under the assumptions of Lemma \ref{alpha}, there exist $\Cl[c]{c-21}$ and $\Cl[eps]{eps-7}$ with the following property.
For $t_0 \in (\delta,T)$ and $0<r<1/2$ define
\begin{equation}
Z_{r,t_0}:=\{x\in \Omega \, :\, \inf_{B_r(x)\times(t_0-r^2,t_0)} |\varphi_{\e}|<\alpha\}.
\label{zrtdef}
\end{equation}
If $0<\e<\Cr{eps-7}$, then 
\begin{equation}
{\mathcal L}^n(Z_{r,t_0})\leq \Cr{c-21} r.
\label{alphad1}
\end{equation}
\end{lem}
\begin{proof}
For $x_0\in Z_{r,t_0}$, we claim that there exist positive constants $\Cl[c]{c-23}$ and $\Cl[c]{c-24}$ such that
\begin{equation}
\mu_{t_0-2r^2}^{\e}(B_{\Cr{c-23} r}(x_0))\geq \Cr{c-24} r^{n-1}.
\label{alphad2}
\end{equation}
Once \eqref{alphad2} is proved, the Besicovitch covering theorem and  \eqref{e0sup} prove \eqref{alphad1}
with an appropriate choice of $\Cr{c-21}$. 
To prove \eqref{alphad2}, for each $x_0\in Z_{r,t_0}$ we have $(x',t')\in B_r(x_0)\times (t_0-r^2,t_0)$
such that $|\varphi_{\e}(x',t')|<\alpha$. Just as in the proof of Lemma \ref{lem-low}, we have 
\begin{equation}
3 \Cr{c-24}\leq \int_{\Omega} {\tilde\rho}_{(x',t'+\e^2)}(x,t')\, d\mu_{t'}^{\e}(x).
\label{alphad3}
\end{equation}
By \eqref{longeq} with $t_1$ and $t_0$ there replaced by $t'$ and $t_0-2r^2$, and restricting 
$r$ and $\e$ appropriately depending on constants appearing in the right-hand side of \eqref{longeq},
we obtain
\begin{equation}
 \left.\int_{\Omega} {\tilde\rho}_{(x',t'+\e^2)}(x,t)\, d\mu_{t}^{\e}(x)\right|_{t=t_0-2r^2}^{t'}\leq \Cr{c-24}.
\label{alphad4}
\end{equation}
The inequalities \eqref{alphad3} and \eqref{alphad4} show that
\begin{equation}
2\Cr{c-24}\leq \int_{\Omega}{\tilde\rho}_{(x',t'+\e^2)} (x,t_0-2r^2)\, d\mu_{t_0-2r^2}^{\e}(x).
\label{alphad5}
\end{equation}
Using the estimate \eqref{denconcl}, we may choose a large $\Cr{c-23}>1$ depending only on $D_1$ so that
\begin{equation}
\int_{\Omega\setminus B_{\Cr{c-23}r}(x')} {\tilde\rho}_{(x',t'+\e^2)}(x,t_0-2r^2)\, 
d\mu_{t_0-2r^2}^{\e}(x)\leq \Cr{c-24}.
\label{alphad6}
\end{equation}
By \eqref{alphad5} and \eqref{alphad6} we obtain
\begin{equation}
\Cr{c-24}\leq\int_{B_{\Cr{c-23}r}(x')} {\tilde\rho}_{(x',t'+\e^2)}(x,t_0-2r^2)\, d\mu_{t_0-2r^2}^{\e}(x).
\label{alphad7}
\end{equation}
Since ${\tilde\rho}_{(x',t'+\e^2)}(x,t_0-2r^2)\leq r^{1-n}$ and $B_{\Cr{c-23}r}(x')\
\subset B_{(\Cr{c-23}+1)r}(x_0)$, by setting $\Cr{c-23}+1$ to be again $\Cr{c-23}$, we
obtain \eqref{alphad2}. We restricted $r$ to be small, but when $r$ does not satisfy 
the restriction, we may choose $\Cr{c-21}$ large so that \eqref{alphad1} holds 
trivially.
\end{proof}
\begin{pro}\label{b}
Suppose $\varphi_{\e}$ and $u_{\e}$ are the solutions for \eqref{beq1} constructed in Section \ref{secexm}.
Given $0<\delta<T$ and $0<s<1$, there exist $0<b<1$ and $0<\Cl[eps]{eps-8}<1$ such that
\begin{equation}
\int _{ \{ x\in\Omega\ : \ |\varphi_{\e}(x,t) |\geq 1-b \}} \frac{W(\varphi_{\e}(x,t))}{\e}\, dx\leq s
\label{s}
\end{equation}
for all $t\in (\delta,T)$ if $0<\e\leq \Cr{eps-8}$.
\end{pro}
\begin{proof}
We restrict $0<b$ to be small in the following independent of $\e$. 
Assume that 
\begin{equation}
1-\sqrt{b}>\alpha, \ \ \ \Cr{c-20}|\log b|\geq 1.
\label{prob1}
\end{equation}
 Choose $J=J( \e ,b) \in \mathbb{N}$ such that 
\begin{equation}
\e ^{\frac{1}{2^{J+1}}} \in (b,\sqrt{b}].
\label{prob2}
\end{equation}
Restrict $\e$ so that $\e \leq\min\{\Cr{eps-6},\Cr{eps-7}\}$ and $\Cr{c-20} |\log \e |\leq \e ^{-1}\min\{\sqrt{\delta/2},1/2\}$. 
Note that, with this choice of $b$ and $J$, we have by \eqref{prob2} and \eqref{prob1} that
\begin{equation}
\Cr{c-20} \frac{1}{2^J}|\log \e|\geq \Cr{c-20} |\log b|\geq 1.
\label{prob3}
\end{equation}
Fix $t_0 \in (\delta,T)$
and we define
\begin{equation}
 A_j :=\{ x\in \Omega \ : \ 1-\e ^{\frac{1}{2^{j+1}}} \leq |\varphi_{\e} (x,t_0)| \leq 1-\e ^{\frac{1}{2^{j}}} \} \qquad \text{for} \ j=1,\dots , J. 
\label{prob3.5}
\end{equation}
For any point $x_0\in A_j$, we apply Lemma \ref{alpha} with $\lambda=\frac{1}{2^j}$. Note that the 
condition \eqref{rrest1} is satisfied due to \eqref{prob3}. Thus setting $\tilde r:= \Cr{c-20}|\log \e|/2^j$, we obtain
\begin{equation}
\inf_{B_{\e \tilde r}(x_0)\times (t_0 -\e ^2 \tilde r ^2,t_0)}|\varphi_\e|  <\alpha.
\label{prob4}
\end{equation}
With the notation of \eqref{zrtdef}, \eqref{prob4} shows 
\begin{equation}
A_j\subset Z_{\Cr{c-20}\e|\log \e|/2^j,t_0}
\label{prob5}
\end{equation}
and the application of Lemma \ref{alphad} to \eqref{prob5} shows
\begin{equation}
{\mathcal L}^n (A_j)\leq \Cr{c-21}\Cr{c-20} 2^{-j}\e|\log \e|
\label{prob6}
\end{equation}
for all $j=1,\cdots,J$.
On $A_j$, by $|\varphi_\e |\geq 1-\e ^{\frac{1}{2^{j+1}}}$, we have
\begin{equation}
\frac{W(\varphi_\e)}{\e} \leq \big( \max _{ [-1 , 1]}|W''|\big)\cdot \e ^{-1} \frac{(\e ^{\frac{1}{2^{j+1}}})^2}{2}\leq 
c(W) \e ^{2^{-j}-1}. 
\label{prob7}
\end{equation}
Set $Y:=\{x\in\Omega \ :\  1-b \leq |\varphi_{\e}(x,t_0) | \leq 1-\sqrt{\e} \}$. By \eqref{prob3.5} and \eqref{prob2}, 
we have
\begin{equation}
Y \subset \cup _{j=1} ^J A_j.
\label{prob8}
\end{equation}
Combining \eqref{prob6}-\eqref{prob8} and setting $\Cl[c]{c-25}:=c(W) \Cr{c-21}\Cr{c-20}$, 
\begin{equation}
\begin{split}
\int _Y \frac{W(\varphi_{\e})}{\e} &\leq \sum _{j=1} ^J \int _{A_j}\frac{W(\varphi_{\e})}{\e} \leq \Cr{c-25}
|\log \e | \sum _{j=1} ^J 2^{-j} \e ^{2^{-j}}\leq \Cr{c-25}|\log \e | \int _{1}^{J+1} 2^{-t}\e ^{2^{-t}}\, dt \\
&= \Cr{c-25}\frac{\e ^{\frac{1}{2^{J+1}}}-\sqrt{\e}}{\log 2}\leq \frac{\Cr{c-25}\sqrt{b}}{\log 2}
\end{split}
\label{prob9}
\end{equation}
where we used the fact that $2^{-x}\e^{2^{-x}}$ is monotone increasing for $x\in [1,J+1]$ as long as $\log\sqrt{b}
\leq -1$, and \eqref{prob2}. We restrict $b$ so that the right-hand side of \eqref{prob9} is less than $s/2$. 
The similar estimate shows 
\begin{equation}
\int _{\{1-\sqrt{\e}\leq |\varphi_{\e}|\leq 1-\e^{\frac23}\}} \frac{W(\varphi_{\e} )}{\e}\leq \Cr{c-25} \e |\log \e |. 
\label{prob10}
\end{equation}
Recalling that $|\varphi_{\e}|\leq 1$, we have
\begin{equation}
\int _{ \{ 1-\e ^{\frac{2}{3}} \leq | \varphi_{\e} |\} } \frac{W(\varphi_{\e} )}{\e} \leq c(W) (\e ^{\frac{2}{3}})^2 \cdot \frac{1}{\e} \leq c(W)\e^{\frac{1}{3}}.
\label{prob11}
\end{equation}
By \eqref{prob9}-\eqref{prob11} we restrict $\e$ depending on $s$ so that we have \eqref{s}. 
\end{proof}
\subsection{Proof of integrality}
Finally we prove the integrality of $\mu _t$. 
\begin{thm}\label{integralitymu}
For a.e. $t>0$, $\mu_t=\theta{\mathcal H}^{n-1}\lfloor_{M_t}$, where
$M_t$ is countably $(n-1)$-rectifiable and $\theta (x,t) = N(x,t)\sigma$ for some ${\mathcal H}^{n-1}$
measurable integer-valued function,
$\mu _t$ a.e. $x\in \Omega$.
\end{thm}
\begin{proof}
By the argument in the proof of Proposition \ref{rec1}, for a.e$.$ $t\geq 0$, we may choose a 
subsequence $\{V_t^{\e_{i_j}}\}_{j=1}^{\infty}$  such that \eqref{assvar6} and (with the notation of \eqref{stack3})
\begin{equation}
c_h(t):=\sup_{j} \int_{\Omega}\e_{i_j}|h_{\e_{i_j}}\nabla\varphi_{\e_{i_j}}|(x,t)\, dx<\infty
\label{inpr1}
\end{equation}
hold while $V_t^{\e_{i_j}}\rightarrow V_t$. Here $V_t$ is the rectifiable varifold uniquely 
determined by $\mu_t$ and recall that $\mu_t=\|V_t\|$. 
In the following we fix any such $t$ and show the claim of the theorem for $\mu_t$.
All functions are evaluated at the same $t$, and we do not write out the time variable (except for
$\mu_t$ and $V_t$ with or without ${\e_i}$) for simplicity. Moreover, though it is important to note that we are discussing a particular
subsequence (or its further subsequence), we denote $\e_{i_j}$ by $\e_i$ for simplicity. 

For any $m\in {\mathbb N}$, we define
\begin{equation}
A_{i,m}:=\big\{x\in \Omega\ :\ \int_{B_r(x)}\e_i |h_{\e_i}\nabla\varphi_{\e_i}|\, dx\leq m \mu_t^{\e_i}(B_r(x))\ \mbox{ for all 
$0<r<1/2$}\big\}.
\label{inpr2}
\end{equation}
The Besicovitch covering theorem with \eqref{inpr1} and \eqref{inpr2} shows that
\begin{equation}
\mu_t^{\e_i}(\Omega\setminus A_{i,m})\leq \frac{c(n) c_h(t) }{m}.
\label{inpr3}
\end{equation}
We then set
\begin{equation}
A_m:=\{x\in \Omega\ : \ \mbox{there exist $x_i\in A_{i,m}$ for infinitely many $i$ with $x_i\rightarrow x$}\}
\label{inqr4}
\end{equation}
and 
\begin{equation}
A:=\cup_{m=1}^{\infty}A_m.
\label{inqr5}
\end{equation}
We claim
\begin{equation}
\mu_t(\Omega\setminus A)=0.
\label{inqr8}
\end{equation}
Otherwise, we would have a compact set $K\subset \Omega\setminus A$ such that
$\mu_t(K)\geq \frac12 \mu_t(\Omega\setminus A)$. 
For any $m\in {\mathbb N}$ we have $K\subset \Omega\setminus A_m$ by \eqref{inqr5}.
For each point $x\in K$, by \eqref{inqr4}, there exists a neighborhood of $x$ which does not
intersect with $A_{i,m}$ for all sufficiently large $i$. Due to the compactness, thus, there exist
$i_0$ and an open set $O_m$ such that $K\subset O_m$ and $O_m\cap A_{i,m}=\emptyset$ for all
$i\geq i_0$. Let $\phi_m\in C_c(O_m;{\mathbb R}^+)$ such that $0\leq \phi_m\leq 1$ and $\phi_m=1$ on
$K$. Then 
\begin{equation}
\mu_t(K)\leq \int_{\Omega}\phi_m\, d\mu_t=\lim_{i\to\infty} \int_{\Omega} \phi_m \, d\mu_t^{\e_i}
= \lim_{i\to \infty} \int_{\Omega\setminus A_{j,m}}\phi_m\, d\mu_t^{\e_i}\leq \liminf_{i\to\infty}\mu_t^{\e_i}(\Omega
\setminus A_{j,m})
\label{inqr1s}
\end{equation}
for all $j\geq i_0$. Since the last quantity of \eqref{inqr1s} is less than $c(n)c_h(t)/m$ by \eqref{inpr3}, and
since $m$ is arbitrary, we obtain $\mu(K)=0$. This proves the claim \eqref{inqr8}.

Since $\mu_t$ is rectifiable, $\mu_t$ a.e. point 
$x$ has an approximate tangent space. By \eqref{inqr8}, we may also assume that for $\mu_t$ a.e. $x$
there exists some $m\in {\mathbb N}$ such that $x\in A_m$. We fix any such point, and after
a parallel translation, we may assume that $x=0$. Furthermore, after a rotation, we may assume that the
approximate tangent space is $P:=\{x_n=0\}$. Denote 
$\theta:=\lim_{r\downarrow 0}\frac{\|V_t\|(B_r(x))}{\omega_{n-1}r^{n-1}}$. We will be done if we prove 
that $\sigma^{-1}\theta\in {\mathbb N}$.

For any sequence $r_i\downarrow 0$, we have $\lim_{i\rightarrow\infty}(\Phi_{r_i})_{\#} V_t=\theta
|P|$, where $\Phi_{r_i}(x)=\frac{x}{r_i}$ and $(\Phi_{r_i})_{\#}$ is the usual push-forward of varifold. $|P|$ is the unit density varifold naturally derived from $P$. 
Since $0\in A_m$, there exists a subsequence (denoted by the same index) $x_i\in A_{i,m}$ such 
that $\lim_{i\rightarrow\infty}x_i=0$. After choosing a further subsequence, we may assume that
\begin{equation}
\lim_{i\rightarrow\infty}(\Phi_{r_i})_{\#} V_t^{\e_i}=\theta |P|,
\label{inqr8.1}
\end{equation}
\begin{equation}
\lim_{i\rightarrow\infty}\frac{x_i}{r_i}=0
\label{inqr8.2}
\end{equation}
and 
\begin{equation}
\lim_{i\rightarrow\infty}\frac{\e_i^{\beta'-\beta}|\log \e_i|}{r_i^{n-1}}=0.
\label{inqrad1}
\end{equation}
For a such choice, we also have $\lim_{i\rightarrow\infty}
\frac{\e_i}{r_i}=0$.
Rescale the coordinates by $\tilde x:=\frac{x}{r_i}$ and define ${\tilde \e}_i:=\frac{\e_i}{r_i}\rightarrow 0$. 
Define $\tilde{\varphi}_{{\tilde\e}_i}(\tilde x):=\varphi_{\e_i}(r_i \tilde x)$. We also define 
${\tilde\xi}_{{\tilde\e}_i}$ and ${\tilde h}_{{\tilde \e}_i}$ as in \eqref{defti} and \eqref{stack3} 
corresponding to ${\tilde \e}_i$ and ${\tilde\varphi}_{{\tilde\e}_i}$. Due to \eqref{assvar6}, we may
choose a further subsequence so that
\begin{equation}
\lim_{i\rightarrow\infty} \int_{B_3} |{\tilde\xi}_{{\tilde\e}_i}|\, d{\tilde x}=0.
\label{inqr8.3}
\end{equation}
Due to Corollary \ref{atode} and \eqref{inqrad1}, for any $y\in B_2$ and $0<r<2$, we have
\begin{equation}
\int_0^r\frac{d{\tilde \tau}}{{\tilde\tau}^n}\int_{B_{\tilde\tau}(y)}({\tilde\xi}_{{\tilde\e}_i})_+\, d{\tilde x}
=\frac{1}{r_i^{n-1}}\int_0^{r r_i}\frac{d\tau}{\tau^n}\int_{B_{\tau}(r_i y)}(\xi_{\e_i})_+\, dx
\leq \frac{2\Cr{c-8} \e_i^{\beta'-\beta}|\log \e_i|}{r_i^{n-1}} \rightarrow 0
\label{inqrad2}
\end{equation}
as $i\rightarrow \infty$. For ${\tilde h}_{{\tilde\e}_i}$, we have
\begin{equation}
{\tilde\e}_i \int_{B_3}|{\tilde h}_{{\tilde \e}_i}\nabla{\tilde\varphi}_{{\tilde\e}_i}|\, d{\tilde x}=\frac{ \e_i }{r_i^{n-2}}\int_{B_{3r_i}}|h_{\e_i}
\nabla\varphi_{\e_i}|\, dx
\leq \frac{m}{r_i^{n-2}}\mu_t^{\e_i}(B_{4r_i}(x_i)) \leq m 4^{n-1}\omega_{n-1}D_1 r_i\rightarrow 0
\label{inqr9}
\end{equation}
as $i\rightarrow \infty$, where we used \eqref{inqr8.2}, $x_i\in A_{i,m}$, \eqref{inpr2} and \eqref{denconcl}.
If one defines a varifold ${\tilde V}_t^{{\tilde\e}_i}$ corresponding to ${\tilde\varphi}_{{\tilde \e}_i}$ as in \eqref{assvar}, then one
can check that ${\tilde V}_t^{{\tilde\e}_i}=(\Phi_{r_i})_{\#} V_t^{\e_i}$. 
Next we claim 
\begin{equation}
\int_{B_3} (1-(\nu_n)^2){\tilde\e}_i |\nabla {\tilde\varphi}_{{\tilde\e}_i}|^2\, d{\tilde x}\rightarrow 0
\label{inqr10}
\end{equation}
as $i\rightarrow \infty$, where $\nu=(\nu_1,\cdots,\nu_n)=\frac{\nabla{\tilde\varphi}_{{\tilde \e}_i}}{|\nabla  {\tilde\varphi}_{{\tilde \e}_i}|}$. Note first that $G_{n-1}({\mathbb R}^n)\cong {\mathbb S}^{n-1}/\{\pm 1\}$ and a function defined
by $\psi \, : \, \pm \nu\in {\mathbb S}^{n-1}/\{\pm 1\}\longmapsto 1-\nu_n^2$ is continuous. Thus for any 
$\phi\in C_c({\mathbb R}^n)$, we have by \eqref{inqr8.1}
\begin{equation}
{\tilde V}_t^{{\tilde\e}_i}(\phi\psi)=\int \phi({\tilde x})(1-(\nu_n)^2)\, d\|{\tilde V}_t^{{\tilde\e}_i}\|({\tilde x})
\rightarrow \theta|P|(\phi\psi)
\label{inqr11}
\end{equation}
and since $P=\{x_n=0\}$, 
\begin{equation}
\theta|P|(\phi\psi)=\theta\int_{P} \phi({\tilde x}) \psi((0,\cdot,0,\pm 1))\, d{\mathcal H}^{n-1}({\tilde x})=0.
\label{inqr12}
\end{equation}
In particular, \eqref{inqr11} and \eqref{inqr12} prove \eqref{inqr10}. In the following we fix
this subsequence and drop 
the tilde for simplicity. 

Assume that $N$ is the smallest positive integer greater than $\sigma ^{-1} \theta$, that is, 
\begin{equation}
\theta \in [(N-1)\sigma, N\sigma).
\label{intf0}
\end{equation}
Let $s>0$
be arbitrary. By Proposition $\ref{b}$ and \eqref{inqr8.3}, there exists $0<b<1$ such that
\begin{equation}
\int _{B_3 \cap \{ |\varphi_{\e_i}|\geq 1-b \}} \Big( \frac{\e _i |\nabla \varphi_{\e_i}|^2}{2} +\frac{W(\varphi_{\e_i})}{\e _i} \Big)\leq s
\label{intf1}
\end{equation}
for all sufficiently large $i$. Corresponding to $s$ and $b$ as well as $c$ given by Lemma \ref{est1}, 
by Proposition \ref{L}, we choose $\varrho$ and $L$ (with a restriction on $\e_i$).
Then with $R=2$, by Proposition \ref{a}, we restrict $\varrho$ further if necessary. We use Proposition \ref{a} with $a=L\e_i$. 
For all large $i$ we define
\begin{equation}
\begin{split}
G_i :=  B_2 \cap \{ |\varphi_{\e_i}|\leq 1-b \} 
 \cap \Big\{ x \ : \ & \int _{B_r (x)} \e _i |h_{\e_i}\nabla\varphi_{\e_i}| + |\xi_{\e_i}|  + (1-(\nu _n)^2 ) \e _i |\nabla \varphi_{\e_i}|^2  \\
& \qquad \qquad  \leq \varrho\, \mu_{t} ^{\e_i} (B_r (x)) \ \text {if} \ \e _i L\leq r\leq 1 \Big\}.
\end{split}
\label{intf2}
\end{equation}
By the Besicovitch covering theorem, we obtain
\begin{equation}
\mu_{t} ^{\e_i} (B_2 \cap \{ |\varphi_{\e_i}|\leq 1-b \}\setminus G_i )
\leq  c(n) \varrho^{-1} \int _{B_3} \e _i |h_{\e_i}\nabla\varphi_{\e_i}| + |\xi_{\e_i}|  + (1-(\nu _n)^2 ) \e _i |\nabla \varphi_{\e_i}|^2.
\label{intf3}
\end{equation}
The right hand side goes to $0$ as $i\to \infty$ by \eqref{inqr9}, \eqref{inqr8.3}, \eqref{inqr10}. 
Next we claim the following lower bound for all sufficiently large $i$:
\begin{equation}
\mu_t^{\e_i}(B_r(x))\geq (\sigma-2s)\omega_{n-1} r^{n-1}
\label{intf4}
\end{equation}
for all $L\e_i\leq r\leq 1$ and $x\in G_i$. To see this, first note that the assumptions of Proposition \ref{L} are all satisfied
due to Lemma \ref{est1}, \eqref{intf2} and \eqref{-beta}. This proves the inequality \eqref{intf4} with $r=L\e_i$ and with
$2s$ replaced by $s$. Next the identity \eqref{s1} with $\zeta_2\equiv 1$, \eqref{s2}, 
\eqref{inqrad2} and \eqref{intf2} shows
\begin{equation}
\frac{1}{\tau^{n-1}}\mu_t^{\e_i}(B_{\tau}(x))\Big|_{\tau=L\e_i}^r \geq o(1)-\int_{L\e_i}^r
\varrho \frac{\mu_t^{\e_i}(B_{\tau}(x))}{\tau^{n-1}}\, d\tau\geq o(1)-\omega_{n-1}D_1
\varrho
\label{intf5}
\end{equation}
after integrating over $[L\e_i,r]$. 
We may restrict $\varrho$ so that $D_1\varrho<s$. Thus \eqref{intf5} gives
\eqref{intf4} for all sufficiently large $i$. Since $\mu_t^{\e_i}=\|V_t^{\e_i}\|\rightarrow \theta {\mathcal H}^{n-1}\lfloor_P$,
\eqref{intf4} shows that points in $G_i$ converge uniformly to $P$ as $i\rightarrow \infty$. 

For any $x\in P \cap B_1$ and $|l| \leq 1-b$, we next prove 
\begin{equation}
\# (P^{-1} (x) \cap G_i \cap \{ \varphi_{\e_i}=l \})\leq N-1.
\label{intf8}
\end{equation}
If the claim were not true, we choose $N$ elements
and set it to be $Y$, and apply Proposition \ref{a} with $R=1$, $\varphi=\varphi_{\e_i}$ and $a=L\e _i$.  
The property $|y-z|>3L\e_i$ holds due to \eqref{clo3},
${\rm diam}\, Y\leq \varrho$ due to the uniform convergence of $G_i$ to $P$,  (6), (7) are due respectively to
\eqref{intf2} and \eqref{inqrad2}. Thus all the assumptions of Proposition \ref{a} are satisfied and we have
\begin{equation}
\sum_{y\in Y}\frac{1}{(L\e_i)^{n-1}}\mu_t^{\e_i}(B_{L\e_i}(y))\leq s+(1+s)\mu_t^{\e_i}(\{z\,:\, {\rm dist}\,(Y,z)<1\})
\label{intf6}
\end{equation}
for all sufficiently large $i$. 
Since $\lim_{i\rightarrow \infty}\mu_t^{\e_i}(\{z\,:\, {\rm dist}\,(Y,z)<1\})=\theta \omega_{n-1}$, $\#Y=N$ and \eqref{intf4}, 
we obtain
\begin{equation}
N(\sigma-2s)\omega_{n-1}\leq s+(1+s)\theta \omega_{n-1}.
\label{intf7}
\end{equation}
Since $\sigma N>\theta$ by definition, \eqref{intf7} is a contradiction for sufficiently small $s$ depending only on
$\sigma$, $\theta$ and $n$. Thus we proved \eqref{intf8}. 

To conclude the proof, we consider push-forward of ${\hat V}_t^{\e_i}:=V_t^{\e_i}\lfloor_{\{ |x_n|\leq 1\}\times {\bf G}(n,n-1)}$ by $P$, $P_{\#} {\hat V}_t^{\e_i}$. For any $\phi(x,S)\in C_c((P\cap B_2)\times {\bf G}(n,n-1))$, we have (for all sufficiently large $i$)
\begin{equation}
P_{\#} {\hat V}_t^{\e_i}(\phi)=\int_{\{|x_n|\leq 1\}} \phi(P(x), P) |\Lambda_{n-1}P\circ
(I-\nu\otimes\nu)|\, d\mu_t^{\e_i}.
\end{equation}
Here $\Lambda_{n-1} A$ denotes the Jacobian of 
$A\in {\rm Hom}({\mathbb R}^n;{\mathbb R}^n)$ (\cite{allard}). 
One can check that $|\Lambda_{n-1} P\circ (I-\nu\otimes\nu)|=|\nu_n|=\frac{|\partial_{x_n}\varphi_{\e_i}|}{|\nabla\varphi_{\e_i}|}$. 
Due to the varifold convergence \eqref{inqr8.1}, we have $P_{\#}{\hat V}_t^{\e_i}\rightarrow P_{\#}(\theta|P|)=\theta|P|$ as $i\rightarrow\infty$. 
In the following we also use
\begin{equation}
\lim_{i\rightarrow\infty}\int_{B_3}\Big| \frac{\e_i |\nabla \varphi_{\e_i}|^2}{2}+\frac{W(\varphi_{\e_i})}{\e_i}-|\nabla \varphi_{\e_i} | \sqrt{2W(\varphi _{\e_i})}  \Big|\,
dx=0
\label{intf9}
\end{equation}
which follows from \eqref{inqr8.3}. Now we have
\begin{equation}
\begin{split}
\omega_{n-1} \theta &=\|\theta|P| \|(B_1)=\lim_{i\rightarrow\infty}
\|P_{\#} {\hat V}_t^{\e_i}\|(B_1)=\lim _{i\to \infty} \int _{B_1 } |\nu_n|\, d\mu_t^{\e_i} \\
&\leq \liminf _{i\to \infty } \int _{B_1 \cap \{ |\varphi_{\e_i} |\leq 1-b \} \cap G_i} |\nu_n|\, d\mu_t^{\e_i}+2s \\
&\leq \liminf_{i\to \infty} \int_{B_1\cap \{|\varphi_{\e_i} |\leq 1-b\}\cap G_i}|\nu_n| |\nabla\varphi_{\e_i}|\sqrt{2W(\varphi_{\e_i})}\, dx
+2s
\end{split}
\label{intf10}
\end{equation}
due to \eqref{intf1}, \eqref{intf3} and \eqref{intf9}.
By the co-area formula \cite[10.6]{simon}, we obtain 
\begin{equation}
\int_{B_1\cap \{|\varphi_{\e_i} |\leq 1-b\}\cap G_i}|\nu_n||\nabla\varphi_{\e_i}|\sqrt{2W(\varphi_{\e_i})}\, dx
=\int_{-1+b}^{1-b}\, d\tau \int_{ \{\varphi_{\e_i}=\tau\}\cap B_1\cap G_i} |\nu_n| \sqrt{2W(\tau)}\, d{\mathcal H}^{n-1}.
\label{intf11}
\end{equation}
Then by the area formula \cite[12.4]{simon} applied to the map $P\, :\, \{\varphi_{\e_i}=\tau\}\rightarrow \{x_n=0\}$, we have
\begin{equation}
\int_{\{\varphi_{\e_i}=\tau\}\cap B_1\cap G_i} |\nu_n|\, d{\mathcal H}^{n-1}=\int_{\{x_n=0\}} {\mathcal H}^0 
(\{\varphi_{\e_i}=\tau\}\cap B_1\cap G_i\cap P^{-1}(x))\, d{\mathcal H}^{n-1}(x).
\label{intf12}
\end{equation}
Now the integrand of the right-hand side of \eqref{intf12} is $\leq N-1$ due to \eqref{intf8} for $|x|\leq 1$, and 0 otherwise.
Combining \eqref{intf10}-\eqref{intf12}, we finally obtain
\begin{equation}
\omega_{n-1}\theta\leq 2s+ \liminf_{i\to\infty} \omega_{n-1}(N-1)\int_{-1+b}^{1-b}\sqrt{2W(\tau)}\, d\tau\leq 2s+\omega_{n-1}(N-1)\sigma.
\label{intf13}
\end{equation}
Since $s>0$ is arbitrary, \eqref{intf13} shows $\theta\leq (N-1)\sigma$. By \eqref{intf0}, we have $\theta = (N-1)\sigma$.
\end{proof}

\section{Proof of  the Main theorem}
We finally define a family of varifolds which will be a generalized solution of \eqref{mcf}.
To remove the multiple of $\sigma$, we re-define $V_t$ as follows.
\begin{definition}
For a.e$.$ $t\geq 0$ when $\mu_t$ is rectifiable and integral modulo division by $\sigma$, 
let $V_t$ be the uniquely defined integral varifold by $\sigma^{-1}\mu_t$. For any other $t>0$,
define $V_t$ by $V_t(\phi):=\sigma^{-1}\int_{U} \phi(x,P_0)\, d\mu_t(x)$ for $\phi\in C_c(G_{n-1}(U))$, 
where $P_0\in {\bf G}(n,n-1)$ is an arbitrary fixed element. 
\label{varvdef}
\end{definition}
With this definition, we have $\|V_t\|=\sigma^{-1}\mu_t$ for all $t\geq 0$, and $V_t\in
{\bf IV}_{n-1}(\Omega)$ for a.e$.$ $t\geq 0$ by Theorem \ref{integralitymu}. Thus 
(a) of Definition \ref{gsdef} is satisfied. The condition (b) is satisfied due to \eqref{denconcl}. 
Let us consider (c). 
The $L^2$ integrability of $u$, $\int_0^T\int_{\Omega} |u|^2\, d\|V_t\|dt<\infty$, may be proved as in
\eqref{e02} and \eqref{e03} once (b) is established. For $h$, we prove the following. 
\begin{pro}
For a.e. $t\geq 0$, $V_t$ has a generalized mean curvature $h(V_t)$ and we have
\begin{equation}
\int _\Omega \phi |h(V_t)|^2 \, d\|V_t\|  
\leq \sigma^{-1}\liminf _{i\to \infty} \int _\Omega \e _i \phi\big( \Delta \varphi _{\e _i} -\frac{W'(\varphi _{\e _i} )}{\e _i ^2} \big)^2 \, dx<\infty
\label{Hineq}
\end{equation}
for any $\phi\in C_c(\Omega\,;\,{\mathbb R}^+)$.
\label{Hineqpr}
\end{pro}
\begin{proof} 
Just as in the proof of Proposition \ref{rec1}, for a.e$.$ $t\geq 0$, we may assume \eqref{assvar6} and \eqref{assvar5}
and there exists a subsequence $\{V_t^{\e_{i_j}}\}_{j=1}^{\infty}$ converging to $\sigma V_t$ (note
that we re-defined $V_t$) with \eqref{assvar5.1}. By arguing 
as in the proof of Proposition \ref{rec1}, for any $g\in C_c^1 (\Omega\,;\, {\mathbb R}^n)$, we have
\begin{equation}
|\delta V_t(g)|\leq \sigma^{-1}\big(\int_{\Omega} |g|^2\, d\mu_t\big)^{1/2}\liminf_{j\rightarrow\infty}\big(\int_{\Omega}
\e_{i_j}(\Delta\varphi_{\e_{i_j}}-\frac{W'}{\e_{i_j}^2})^2\, dx\big)^{1/2}.
\label{kte1}
\end{equation}
The inequality and \eqref{assvar5.1} show that the total variation $\|\delta V_t\|$ of $\delta V_t$
is absolutely continuous with respect to $\mu_t=\sigma\|V_t\|$. Thus by the Radon-Nikodym theorem
there exists a $\|V_t\|$ measurable vector field $h(V_t)$ (generalized mean curvature vector) such that
\begin{equation}
\delta V_t(g)=-\int_{\Omega} g\cdot h(V_t)\, d\|V_t\|.
\label{kte2}
\end{equation}
Since $V_t$ is rectifiable, going back to the definition of countably $(n-1)$-rectifiable set, one can
show that $C_c^1(\Omega)$ is dense in $L^2(\|V_t\|)$. Then a standard approximation argument shows 
$h(V_t)\in L^2(\|V_t\|)$ and \eqref{Hineq} with $\phi=1$. Next, given $\phi\in C_c(\Omega\,;\,{\mathbb R}^+)$, 
let $\psi_j\in C_c^1(\Omega\,;\,{\mathbb R}^+)$ be a sequence such that $\lim_{k\rightarrow\infty}\|\phi-
\psi_k\|_{C^0(\Omega)}=0$. Using $\psi_k g$ in the proof of Proposition \ref{rec1} and letting
$k\to\infty$, we obtain
\begin{equation}
\big|\int_{\Omega}\phi g\cdot h(V_t)\, d\mu_t\big|\leq
 \big(\int_{\Omega}\phi |g|^2\, d\mu_t\big)^{1/2}
\liminf_{j\to \infty}\big(\int_{\Omega} \e_{i_j}\phi(\Delta\varphi_{\e_{i_j}}-\frac{W'}{\e_{i_j}^2})^2\, dx\big)^{1/2}.
\label{kte3}
\end{equation}
By approximation, we obtain \eqref{Hineq} from \eqref{kte3}.
\end{proof}
Now Proposition \ref{Hineqpr} combined with Lemma \ref{erlem3} and Fatou's lemma proves (c). 
For the proof of (d), 
one point which we need to be careful about is that we may not have the whole sequence
$\{V_t^{\e_i}\}_{i=1}^{\infty}$ converging to $V_t$ as varifold for a.e$.$ $t\geq 0$ even though $\{\|V_t^{\e_i}\|\}_{i=1}^{\infty}$ converges
to $\sigma \|V_t\|=\mu_t$ for all $t\geq 0$.  
\begin{pro}
The family of varifolds $\{V_t\}_{t\geq 0}$ defined in Definition \ref{varvdef} is a generalized solution of 
\eqref{mcf}.
\label{prmcffin}
\end{pro}
\begin{proof} We prove \eqref{mcfweak} for $\phi\in C_c^2 (\Omega\times[0,\infty)\,;\, {\mathbb R}^+)$. For
$\phi\in C_c^1$, one can approximate $\phi$ by a sequence of $C_c^2$ functions and obtain the same result in
the limit. First by modifying \eqref{derivativemu} we obtain (with the notation \eqref{stack3})
\begin{equation}
\begin{split}
\mu _{t} ^{\e _i }(\phi(\cdot,t))\Big|_{t=t_1}^{t_2}= \int _{t_1} ^{t_2} \int_\Omega &-\e_i \phi h_{\e_i} ^2 -\e_i h_{\e_i}\nabla \phi \cdot \nabla \varphi_{\e_i}  
 +\e_i \phi h_{\e_i}
u_{\e_i}\cdot \nabla \varphi_{\e_i} \\
&+ \e_i (\nabla \varphi_{\e_i} \cdot \nabla \phi ) (u_{\e_i}\cdot \nabla \varphi_{\e_i} ) \, dxdt+\int_{t_1}^{t_2}
\int_{\Omega}\frac{\partial \phi}{\partial t}\, d\mu_t^{\e_i}dt. 
\end{split}
\label{brakke1}
\end{equation}
Modulo division by $\sigma$, 
the left-hand side of \eqref{brakke1} converges to that of \eqref{mcfweak} due to Proposition \ref{seqconv}.
The same is true for the last term of \eqref{brakke1}. Hence we focus on the middle 4 terms. 
First we approximate $u_{\e_i}$ by a fixed smooth ${\tilde u}$ as follows. Given $\epsilon>0$, we choose 
a large $j$ so that $t_2<T_j$ and 
\begin{equation}
\|u-u_{\e_j}\|_{L^q([0,T_j];W^{1,p}(\Omega))}<\epsilon\ \ \mbox{and}\ \ 
\|u_{\e_j}-u_{\e_i}\|_{L^q([0,T_j];W^{1,p}(\Omega))}<\epsilon
\label{brakke1.5}
\end{equation}
for all $i\geq j$. This is possible since $u_{\e_i}$ converges to $u$ in this norm. 
Set ${\tilde u}:=u_{\e_j}$. Then we have
\begin{equation}
\begin{split}
&\Big|\int_{t_1}^{t_2}\int_{\Omega} \e_i\phi h_{\e_i}(u_{\e_i}-{\tilde u})\cdot\nabla\varphi_{\e_i}
+\e_i (\nabla\varphi_{\e_i}\cdot \nabla\phi)((u_{\e_i}-{\tilde u})\cdot \nabla\varphi_{\e_i})\, dxdt\Big| \\
&\leq \Big(\int_{t_1}^{t_2}\int_{\Omega} 2\e_i (\phi^2 h_{\e_i}^2+|\nabla\phi|^2 |\nabla\varphi_{\e_i}|^2)\,dxdt\Big)^{1/2}
\Big(\int_{t_1}^{t_2}\int_{\Omega} |u_{\e_i}-{\tilde u}|^2\, d\mu_t^{\e_i}dt\Big)^{1/2}.
\end{split}
\label{brakke2}
\end{equation}
As in the proof of Lemma \ref{erlem3}, and by \eqref{denconcl} and \eqref{brakke1.5}, we have
\begin{equation}
\int_{t_1}^{t_2}dt\int_{\Omega}|u_{\e_i}-{\tilde u}|^2\, d\mu_t^{\e_i}\leq c(n) D_1 (t_2-t_1)^{1-\frac{2}{q}}\|u_{\e_i}-{\tilde u}\|_{
L^q([t_1,t_2];W^{1,p}(\Omega))}^2<c\epsilon^2.
\label{brakke3}
\end{equation}
By \eqref{brakke2} and \eqref{brakke3}, replacing $u_{\e_i}$ by ${\tilde u}$ in \eqref{brakke1} produces
error of $c\epsilon^2$. Similarly we have
\begin{equation}
\Big|\int_{t_1}^{t_2}\int_{\Omega}( -h\phi+\nabla\phi)\cdot ( (u-{\tilde u})\cdot \nu)\nu\, d\mu_tdt\Big|
\leq c'\epsilon.
\label{brakke3.5}
\end{equation}
Thus we will finish the proof if we prove
\begin{equation}
\begin{split}
\liminf_{i\to\infty}\int_{t_1}^{t_2}\int_{\Omega}& -\e_i \phi h_{\e_i}^2-\e_i h_{\e_i}\nabla\phi\cdot\nabla\varphi_{\e_i}
+\e_i\phi h_{\e_i}{\tilde u}\cdot\nabla\varphi_{\e_i} \\ & +\e_i (\nabla\varphi_{\e_i}\cdot \nabla\phi)({\tilde u}\cdot
\nabla\varphi_{\e_i})\, dxdt 
\leq \int_{t_1}^{t_2}{\mathcal B}(\mu_t,{\tilde u}(\cdot,t),\phi(\cdot,t))\, dt,
\end{split}
\label{brakke4}
\end{equation}
where we denote
\begin{equation*}
{\mathcal B}(\mu_t,{\tilde u}(\cdot,t),\phi(\cdot,t)):=\int_{\Omega}
(\nabla\phi-h\phi)\cdot(h+(\tilde{u}\cdot\nu)\nu)\, d\mu_t.
\end{equation*}
By the Cauchy-Schwarz inequality, we have
\begin{equation}
\begin{split}
{\hat a}_i(t):=&\int_{\Omega} -\e_i \phi h_{\e_i}^2-\e_i h_{\e_i}\nabla\phi\cdot\nabla\varphi_{\e_i}
+\e_i\phi h_{\e_i}{\tilde u}\cdot\nabla\varphi_{\e_i}  +\e_i (\nabla\varphi_{\e_i}\cdot \nabla\phi)({\tilde u}\cdot
\nabla\varphi_{\e_i})\, dx \\
&\leq \int_{\Omega}\frac{\e_i}{2}|\nabla\varphi_{\e_i}|^2( \frac{|\nabla\phi|^2}{\phi}+\phi|{\tilde u}|^2
+2 |{\tilde u}||\nabla\phi|)\, dx \\
&\leq \int_{\Omega}\frac{\e_i}{2}|\nabla\varphi_{\e_i}|^2( \hat{\phi}+\phi|{\tilde u}|^2
+2 |{\tilde u}||\nabla\phi|)\, dx=: {\hat b}_i(t),
\end{split}
\label{brakke5}
\end{equation}
where $\hat\phi\in C_c(\Omega;{\mathbb R}^+)$ is chosen so that $\frac{|\nabla\phi|^2}{\phi}\leq \hat\phi$. 
This in particular shows $\hat{b}_i(t)-\hat{a}_i(t)\geq 0$ for $t_1\leq t\leq t_2$. Using the general fact that $\liminf_{i\to\infty}
(a_i+b_i)\leq \limsup_{i\to\infty}a_i+\liminf_{i\to\infty}b_i$ and Fatou's lemma, we have
\begin{equation}
\begin{split}
\liminf_{i\to\infty}\int_{t_1}^{t_2} \hat{a}_i(t)\, dt & \leq- \liminf_{i\to\infty} \int_{t_1}^{t_2} (\hat{b}_i(t)-\hat{a}_i(t))\, dt
+\liminf_{i\to\infty} \int_{t_1}^{t_2} \hat{b}_i(t)\, dt\\
& \leq - \int_{t_1}^{t_2} \liminf_{i\to\infty} (\hat{b}_i(t)-\hat{a}_i(t))\, dt
+\liminf_{i\to\infty} \int_{t_1}^{t_2} \hat{b}_i(t)\, dt.
\end{split}
\label{brakke6}
\end{equation}
Since ${\hat b}_i(t)$ converges to $\frac12\int_{\Omega}( \hat\phi+\phi|{\tilde u}|^2
+2 |{\tilde u}||\nabla\phi|)\, d\mu_t$ for all $t_1\leq t\leq t_2$ and bounded uniformly, from \eqref{brakke6}
and the dominated convergence theorem we have
\begin{equation}
\liminf_{i\to\infty}\int_{t_1}^{t_2} \hat{a}_i(t)\, dt\leq -\int_{t_1}^{t_2}\liminf_{i\to\infty}(-\hat{a}_i(t))\, dt.
\label{brakke7}
\end{equation}
Thus we may finish the proof of \eqref{brakke4} via \eqref{brakke7} if we prove
\begin{equation}
-\liminf_{i\to\infty} (-\hat{a}_i(t))\leq {\mathcal B}(\mu_t, {\tilde u}(\cdot,t),\phi(\cdot,t))
\label{brakke8}
\end{equation}
for a.e$.$ $t\in [t_1,t_2]$. Fix $t$ such that the claim of Proposition \ref{Hineqpr} holds. Let $\{\e_{i_j}\}_{j=1}^{\infty}$ be a subsequence such that 
\begin{equation}
\liminf_{i\to\infty}(-\hat{a}_i(t))=\lim_{j\to\infty} (-\hat{a}_{i_j}(t)).
\label{brakke8.5}
\end{equation}
We may choose a further subsequence
(denoted by the same index) such that $V_t^{\e_{i_j}}\to \sigma {\tilde V}_t$ as varifold. By the Cauchy-Schwarz inequality,
\begin{equation}
-\hat{a}_i(t)\geq \int_{\Omega}\frac12 \e_i \phi h_{\e_i}^2 -\big(\frac{|\nabla\phi|^2}{\phi}+|\tilde{u}|^2+|\tilde{u}||\nabla\phi|\big)\e_i
|\nabla\varphi_{\e_i}|^2\, dx
\label{brakke9}
\end{equation}
where the last negative term is bounded uniformly. If $\liminf_{j\to\infty}\int_{\Omega}\e_{i_j}\phi h_{\e_{i_j}}^2\, dx=\infty$,
we have \eqref{brakke8} with the left-hand side $=-\infty$. Thus we may assume otherwise. 
At this point, arguing just as
in the proof of Proposition \ref{rec1}, we may prove that ${\tilde V}_t\lfloor_{\{\phi>0\}}$ is rectifiable and ${\tilde V}_t\lfloor_{\{\phi>0\}}
=V_t\lfloor_{\{\phi>0\}}$. Then the argument in the proof of Proposition \ref{Hineqpr} shows \eqref{Hineq}.
For the remaining three terms in $\hat{a}_{i_j}(t)$, since $V_t^{\e_{i_j}}\lfloor_{\{\phi>0\}}\to\sigma V_t\lfloor_{\{\phi>0\}}$ 
as varifold and by \eqref{assvar}, we have for any $\tilde{\phi}\in C^2_c (\{\phi>0\}\,;\, {\mathbb R}^+)$
\begin{equation}
\begin{split}
&\lim_{j\to\infty} \int_{\Omega} \e_{i_j} h_{\e_{i_j}}\nabla\tilde\phi\cdot\nabla\varphi_{\e_{i_j}} - \e_{i_j} \tilde\phi
h_{\e_{i_j}} \tilde{u}\cdot\nabla\varphi_{\e_{i_j}} - \e_{i_j}(\nabla\varphi_{\e_{i_j}}\cdot\nabla\tilde\phi)
(\tilde{u}\cdot \nabla\varphi_{\e_{i_j}})\, dx \\
&=\sigma\delta V_t(\nabla\tilde\phi-\tilde{u}\tilde\phi)-\int_{\Omega}(\nabla\tilde\phi\cdot \nu)(\tilde{u}\cdot\nu)\, d\mu_t 
=\int_{\Omega} -h\cdot(\nabla\tilde\phi-\tilde{u}\tilde\phi)-(\nabla\tilde\phi\cdot \nu)(\tilde{u}\cdot\nu)\, d\mu_t.
\end{split}
\label{brakke10}
\end{equation}
We may construct a sequence of approximation $\{\tilde\phi_{k}\}_{k=1}^{\infty}$ such that
$\lim_{k\to\infty}\|\phi-\tilde\phi_k\|_{C^2}=0$, $\phi\geq \tilde\phi_k$ and ${\rm spt}\,\tilde\phi_k\subset \{\phi>0\}$.
For such approximating sequence, 
\begin{equation}
\begin{split}
\Big|\int_{\Omega} \e_{i_j}h_{\e_{i_j}}\nabla(\phi-\tilde\phi_k)\cdot \nabla\varphi_{\e_{i_j}}\Big| 
&\leq \big(\int_{\Omega}\e_{i_j} h_{\e_{i_j}}^2 \phi\big)^{1/2}\big(\int_{\Omega}
\frac{|\nabla(\phi-\tilde\phi_k)|^2}{\phi-\tilde\phi_k} \e_{i_j}|\nabla\varphi_{\e_{i_j}}|^2\big)^{1/2}\\
&\leq \big(\int_{\Omega}\e_{i_j} h_{\e_{i_j}}^2 \phi\big)^{1/2}\big(2 \|\phi-\tilde\phi_k\|_{C^2}\big)^{1/2}
(2\mu_{t}^{\e_{i_j}}(\Omega))^{1/2}\to 0
\end{split}
\label{brakke11}
\end{equation}
as $k\to \infty$ uniformly in $j$. The error of replacing $\tilde\phi=\tilde\phi_k$ in \eqref{brakke10} by 
$\phi$ can be approximated 
similarly. Thus \eqref{brakke10} holds
also for $\phi$ instead of $\tilde\phi$. Recall that we have taken a subsequence so that \eqref{brakke8.5}
holds. Combined with \eqref{Hineq} and \eqref{brakke10} with $\tilde\phi=\phi$, and recalling
that $h\cdot\tilde{u}=h\cdot(\tilde{u}\cdot\nu)\nu$ for $\mu_t$ a.e$.$ by Brakke's perpendicularity theorem \cite[Ch. 5]{brakke},
we have proved 
\eqref{brakke8}. This concludes the proof.
\end{proof}
We next discuss the proof of Theorem \ref{existence} (2).  
\begin{pro}
There exists a further subsequence (denoted by the same index) $\{ \varphi_{\e_i} \}_{i=1}^{\infty}$ and a function $\varphi\in BV_{loc}(\Omega\times [0,\infty))\cap C^{\frac12}_{loc}([0,\infty);L^1(\Omega))$ such that for all $t\geq 0$, 
\begin{equation}
w_{\e_i}(\cdot,t)\to \varphi(\cdot,t)
\label{w1sup}
\end{equation}
strongly in $L^1_{loc}(\Omega)$ and $\varphi$ satisfies the properties of Theorem \ref{existence} (2). 
Here $w_{\e_i}$ is defined by
\[ w_{\e_i} := \Phi \circ \varphi_{\e_i}\mbox{ with } \Phi (s): = \sigma^{-1}\int _{-1} ^s \sqrt{2W(y)} \, dy. \]
\end{pro}
\begin{proof}
Note that $\Phi (1)=1$ and $\Phi(-1)=0$. We compute
\[ |\nabla w_{\e_i}|=\sigma^{-1} |\nabla \varphi_{\e_i} | \sqrt{2W(\varphi_{\e_i})} 
\leq \sigma^{-1} \Big( \frac{\e _i |\nabla \varphi _{\e_i} |^2}{2} +\frac{W(\varphi_{\e_i})}{\e _i} \Big). \]
Fix $T>0$. For all sufficiently large $i$, by \eqref{denconcl} we have
\begin{equation}
\int _\Omega |\nabla w_{\e_i}(\cdot ,t )| \, dx
\leq \int _\Omega \sigma^{-1} \Big( \frac{\e _i |\nabla \varphi _{\e_i} |^2}{2} +\frac{W(\varphi_{\e_i} )}{\e _i} \Big) \, dx \leq \sigma^{-1}D_1
\label{w1}
\end{equation}
for any $t\in[0,T]$. By the similar argument we have
\begin{equation}
\begin{split}
&\int _0 ^T \int _\Omega |\partial _t w_{\e_i} | \, dxdt \leq \sigma ^{-1} 
\int _0 ^T \int _\Omega \Big( \frac{\e _i |\partial _t \varphi_{\e_i} |^2}{2} +\frac{W(\varphi_{\e_i} )}{\e _i} \Big) \, dxdt \\
\leq & \sigma ^{-1} \int _0 ^T \int _\Omega \e _i \Big\{
(u_{\e_i}\cdot \nabla \varphi_{\e_i} )^2 +\Big( \Delta \varphi_{\e_i} -\frac{W'(\varphi_{\e_i})}{\e _i} \Big) ^2
 \Big\} \, dxdt
 + \sigma ^{-1} \int _0 ^T \int _\Omega  \frac{W(\varphi_{\e_i} )}{\e _i} \, dxdt,
\end{split}
\label{w2}
\end{equation}
and the last quantity is uniformly bounded due to Lemma \ref{erlem3}. By \eqref{w1} and \eqref{w2} $\{w_{\e_i}\} _{i= 1}^{\infty}$ is bounded in $BV_{loc}(\Omega \times [0,T])$. By the standard compactness theorem
and a diagonal argument, there exists a subsequence (denoted by the same index) $\{w_{\e_i}\} _{i=1}^{\infty}$ and $w \in BV_{loc}(\Omega \times [0,\infty))$ such that 
\begin{equation}
w_{\e_i} \to w \qquad \text{strongly in }L^1_{loc}(\Omega \times [0,\infty)) 
\label{stwa1}
\end{equation}
and a.e. pointwise. 
 We set $\varphi :=(1+\Phi ^{-1} \circ w)/2$. We have
\[ \varphi_{\e_i} \to 2\varphi-1 \qquad \text{a.e. in }\Omega \times[0,\infty)\]
and by this with $|\varphi_{\e_i}|\leq 1$ we obtain
\[ \varphi _{\e_i} \to 2\varphi-1 \qquad \text{in }L^1_{loc}( \Omega \times[0,\infty)). \]
Due to the uniform bound on $\int_{\Omega}\frac{W(\varphi_{\e_i})}{\e_i}\, dx$, one
can prove by Fatou's lemma that $\varphi_{\e_i}\to \pm 1$ for a.e$.$ $(x,t)$
and hence $\varphi =1$ or $=0$ a.e. on $\Omega \times [0,\infty)$. In particular,
since $\varphi=1\iff w=1$ and $\varphi=0\iff w=0$, 
we have $w=\varphi$ on $\Omega\times [0,\infty)$. This in particular proves the $BV_{loc}(\Omega
\times[0,\infty))$
property of $\varphi$. 
For a.e. $0\leq t_1 < t_2 \leq T$ and any open set $U\subset\subset \Omega$, we have
\begin{equation*}
\begin{split}
\int _{U} |\varphi (\cdot ,t_1 )-\varphi(\cdot, t_2 )| \, dx 
&= \lim _{i\to \infty} \int _{U} |w_{\e_i}(\cdot ,t_1 )- w_{\e_i}(\cdot ,t_2)| \, dx
\leq \liminf _{i\to \infty } \int _{U} \int _{t_1} ^{t_2} |\partial _t w_{\e_i}| \, dtdx \\
& \leq  \liminf_{i\to \infty} \sigma ^{-1} \int _\Omega \int _{t_1} ^{t_2} 
\Big( \frac{\e _i |\partial _t \varphi_{\e_i} |^2}{2} \sqrt{t_2-t} +\frac{W(\varphi_{\e_i})}{\e _i \sqrt{t_2-t}} \Big) \, dtdx.
\end{split}
\end{equation*}
Note that the right-hand side does not depend on $U$. 
Thus, by the similar argument to \eqref{w2} we have with $c=c(\Cr{c-1},n,p,q,D_0,T,W)$
\begin{equation}
\int _\Omega |\varphi(\cdot , t_1)-\varphi(\cdot ,t_2)| \, dx \leq c \sqrt{t_2-t_1} .
\label{conphi}
\end{equation}
Since $(1+\varphi_{\e_i}(\cdot,0))/2\to \chi_{\Omega_0}$ by \eqref{coninit},  
we have (2c). We assumed that $\Omega_0$ is a
bounded domain, hence, \eqref{conphi} shows that $\varphi(\cdot,t)\in L^1(\Omega)$
for a.e$.$ $t\geq 0$. 
Moreover, we may define $\varphi(\cdot,t)$ as  a characteristic function for all $t\geq 0$ so 
that $\varphi\in C^{\frac{1}{2}}_{loc} ([0,\infty) ;L^1 (\Omega))$ due to \eqref{conphi}. 
This proves (2a) and $C^{\frac12}_{loc}$ property for $\varphi$. 
From \eqref{stwa1}, for a.e$.$ $t\geq 0$, $w_{\e_i}(\cdot,t)\to \varphi(\cdot,t)$ in $L^1_{loc}(\Omega)$
strongly. Using \eqref{conphi}, one can show by a simple telescopic argument that the 
convergence is true for all $t\geq 0$ instead of a.e$.$ $t$, which proves \eqref{w1sup}. 
By the standard lower
semicontinuity property of $BV$ norm,
for any $\phi \in C_c(\Omega;{\mathbb R}^+)$ and $0\leq t<\infty$, we have
\begin{equation*}
\begin{split}
&\int _\Omega \phi \, d\|\nabla \varphi(\cdot,t)\|
 \leq \liminf _{i\to \infty} \int _\Omega \phi |\nabla w_{\e_i} | \, dx \\
\leq & \lim _{i \to \infty} \sigma^{-1} \int _\Omega 
\Big( \frac{\e _i |\nabla \varphi_{\e_i} |^2}{2}  +\frac{W(\varphi_{\e_i})}{\e _i} \Big) \phi \, dx =
\int _\Omega \phi \, d\|V_t\|.
\end{split}
\end{equation*}
This proves (2b). 

To prove (2d), we consider the a.e$.$ $t\geq 0$ for which we have proved the integrality
of $V_t$. Writing $\|V_t\|=\theta{\mathcal H}^{n-1}\lfloor_{M_t}$, we already know that $\theta$ is 
integer-valued $\|V_t\|$ a.e$.$ and that $M_t$ is countably $(n-1)$-rectifiable. In addition, by \eqref{exden},
we have $1\leq \theta\leq N(t)$, ${\mathcal H}^{n-1}$ a.e$.$ on $M_t$ for some integer $N(t)$. The
latter shows in particular that 
\begin{equation}
{\mathcal H}^{n-1}\lfloor_{M_t}\leq \|V_t\|\leq N(t){\mathcal H}^{n-1}\lfloor_{M_t}.
\label{Vh}
\end{equation}
By (2a) and (2b), 
we know that $\|\nabla\varphi(\cdot,t)\|={\mathcal H}^{n-1}\lfloor_{\tilde M_t}$ for some countably
$(n-1)$-rectifiable set by De Giorgi's theorem (see \cite[4.4]{Giusti}). To prove \eqref{parity0}, assume 
the contrary. 
Then by the standard argument (see \cite[3.5]{simon}), 
there would be a point $x\in \tilde M_t\setminus M_t$ with
$\lim_{r\downarrow 0}{\mathcal H}^{n-1}(B_r(x)\cap \tilde M_t)/\omega_{n-1}r^{n-1}=1$ while 
$\lim_{r\downarrow 0}{\mathcal H}^{n-1}(B_r(x)\cap M_t)/\omega_{n-1}r^{n-1}=0$. Then, using 
also \eqref{Vh}, one would then have a contradiction 
to Theorem \ref{existence} (2b). Thus we have \eqref{parity0}. 

To prove \eqref{parity}, we closely follow the proof of integrality again. We already know that for $\|V_t\|$
a.e$.$ $x$, we have the properties
described in the proof of Theorem \ref{integralitymu}. By the well-known property of set of finite perimeter (\cite[3.8]{Giusti}), 
for ${\mathcal H}^{n-1}$ a.e$.$ $x\in \tilde M_t$, the blow-up limit of $\varphi$ centered at $x$ is supported 
by a half-space. For ${\mathcal H}^{n-1}$ a.e$.$ $x\in \Omega\setminus \tilde M_t$ 
(in particular on $M_t\setminus \tilde M_t$), the blow-up
limit centered at $x$ is a constant function with value either $0$ or $1$. By \eqref{w1sup}, up to ${\mathcal H}^{n-1}$ null set,
we may assume in addition to the properties of $\{V_t^{\e_i}\}_{i=1}^{\infty}$ in the proof of Theorem \ref{integralitymu} that
$\tilde w_{\e_i}(\tilde x):=w_{\e_i}(r_i \tilde x)$ converges strongly in $L^1_{loc}({\mathbb R}^n)$ and pointwise
${\mathcal L}^{n}$ a.e$.$ to $\chi_{\{x_n\geq 0\}}$ (or
$\chi_{\{x_n\leq 0\}}$)
if $x=0$ is in $\tilde M_t$, or to 1 (or 0) if $x=0$ is in $M_t\setminus \tilde M_t$. Since the proof for other cases
is similar, we only discuss the case of $\tilde M_t$ and $\lim_{i\to\infty} \tilde w_{\e_i}=\chi_{\{x_n\geq 0\}}$ in the following. 
In terms of $\varphi_{\e_i}$ (which is the relabeling of $\tilde\varphi_{\e_i}$), note that this means that $\varphi_{\e_i}$ converges a.e$.$ to $\chi_{\{x_n\geq 0\}}-\chi_{\{x_n<0\}}$. 

As one follows the proof of Theorem \ref{integralitymu}, the difference occurs at \eqref{intf0}, where we already
know that $\theta$ is an integer multiple of $\sigma$. So let $N-1:=\sigma^{-1} \theta(\geq 1)$. We want to conclude that
$N$ is an even integer. We follow the proof until \eqref{intf12}, and at this point, define
for $i\in {\mathbb N}$ (and writing $Y(\tau,x):=\{\varphi_{\e_i}=\tau\}\cap B_1\cap G_i\cap P^{-1}(x)$)
\begin{equation}
\begin{split}
& \tilde A_i := \{x\in B_1^{n-1}\,:\, \forall \tau\in (-1+b,1-b)\Rightarrow {\mathcal H}^0 (Y(\tau,x))
\leq N-2\},\\
& A_i:=\{x\in B_1^{n-1}\,:\, \exists \tau\in (-1+b,1-b)\Rightarrow {\mathcal H}^0 (Y(\tau,x))
= N-1\}.
\end{split}
\label{ints0}
\end{equation}
We know from \eqref{intf8} that ${\mathcal H}^0(Y(\tau,x))$ has to be $\leq N-1$, thus, $B_1^{n-1}
=\tilde A_i\cup A_i$ and 
\begin{equation}
{\mathcal H}^{n-1}(\tilde A_i)=\omega_{n-1}-{\mathcal H}^{n-1}(A_i)
\label{ints01}
\end{equation}
for all sufficiently large $i$. In \eqref{intf13}, we have 
\begin{equation}
\begin{split}
\omega_{n-1}\sigma (N-1)  &\leq 2s+\liminf_{i\to\infty} \int_{-1+b}^{1-b} \sqrt{2W(\tau)} \{(N-2) {\mathcal H}^{n-1}
(\tilde A_i)+(N-1){\mathcal H}^{n-1}(A_i)\}\, d\tau \\
& \leq 2s+(N-2)\sigma \omega_{n-1}+\sigma \liminf_{i\to\infty} {\mathcal H}^{n-1}(A_i)
\end{split}
\label{ints1}
\end{equation}
where we used \eqref{ints01}. Thus we have from \eqref{ints1}
\begin{equation}
\omega_{n-1}-2\sigma^{-1} s\leq \liminf_{i\to\infty} {\mathcal H}^{n-1}(A_i).
\label{ints2}
\end{equation}
By \eqref{clo3}, for all sufficiently large $i$ and any point $x\in A_i$, the image 
$\varphi_{\e_i}(B_1\cap P^{-1}(x))$ covers $[-1+b,1-b]$ 
at least $N-1$ times. The each covering is monotone, thus we know that
$\varphi_{\e_i}(y)$ as $y$ moves from $P^{-1}(x)\cap \{x_n=-s\}$  to 
$P^{-1}(x)\cap \{x_n=s\}$ along $P^{-1}(x)$ has to go up and down
between $-1+b$ and $1-b$ at least $N-1$ times.
Next, since $\varphi_{\e_i}$ converges a.e$.$ pointwise to $\chi_{\{x_n\geq 0\}}-
\chi_{\{x_n<0\}}$, by Egoroff's Theorem and then Fubini's Theorem, there
exists $s_1\in [s,2s]$, $s_2\in [-2s,-s]$, $C_1\subset B_1^{n-1}$
and $C_2\subset B_1^{n-1}$ such that $\varphi_{\e_i}$ converges
uniformly to 1 on $C_1\times\{s_1\}$ and to $-1$ on $C_2\times\{s_2\}$ while
\begin{equation}
{\mathcal H}^{n-1}(C_i)\geq \omega_{n-1} -s\ \ \mbox{for $i=1,2$}.
\label{ints3}
\end{equation}
Set $C_3=C_1\cap C_2$ so that, by \eqref{ints3}, 
\begin{equation}
{\mathcal H}^{n-1}(C_3)\geq \omega_{n-1}-2s.
\label{ints4}
\end{equation}
Now, for a contradiction, assume that $N$ is odd.
For $x\in A_i\cap C_3$, consider the image of $\varphi_{\e_i}$ on $\{(x,x_n)\,:\, x_n\in [s_2,s_1]\}$.
By the uniform convergence and $x\in C_3$, for sufficiently large $i$, $\varphi_{\e_i}(x,s_2)<-1+b$
and $\varphi_{\e_i}(x,s_1)>1-b$. Since $\varphi_{\e_i}$ is continuous, image of $\varphi_{\e_i}$
having at least even $N-1$ covering of $[-1+b,1-b]$ implies that there has to be at least another covering
of $[-1+b,1-b]$. Thus, for each $\tau\in [-1+b,1-b]$ and $x\in A_i\cap C_3$, we have 
\begin{equation}
{\mathcal H}^{0}(\{x_n\in [s_2,s_1]\,:\,
\varphi_{\e_i}(x,x_n)=\tau\})\geq N.
\label{ints5}
\end{equation}
Then by the coarea formula and \eqref{ints5}, we have
\begin{equation}
\begin{split}
&\int_{s_2}^{s_1} \sqrt{2W(\varphi_{\e_i}(x,x_n))}|\partial_{x_n}\varphi_{\e_i}(x,x_n)|\, dx_n \\
&=\int_{-1}^1 \sqrt{2W(\tau)} {\mathcal H}^0 (\{x_n\in[s_2,s_1]\,:\,\varphi_{\e_i}(x,x_n)=\tau\})\, d\tau  
\geq N \int_{-1+b}^{1-b} \sqrt{2W(\tau)}\, d\tau.
\end{split}
\label{ints6}
\end{equation}
Note that by \eqref{ints2} and \eqref{ints4}, we have for sufficiently large $i$
\begin{equation}
{\mathcal H}^{n-1}(A_i\cap C_3)\geq \omega_{n-1}-(3+2\sigma^{-1})s.
\label{ints7}
\end{equation}
Integrating \eqref{ints6} over $A_i\cap C_3$ and \eqref{ints7} give
\begin{equation}
\begin{split}
\int_{B_1}\sqrt{2W(\varphi_{\e_i})}|\nabla\varphi_{\e_i}| & \geq 
\int_{(A_i\cap C_3)\times [s_2,s_1]}\sqrt{2W(\varphi_{\e_i})}|\partial_{x_n}\varphi_{\e_i}| \\
&\geq (\omega_{n-1}-(3+2\sigma^{-1})s) N \int_{-1+b}^{1-b}\sqrt{2W(\tau)}\,d\tau.
\end{split}
\label{ints8}
\end{equation}
We may choose $b$ so that $\int_{-1+b}^{1-b}\sqrt{2W(\tau)}\, d\tau\geq 
\sigma-s$. On the other hand, by \eqref{inqr8.1}, we have
\begin{equation}
\int_{B_1}\sqrt{2W(\varphi_{\e_i})}|\nabla\varphi_{\e_i}|\, dx\leq \int_{B_1}
\frac{\e_i |\nabla\varphi_{\e_i}|^2}{2}+\frac{W}{\e_i}\, dx\to \omega_{n-1}(N-1)\sigma.
\label{ints9}
\end{equation}
For sufficiently small $s$ depending only on $n,N$ and $\sigma$, \eqref{ints8}
and \eqref{ints9} lead to a contradiction. This proves $N$ has to be even. 
As we mentioned, other cases of $\varphi$ being constant (either 0 or 1) can be
similarly proved. This concludes the proof of \eqref{parity} and (2d).
\end{proof}
We next verify
\begin{pro}
The function $u$ satisfies the property of Theorem \ref{existence} (3).
\end{pro}
\begin{proof}
Consider the case $p<n$ and fix $T>0$. Since 
$\lim_{i\to\infty}\|u_{\e_i}-u\|_{L^q ([0,T];(W^{1,p})^n )}=0$, $\{u_{\e_i}\}$ is a Cauchy sequence in this norm. By 
\eqref{mzineq} with $s=\frac{p(n-1)}{n-p}$, we have
\begin{equation}
\int_0^T\, dt\big(\int_{\Omega} |u_{\e_i}-u_{\e_j}|^{s}\, d\|V_t\|\big)^{\frac{q}{s}}
\leq c(n,p,q,D_1) \|u_{\e_i}-u_{\e_j}\|_{L^q([0,T];(W^{1,p}(\Omega))^n)}^q.
\label{uest1}
\end{equation}
By a standard argument, we may subtract a subsequence $\{u_{\e_{i_j}}\}_{j=1}^{\infty}$ which
converges pointwise $\|V_t\|\times dt$ a.e$.$ on $\Omega\times[0,T]$ to an element of 
$L^q([0,T];(L^s (\|V_t\|))^n)$. This limit function is uniquely determined by $u$ independent
of the approximate sequence and
\eqref{intu2} holds. For $p=n$, we apply the same argument locally
for $p'<n$ which gives \eqref{intu2} with any $2\leq s<\infty$. For $p>n$, the standard Sobolev 
inequality and the H\"{o}lder inequality prove the claim immediately.
\end{proof}
To conclude the proof of Theorem \ref{existence} we prove
\begin{pro}
We have $T_1>0$ with the property described in Theorem \ref{existence} (4).
\end{pro}
\begin{proof}
By integrality, we already know that $\|V_t\|=\theta{\mathcal H}^{n-1}\lfloor_{M_t}$ for a.e$.$ $t\geq 0$,
where $\theta$ is integer-valued ${\mathcal H}^{n-1}$ a.e$.$ on $M_t$. Thus we should prove that 
$\mathcal{H}^{n-1} (\{\theta(\cdot,t) \geq 2 \})=0$ for a.e. $0< t< T_1$ for some $T_1>0$. We will determine
the lower bound of $T_1$ in the following.
Assume there exist $0<\hat t<T_1$ and $\hat x \in M_{\hat t}$ 
such that $M_{\hat t}$ has the approximate tangent space at $\hat x$ and the density 
$\theta(\hat x,\hat t)\geq 2$.
Then it is not difficult to check that
\begin{equation}
\lim _{r\to 0} \int _\Omega \tilde \rho _{(\hat x,\hat t+r^2)} \, d\|V_{\hat t}\|=\theta(\hat x,\hat t)\geq 2. 
\label{dwro}
\end{equation}
Since $\|V_0\|={\mathcal H}^{n-1}\lfloor_{M_0}$ and $M_0$ is $C^1$, we have 
\begin{equation}
\int _\Omega \tilde \rho _{(x,t)} \,d\|V_0\| \leq 3/2
\label{dwr2}
\end{equation}
for any $(x,t)\in \Omega\times (0,T_1]$, where $T_1$ depends only on 
$M_0$. We then use 
\eqref{longeq} with $\e\to 0$. We then have 
\begin{equation}
\lim_{r\to 0} \int_{\Omega} \tilde\rho_{(\hat x,\hat t+r^2)}\, d\|V_t\|\Big|_{t=0}^{\hat t}\leq 
\Cr{c-12}\Cr{c-1}^2 {\hat t}^{\hat p}D_1+\Cr{c-2} e^{-\frac{1}{128\hat t}}\hat t D_1,
\label{dwr1}
\end{equation}
and the right hand side of \eqref{dwr1} may be made smaller than $1/2$ by restricting $T_1$. 
Then we would have a contradiction since the left-hand side is $\geq 1/2$ due to \eqref{dwro}
and \eqref{dwr2}. This proves the first part of (4). 
We next prove $\|\nabla\varphi(\cdot,t)\| =\|V_t\|$ a.e$.$ $t\in [0,T_1]$. 
With the notation of (2d), for a.e$.$ $t\in [0,T_1]$, we have $\|V_t\|={\mathcal H}^{n-1}\lfloor_{M_t}$
since $\theta=1$ a.e$.$ from the first part. But then, by \eqref{parity}, ${\mathcal H}^{n-1}(M_t\setminus
\tilde M_t)=0$ since $\theta=1$ and odd. Thus combined with \eqref{parity0}, $\tilde M_t=M_t$ modulo null set, and this
shows the claim. We may take $T_1$ to be $\sup\{t>0\,:\, V_t \mbox{ is unit density for a.e$.$ $t\in [0,t]$}\}$.
\end{proof}
As for the proof of Theorem \ref{regreg}, (1) and (3) follow from \cite{kasaitonegawa} and \cite{tonegawa2012}, 
respectively, which give criterion for partial $C^{1,\zeta}$ and $C^{2,\alpha}$ regularity.  For (1), we check that
 \cite[Sec. 3.1 (A1)-(A4)]{kasaitonegawa} are all satisfied. Namely, (A1) asks $V_t$ to be unit density
for a.e$.$ $t$, (A2) is on the uniform density ratio upper
bound which follows from \eqref{exden}, (A3) is on the integrability of $u$ which is given by \eqref{intu2} and
(A4) is the flow equation which is \eqref{mcfweak}. If $p<n$, the exponent of integrability of $u$ in \eqref{intu2}
has to satisfy $\zeta:=1-(n-1)/s-2/q=2-n/p-2/q>0$, and this follows from \eqref{pqncon}. If $p\geq n$, 
we may choose any $s>(n-1)q/(q-2)$ in \eqref{intu2} so that we have $0<\zeta$, and we may
take sufficiently large $s$ so that $0<\zeta<1-2/q$ can be arbitrarily close to $1-2/q$. This proves (1).
The conclusion for $C^{2,\alpha}$ is precisely the claim of \cite{tonegawa2012}. Thus we only 
need to prove (2) and (4).
\begin{pro}
The family of varifolds $\{V_t\}_{t\geq 0}$ satisfies the property of Theorem \ref{regreg} (2) and (4)
\end{pro}
\begin{proof}
For a.e$.$ $0\leq t<T_1$, we have proved that $V_t$ has unit density property, 
thus we may use results in \cite{kasaitonegawa} for $\{V_t\}_{0\leq t<T_1}$. 
We first claim that there exists $0<T_3\leq T_1$ depending only on $D_1,n,p,q,\|u\|_{L^q([0,T_1];
(W^{1,p}(\Omega))^n)}$ ($D_1$ corresponding to $T_1$) and $\Cl[c]{c-26}=\Cr{c-26}(D_1,n)$ such that
\begin{equation}
{\rm dist}\,({\rm spt}\,\|V_t\|,M_0)\leq\Cr{c-26} \sqrt{t}
\label{indis}
\end{equation}
for a.e$.$ $0\leq t\leq T_3$.  For the proof, we use \cite[Proposition 6.2]{kasaitonegawa}. Citing the result for the convenience
of the reader, we have for $x\in \Omega$ and $0<r<1$
\begin{equation}
\int_{B_r(x)} \hat\rho_{(x,t+\epsilon)}(\cdot,t)\, d\|V_t\|-\int_{B_r(x)}\hat \rho_{(x,t+\epsilon)}(\cdot,0)\, d\|V_0\|
\leq c(n,s,q) \|u\|^{2}_{L^{s,q}} D_1^{1-\frac{2}{s}}t^{\zeta}+c(n)D_1 r^{-2} t,
\label{indis0}
\end{equation}
where $s:=\frac{p(n-1)}{n-p}$ if $p<n$ and any $\frac{(n-1)q}{q-2}< s<\infty$ if 
$p\geq n$, $\zeta=1-(n-1)/s-2/q$ and 
$\|u\|_{L^{s,q}}:=(\int_0^t (\int_{B_r(x)} |u|^{s}\, d\|V_{\lambda}\|)^{q/s}\,d\lambda)^{1/q}$. $\hat\rho_{(x,t+\epsilon)}$
is $\rho_{(x,t+\epsilon)}$ times a radially symmetric cut-off function with support in $B_{14r/15}(x)$ and $=1$ near $x$. 
Note that $\|u\|_{L^{s,q}}$ may be bounded in terms of $D_1$ and $\|u\|_{L^q([0,T_1];(W^{1,p}(\Omega))^n)}$
as was done for the proof of \eqref{intu2}. 
By restricting $T_3$ small, we may conclude from \eqref{indis0} that
\begin{equation}
\int_{B_r(x)}\hat{\rho}_{(x,t+\epsilon)}(\cdot,t)\, d\|V_{t}\|-\int_{B_r(x)} \hat{\rho}_{(x,t+\epsilon)}(\cdot, 0)\, d\|V_0\|
\leq \frac12 +c(n)D_1 r^{-2}t .
\label{indis2}
\end{equation}
Let $\Cr{c-26}$ be a constant to 
be fixed shortly and assume that there exists $x\in {\rm spt}\, \|V_t\|$ such that ${\rm dist}\,(x,M_0)
> \Cr{c-26}\sqrt{t}$ and $0<t\leq T_3$. We may assume that $V_t$ is unit density and has approximate
tangent space with multiplicity 1 at $x$, since such time and point are generic. In particular, one can
check that $
\lim_{\epsilon\to 0+}\int_{B_r(x)}\hat{\rho}_{(x,t+\epsilon)}(\cdot,t)\, d\|V_t\|=1$ and \eqref{indis2} thus shows
\begin{equation}
\frac12-\int_{B_r(x)} \hat{\rho}_{(x,t)}(\cdot, 0)\, d\|V_0\|
\leq c(n)D_1 r^{-2}t .
\label{indis3}
\end{equation}
We now choose $r=\Cr{c-26}\sqrt{t}/2$. Since $B_r(x)\cap M_0=\emptyset$, the integral in \eqref{indis3} is 0.
Hence we obtain $\frac12\leq  4 c (n)D_1 \Cr{c-26}^{-2}$. If we choose a sufficiently large $\Cr{c-26}$ depending
only on $n$ and $D_1$, we obtain a contradiction. This proves \eqref{indis}. 

Next, since ${\rm spt}\,\|\nabla\varphi(\cdot,t)\|\subset{\rm spt}\,\|V_t\|$ by Theorem \ref{existence} (2b), 
\eqref{indis} shows that $\varphi(\cdot,t)$
is a constant function on each connected component of $\Omega\setminus \{x\,:\, {\rm dist}\,(x,M_0)
\leq \Cr{c-26} \sqrt{t}\}$ for a.e$.$ $0\leq t\leq T_3$. Since $\varphi(\cdot,t)$ is a characteristic function 
and is continuous in $L^1$ norm with respect to time, one sees that 
\begin{equation}
\begin{split}
&\varphi(\cdot,t)=1\ \ \mbox{ on } \{x\in \Omega_0\,:\, {\rm dist}\,(x,M_0)>\Cr{c-26}\sqrt{t}\},\\
&\varphi(\cdot,t)=0 \ \ \mbox{ on } \{x\notin \Omega_0\, :\, {\rm dist}\, (x,M_0)>\Cr{c-26}\sqrt{t}\}
\end{split}
\label{indis4}
\end{equation}
for all $0\leq t\leq T_3$. We now estimate the location of ${\rm spt}\,\|V_t\|$
during the short initial time. Since $M_0$ is assumed to be $C^1$, there exists $r_1>0$ 
such that, for each $x\in M_0$ 
(we may assume that $x$ is the origin and
$T_x M_0={\mathbb R}^{n-1}\times\{0\}$ after parallel translation and orthogonal rotation),
$M_0$ is locally represented as a $C^1$ graph $g:B_{r_1}^{n-1}\to {\mathbb R}$ on $B_{r_1}^{n-1}\times
(-r_1,r_1)$. We take the coordinate system so that $\Omega_0$ is located on the upper side, above the graph of $g$. 
We may also restrict $r_1$ (uniformly on $M_0$) so that for all $r\leq r_1$, we have
\begin{equation}
\sup_{x\in B_{r}^{n-1}} |g(x)|\leq \frac{r}{10}.
\label{indis4.5}
\end{equation}
For $t\in [0,(10\Cr{c-26})^{-2} r^2]$, \eqref{indis4} and \eqref{indis4.5} show that
\begin{equation}
\begin{split}
&\varphi(\cdot,t)=1\ \ \mbox{ on } B_{9r/10}^{n-1}\times [r/5,r_1),\\
&\varphi(\cdot,t)=0 \ \ \mbox{ on } B_{9r/10}^{n-1}\times (-r_1,-r/5].
\end{split}
\label{indis4.6}
\end{equation}

Next we use \cite[Theorem 8.7]{kasaitonegawa}. Using the notation there, 
corresponding to $1\leq E_1<\infty$, $0<\nu<1$,
$p,q$ with $1-(n-1)/p-2/q>0$, there exist 4 constants ($\varepsilon_6,\sigma,\Lambda_3,c_{19}$
in \cite{kasaitonegawa}) with the stated properties. 
Here, we use $E_1=D_1$, $\nu=1/2$, $p=s$ above and the same $q$. 
The condition $1-(n-1)/p-2/q>0$ is then satisfied. To avoid confusion in the following, we denote the constants in
\cite{kasaitonegawa} corresponding to these choices by $\varepsilon_{6,KT}, \sigma_{KT},
\Lambda_{3,KT},c_{19,KT}$. 
In the following, let $P\in {\bf G}(n,n-1)$ be the projection ${\mathbb R}^n\to {\mathbb R}^{n-1}\times\{0\}$
and $P^{\perp}$ be its orthogonal complement. 
We then use \cite[Proposition 6.5]{kasaitonegawa} 
with 
\begin{equation}
\Lambda= \Lambda_{3,KT}/18
\label{indis5.1}
\end{equation}
 to obtain $c_{6,KT}$ with the property that
\begin{equation}
\begin{split}
\frac{1}{r^{n+1}} \int_{B_r} & |P^{\perp}(x)|^2\, d\|V_t\|  \leq \exp(1/(4\Lambda)) \frac{1}{r^{n+1}}\int_{B_{Lr}} |P^{\perp}(x)|^2\, d\|V_0\|\\
&+c_{6,KT}\{(r^{2\zeta} \|u\|_{L^{s,q}}^2+r^{\zeta}\|u\|_{L^{s,q}})L^2+L^{n+1}\exp(-(L-1)^2/(8\Lambda))\}
\end{split}
\label{indis5}
\end{equation}
for all $t\in [0,\Lambda r^2]$ provided $2\leq L<\infty$ and $rL\leq 1$. Here $c_{6,KT}$ depends only on
$s,q, D_1, \Lambda_{3,KT}$ but not on $L$. Given $1>\varepsilon>0$, we may choose $L\geq 2$ so that
\begin{equation}
c_{6,KT}L^{n+1}\exp(-(L-1)^2/(8\Lambda))<\varepsilon
\label{indis6}
\end{equation}
and then choose $r_2\leq L^{-1}$ uniformly on $M_0$ so that (using $M_0$ is $C^1$)
\begin{equation}
\begin{split}
&\exp(1/(4\Lambda)) \sup_{0<r\leq r_2} \frac{1}{r^{n+1}}
\int_{B_{Lr}} |P^{\perp}(x)|^2\, d\|V_0\|< \varepsilon, \\  &c_{6,KT}(r_2^{2\zeta}\|u\|_{L^{s,q}}^2+r_2^{\zeta}\|u\|_{L^{s,q}})L^2<\varepsilon.
\end{split}
\label{indis7}
\end{equation}
The inequalities \eqref{indis5}-\eqref{indis7} gives for $r\leq r_2$ and $t\in [0,\Lambda r^2]$
\begin{equation}
\frac{1}{r^{n+1}}\int_{B_r}|P^{\perp}(x)|^2\, d\|V_t\|\leq 3\varepsilon.
\label{indis8}
\end{equation}
We next use \cite[Proposition 6.4]{kasaitonegawa} on $B_r\times [0,\Lambda r^2]$ with a slight
modification. Instead of obtaining result on the time interval $[R^2/5,\Lambda]$ as in \cite{kasaitonegawa},
we modify the proof so that we obtain the similar estimate on the time interval $[(10 \Cr{c-26})^{-2} r^2,
\Lambda r^2]$. This is achieved by a simple replacement of the cut-off function. We have a different
constants which depends also on $\Cr{c-26}$. Citing the result from 
\cite[Proposition 6.4]{kasaitonegawa}, we obtain
\begin{equation}
{\rm spt}\, \|V_t\|\cap B_{4r/5}\subset \{|P^{\perp}(x)|\leq \mu r\}\ \ \mbox{for $t\in [(10\Cr{c-26})^{-2}r^2,\Lambda r^2]$},
\label{indis10}
\end{equation}
where 
\begin{equation}
\mu^2:= \frac{c_{5,KT}}{r^{n+3}}\int_0^{\Lambda r^2}\int_{B_r} |P^{\perp}(x)|^2\, d\|V_t\|dt
+c_{2,KT} \|u\|_{L^{s,q}}^2D_1^{1-\frac{2}{s}} \Lambda^{\zeta} r^{2\zeta} (2+\Lambda)
\label{indis11}
\end{equation}
and where $c_{5,KT}$ and $c_{2,KT}$ depend only on $n,s,q$ and $\Cr{c-26}$. 
If we restrict $r_2$ further so that the second term of \eqref{indis11} is smaller than $\varepsilon$, \eqref{indis8}-\eqref{indis11}
with sufficiently small $\varepsilon$ gives
\begin{equation}
{\rm spt}\,\|V_t\|\cap B_{4r/5}\subset \{|P^{\perp}(x)|\leq r/5\}\ \ \mbox{ for } t\in [(10\Cr{c-26})^{-2}r^2,\Lambda r^2].
\label{indis12}
\end{equation}
Combining \eqref{indis4.6} and \eqref{indis12}, and using the $L^1$ continuity of $\varphi(\cdot,t)$, we obtain
\begin{equation}
\begin{split}
& \varphi(\cdot,t)=1 \ \ \mbox{on } B_{4r/5}\cap \{P^{\perp}(x)\geq r/5\},\\
& \varphi(\cdot,t)=0 \ \ \mbox{on } B_{4r/5}\cap \{P^{\perp}(x)\leq -r/5\}
\end{split}
\label{indis13}
\end{equation}
for $t\in [0,\Lambda r^2]$. Since $B_{r/2}^{n-1}\times [-r/2,r/2]\subset B_{4r/5}$,
\eqref{indis13} shows
\begin{equation}
\begin{split}
& \varphi(\cdot, t)=1 \ \ \mbox{on }B_{r/2}^{n-1}\times [r/5,r/2],\\
& \varphi(\cdot, t)=0 \ \ \mbox{on }B_{r/2}^{n-1}\times [-r/2,-r/5],\\
& {\rm spt}\, \|V_t\|\cap (B_{r/2}^{n-1}\times [-r/2,r/2])\subset B_{r/2}^{n-1}\times [-r/5,r/5]
\end{split}
\label{indis14}
\end{equation}
for $t\in [0,\Lambda r^2]$ and $r\leq r_2$. At this point, because of the third claim of \eqref{indis14}, 
by setting $V_t=0$ on $B_{r/2}^{n-1}\times({\mathbb R}\setminus [-r/2,r/2])$,
we may assume that $\{V_t\}_{0\leq t\leq \Lambda r^2}$ satisfies \eqref{mcfweak} on $(B_{r/2}^{n-1}
\times {\mathbb R})\times [0,\Lambda r^2]$. We next want to apply \cite[Theorem 8.7]{kasaitonegawa}
with $R:=r/6$.
For the application, we need to check the conditions (8.83)-(8.86) of \cite{kasaitonegawa}. The first
condition (8.83), the smallness of space-time $L^2$-height may be achieved due to \eqref{indis8}, \eqref{indis14}
and by restricting $\varepsilon$ depending on $\varepsilon_{6,KT}$ and $\Lambda_{3,KT}$. The 
second condition (8.84), the smallness of $\|u\|$, may be achieved by simply restricting $r_2$. Thus
we need to check the last two conditions, (8.85) and (8.86) of \cite{kasaitonegawa}. 
Let $\phi_{P,R}$ and ${\bf c}$ be defined as in \cite[Definition 5.1]{kasaitonegawa}. We need to 
show that (recall that we have set $\nu=1/2$)
\begin{equation}
\exists t_1\in (3R^2 /2,2R^2 )\ \ : \ \ R^{-(n-1)} \|V_{t_1}\|(\phi_{P,R}^2)<\frac32 {\bf c}
\label{indis15}
\end{equation}
and 
\begin{equation}
\exists t_2\in ((2 \Lambda_{3,KT}-2)R^2 ,  (2\Lambda_{3,KT}-3/2)R^2 ) \ \ :\ \ R^{-(n-1)}\|V_{t_2}\|(\phi_{P,R}^2)>\frac12 {\bf c}.
\label{indis16}
\end{equation}
First we show \eqref{indis15}. Since $M_0$ is $C^1$, we may restrict $r_2$ uniformly in $x$ 
so that for all $R=r/6\leq r_2/6$, we have
\begin{equation}
R^{-(n-1)}\|V_0\|(\phi_{P,R}^2)\leq R^{-(n-1)} \int_{P} \phi_{P,R}^2\, d{\mathcal H}^{n-1}+\frac{1}{10} {\bf c}
=\frac{11}{10}{\bf c}.
\label{indis17}
\end{equation}
By \eqref{mcfweak}, we have for $t_1\in (3R^2/2,2R^2)$
\begin{equation}
\|V_{t}\|(\phi_{P,R}^2)\Big|_{t=0}^{t_1}\leq \int_0^{t_1}\int(-h\phi_{P,R}^2+\nabla\phi_{P,R}^2)\cdot(h+(u\cdot\nu)\nu)\,
d\|V_t\|dt.
\label{indis18}
\end{equation}
By \eqref{hperp1} and \eqref{hperp2}, we may replace $\nabla\phi_{P,R}^2$ by $S^{\perp}(\nabla\phi_{P,R}^2)$
for $\|V_t\|$ a.e$.$, where $S$ is the approximate tangent space at the point. Since $\nabla \phi_{P,R}
=P(\nabla\phi_{P,R})$ (note $\phi_{P,R}(x)=\phi_{P,R}(P(x))$ by definition), we have
\begin{equation}
S^{\perp}(\nabla\phi_{P,R}^2)=(I-S)\circ (P(\nabla\phi_{P,R}^2))=(P-S)\circ(P(\nabla\phi_{P,R}^2)).
\label{indis19}
\end{equation}
Thus, by using the Cauchy-Schwarz inequality to \eqref{indis18} and by \eqref{indis19}, we obtain
\begin{equation}
\|V_{t}\|(\phi_{P,R}^2)\Big|_{t=0}^{t_1}\leq\int_0^{t_1} \int -\frac12 |h|^2\phi_{P,R}^2+2|u|^2\phi_{P,R}^2+8
\|S-P\|^2 |\nabla\phi_{P,R}|^2\, dV_t(\cdot,S)dt.
\label{indis20}
\end{equation}
The first term on the right-hand side of \eqref{indis20} can be dropped. The second term can be estimated 
using the H\"{o}lder inequality as
\begin{equation}
\begin{split}
\int_0^{t_1}\int 2|u|^2\phi_{P,R}^2\, d\|V_t\|dt& \leq \int_0^{t_1}\big(\int |u|^s\, d\|V_t\|\big)^{\frac{2}{s}}\,dt
\cdot \sup_{t\in [0,t_1]}\|V_t\|(\phi_{P,R}^2)^{1-\frac{2}{s}} \\
&\leq \|u\|_{L^{s,q}}^2 t_1^{1-\frac{2}{q}}\cdot \sup_{t\in [0,t_1]}\|V_t\|(\phi_{P,R}^2)^{1-\frac{2}{s}}.
\end{split}
\label{indis21}
\end{equation}
Due to the third claim of \eqref{indis14}, ${\rm spt}\, \|V_t\|\cap {\rm spt}\, \phi_{P,R}\subset B_{3R}$, for
example. Thus we have $\|V_t\|(\phi_{P,R}^2)\leq D_1 \omega_{n-1}(3R)^{n-1}$. Since $t_1\leq 2R^2$, 
we obtain from \eqref{indis21}
\begin{equation}
\int_0^{t_1}\int 2|u|^2\phi_{P,R}^2\, d\|V_t\|dt\leq c(D_1,n,s,q)  \|u\|_{L^{s,q}}^2 R^{n-1+2\zeta}.
\label{indis22}
\end{equation}
For the third term of \eqref{indis20}, we use \cite[Lemma 11.2]{kasaitonegawa} (or \cite[8.13]{allard}), 
namely, for $\phi=\phi_{P,R}$
\begin{equation}
\begin{split}
\int \|S-P\|^2\, |\nabla\phi|^2\, dV_t(\cdot,S)&\leq 4\big(\int |h|^2|\nabla\phi|^2\, d\|V_t\|\big)^{\frac12}\big(\int |P^{\perp}(x)|^2
|\nabla\phi|^2\,d\|V_t\|\big)^{\frac12}\\ & +16 \int |P^{\perp}(x)|^2|\nabla|\nabla\phi||^2\, d\|V_t\|.
\end{split}
\label{indis23}
\end{equation}
By repeating a similar argument leading to \eqref{indis20} with slightly larger test function which is 1 on 
${\rm spt}\, \phi_{P,R}$, one can obtain 
\begin{equation}
\int_0^{t_1}\int |h|^2 |\nabla \phi_{P,R}|^2\, d\|V_t\|\leq c(n) R^{n-3}.
\label{indis24}
\end{equation}
Since we have ${\rm spt}\,\|V_t\|\cap {\rm spt}\,\phi_{P,R}\subset B_{3R}$ and
by \eqref{indis8}, we obtain
\begin{equation}
\int_0^{t_1}\int |P^{\perp}(x)|^2 |\nabla\phi_{P,R}|^2\, d\|V_t\|dt\leq 3\varepsilon (3R)^{n+1} t_1 \sup|\nabla\phi_{P,R}|^2
\leq c(n)\varepsilon R^{n+1}.
\label{indis25}
\end{equation}
Thus, by \eqref{indis24} and \eqref{indis25} and similarly estimating the last term, we obtain from \eqref{indis23} that
\begin{equation}
\int_0^{t_1}\int \|S-P\|^2|\nabla\phi_{P,R}|^2\, dV_t(\cdot,S)dt\leq c(n)(\sqrt{\varepsilon}+\varepsilon)R^{n-1}.
\label{indis26}
\end{equation}
Combining \eqref{indis17}, \eqref{indis20}, \eqref{indis22} and \eqref{indis26}, we obtain
\begin{equation}
R^{-(n-1)}\|V_{t_1}\|(\phi_{P,R}^2)\leq \frac{11}{10}{\bf c}+c(D_1,n,s,q)\|u\|_{L^{s,q}}^2 R^{2\zeta}+c(n)(\sqrt{\varepsilon}
+\varepsilon).
\label{indis27}
\end{equation}
Thus, by restricting $r<r_2$ and $\varepsilon$ in \eqref{indis27}, we can guarantee that \eqref{indis15} holds. 
To see \eqref{indis16} holds, we use the first two claims of \eqref{indis14}. Due to the unit density property, 
recall that for a.e$.$ $t$, we have $\|V_t\|=\|\nabla\{\varphi(\cdot,t)=1\}\|={\mathcal H}^{n-1}\lfloor_{\partial^*\{\varphi(\cdot,t)=1\}}$, where
$\partial^{*}A$ denotes the reduced boundary of $A$ (see \cite{Giusti}). Let $\nu_n$ 
be the $x_n$ component of the inward pointing unit normal vector of 
$\partial^*\{\varphi(\cdot,t)=1\}$. We apply the generalized divergence theorem valid for sets of finite perimeter,
in this case, $\{\varphi(\cdot,t)=1\}\cap \{x_n\leq r/3\}$. 
Then we have for a.e$.$ $t\in [0,\Lambda r^2]$
\begin{equation}
\begin{split}
\int \phi_{P,R}^2\, d\|V_t\|&\geq \int_{\partial^*\{\varphi(\cdot,t)=1\}} \nu_n \phi_{P,R}^2\, d{\mathcal H}^{n-1} \\
&=-\int_{\{\varphi(\cdot,t)=1\}\cap \{x_n\leq r/3\}}\partial_{x_n}\phi_{P,R}^2\, dx+\int_{\{x_n=r/3\}} \phi_{P,R}^2\,
d{\mathcal H}^{n-1}=R^{n-1}{\bf c}
\end{split}
\end{equation}
since $\phi_{P,R}^2$ does not depend on $x_n$ and by the definition of ${\bf c}$. In particular, we have proved
\eqref{indis16}. Now we are ready to apply \cite[Theorem 8.7]{kasaitonegawa}. For all sufficiently small $\varepsilon>0$,
we have seen that we may choose $r_2$ independent of $x\in M_0$ such that all the 
assumptions of \cite[Theorem 8.7]{kasaitonegawa} hold on $(B_{r/2}\times{\mathbb R})\times [0,\Lambda r^2]$ for
all $r\leq r_2$. The conclusion is that in $B_{\sigma_{KT} R}^{n-1}\times{\mathbb R}$ and for $t\in (
(\Lambda_{3,KT}-1/4)R^2,(\Lambda_{3,KT}+1/4)R^2)$, ${\rm spt}\,\|V_t\|$ is represented as a graph $F(\cdot,t)$ of 
$C^{1,\zeta}$ function and it is $C^{(1+\zeta)/2}$ in time, with $|\nabla F|+R^{-1}|F|$ bounded by a constant 
multiple of $\varepsilon$ (see (8.89) of \cite{kasaitonegawa}). The argument up to this point can be carried out
for each point on $x\in M_0$ uniformly and ${\rm spt}\,\|V_t\|$ can be covered by such graphs. This shows that
for all small $t>0$, ${\rm spt}\, \|V_t\|$ is $C^{1,\zeta}$ everywhere. We have the local graph representation
as claimed in (4) and $t^{-1/2}{\rm dist}\, ({\rm spt}\,\|V_t\|, M_0)
\to 0$ as $t\to 0$. 
It is possible that ${\rm spt}\,\|V_t\|$ remains $C^{1,\zeta}$ for some more time, and let $T_2$ be the maximal time
without non-$C^{1,\zeta}$ regular point. In case that $u$ is $\alpha$-H\"{o}lder continuous, 
the regularity criterion are the same
(see \cite[Theorem 3.6]{tonegawa2012}) except that the constant corresponding to $\varepsilon_{6,KT}$ may need to be smaller there.
Thus, in this case, there is a short initial time interval such that ${\rm spt}\,\|V_t\|$ is a $C^{2,\alpha}$ hypersurface. 
This ends the proof of (2) and (4).
\end{proof}
\section{Final remarks}
\subsection{Non-uniqueness}
The solution may be non-unique without having singularities of $M_t$, as a simple example
demonstrates. 
An example such as $M_0=\{x_2=0\}\subset{\mathbb T}^2$ and $u(x_1,x_2)=(0,\sqrt{|x_2|})\in (W^{1,p}({\mathbb T}^2))^2$  ($p<2$) has an 
obvious ODE-level non-uniqueness. 
Thus, on top of the non-uniqueness issues generally associated with singularity occurrences of the MCF, one has
far richer source of possible non-uniqueness with irregular $u$, even though we have a local
regularity theory. It is interesting to investigate how generic the uniqueness may hold 
for the flow in this paper with respect to the initial data and the transport term. 
We mention that there is a nice generic property for the MCF besides the existence
of unique viscosity solution. If $M_0$ is $C^2$ and $d_0$ is the
signed distance function to $M_0$, then the viscosity solution for the MCF starting from
$\{d_0=s\}$ in the sense of \cite{giga1991,evans-spruck1991} 
is a unit density Brakke MCF for a.e$.$ $s\in (-r_0,r_0)$, 
where $r_0>0$ is some small number depending on $M_0$ \cite{evans-spruck1995}. 
For such level set, a phenomena called fattening does not occur in particular. 
It is interesting to see if there is some generalization of this type to the setting of this paper.
\subsection{Structure of singularities}
There have been intensive effort to understand the nature of singularities for the MCF in
recent years. A particular emphasis has been placed on the mean convex flow 
and we mention
names of Andrews, Huisken, Sinestrari and White who analyzed 
structure of singularities in depth. We mention a recent work by Haslhofer
and Kleiner \cite{haslhofer} for a streamlined treatment of the regularity theory of mean convex
flows as well as up-to-date references. 
Note that many of the techniques used by White such as the dimension
reducing and stratification of singularities \cite{white} may be used for the flow in this paper. 
While there may be some limitation compared to the mean convex flow, 
it is interesting and challenging problem to investigate the singularities in the setting 
of the present paper. 


\begin{thebibliography}{99}
\bibitem{allard}
Allard, W.,
\textit{On the first variation of a varifold}.
Ann. of Math. {\bf 95}(2), 417--491 (1972)
\bibitem{allen}
Allen, S. M., Cahn, J. W., 
\textit{A macroscopic theory for antiphase boundary motion and its application to antiphase domain coarsening}.
Acta. Metal. {\bf 27} 1085--1095 (1979)
\bibitem{almgren} Almgren, F. J., Taylor, J. E., Wang, L.-H., \textit{Curvature-driven flows:
a variational approach}. SIAM J. Control Optim. {\bf 31}(2) 387--438 (1993)
\bibitem{Ambrosio1} 
Ambrosio, L.,
\textit{Geometric evolution problems, distance function and viscosity solutions}.
Calculus of variations and partial differential equations (Pisa, 1996), 5--93, Springer, Berlin (2000)
\bibitem{Bellettini}
Bellettini, G.,
\textit{Lecture notes on mean curvature flow: barriers and singular perturbations}.
Edizioni Della Normale, {\bf 12}, Scuola Normale Supriore, Pisa (2014)
\bibitem{brakke}
Brakke, K.,
\textit{The motion of a surface by its mean curvature}. Mathematical notes 20,
Princeton University Press, Princeton (1978)
\bibitem{bron1}
Bronsard, L., Kohn, R. V.,
\textit{Motion by mean curvature as the singular limit of Ginzburg-Landau dynamics}.
J. Diff. Eqns. {\bf 90}(2), 211--237 (1991)
\bibitem{bron2}
Bronsard, L., Stoth, B.,
\textit{On the existence of high multiplicity interfaces}.
Math. Res. Lett. {\bf 3}(1), 41--50 (1996)
\bibitem{chen}
Chen, X.,
\textit{Generation and propagation of interface in reaction-diffusion equations}.
J. Diff. Eqns. {\bf 96}(1), 116--141 (1992)
\bibitem{XYChen}
Chen, X.-Y., 
\textit{Dynamics of interfaces in reaction diffusion systems}.
Hiroshima Math. J. {\bf 21}(1), 47--83 (1991)
\bibitem{giga1991}
Chen, Y.-G., Giga, Y., Goto, S.,
\textit{Uniqueness and existence of viscosity solutions of generalized mean curvature flow equations}.
J. Differe. Geom. {\bf 33}(3), 749--786 (1991)
\bibitem{Colding} 
Colding, T. H., Minicozzi II, W. P.,
\textit{Minimal surfaces and mean curvature flow}.
Surveys in geometric analysis and relativity, Adv. Lect. Math. 
{\bf 20}, 73--143, Int. Press, Somerville, MA, (2011)
\bibitem{mot}
de Mottoni, P., Schatzman, M.,
\textit{Evolution g\'{e}om\'{e}tric d'interfaces}.
C.R. Acad. Paris Sci. S\'{e}r. I Math. {\bf 309}(7), 453--458 (1989)
\bibitem{Ecker} 
Ecker, K.,
\textit{Regularity theory for mean curvature flow}.
Progress in Nonlinear Differential Equations and their Applications 
{\bf 57}, Birkh\"{a}user Boston, Inc., Boston, MA, (2004)
\bibitem{ESS}
Evans, L. C., Soner, H. M., Souganidis, P. E.,
\textit{Phase transitions and generalized motion by mean curvature}.
Comm. Pure Appl. Math. {\bf 45}(9), 1097--1123 (1992)
\bibitem{evans-spruck1991}
Evans, L. C., Spruck, J.,
\textit{Motion of level sets by mean curvature I}.
J. Differe. Geom. {\bf 33}(3), 635--681 (1991)
\bibitem{ES2}
Evans, L. E., Spruck, J.,
\textit{Motion of level sets by mean curvature II}.
Trans. Amer. Math. Soc. {\bf 330}(1), 321--332 (1992)
\bibitem{evans-spruck1995}
Evans, L. C., Spruck, J.,
\textit{Motion of level sets by mean curvature. IV}. 
J. Geom. Anal. {\bf 5}(1), 77--114 (1995)
\bibitem{Fife}
Fife, P. C.,
\textit{Dynamics of internal layers and diffusive interfaces}. 
CBMS-NSF Reg. Conf. Ser. Applied Math., {\bf 53}, SIAM, Philadelphia, PA, (1988)
\bibitem{Giga} 
Giga, Y.,
\textit{Surface evolution equations. A level set approach}.
Monographs in Mathematics, {\bf 99}, Birkh\"{a}user Verlag, Basel, (2006)
\bibitem{GG}
Giga, Y., Goto, S.,
{\textit Geometric evolution of phase-boundaries}, On the evolution of phase 
boundaries (Minneapolis, MN, 1990-91), IMA Vol. Math. Appl., {\bf 43}, 51--65 (1992)
\bibitem{GGIS}
Giga, Y., Goto, S., Ishii, H., Sato, M.-H.,
\textit{Comparison principle and convexity preserving properties for singular degenerate parabolic
equations on unbounded domains}. Indiana Univ. Math. J. {\bf 40}(2), 443--470 (1991)
\bibitem{GT}
Gilbarg, D., Trudinger, N. S.,
\textit{Elliptic partial differential equations of second order}. 
2nd ed., Springer (1983)
\bibitem{Giusti}
Giusti, E., 
\textit{Minimal surfaces and functions of bounded variation}. 
Monographs in mathematics 80, Birkh\"{a}user (1984)
\bibitem{haslhofer}
Haslhofer, R., Kleiner, B.,
\textit{Mean curvature flow of mean convex hypersurfaces}.
arXiv:1304.0926
\bibitem{huisken1990}
Huisken, G.,
\textit{Asymptotic behavior for singularities of the mean curvature flow}.
J. Differe. Geom. {\bf 31}(1), 285--299 (1990)
\bibitem{tonegawa2000}
Hutchinson, J. E., Tonegawa, Y.,
\textit{Convergence of phase interfaces in
the van der Waals-Cahn-Hilliard theory}.
Calc. Var. PDE {\bf 10}(1), 49--84 (2000)
\bibitem{ilmanen1993}
Ilmanen, T.,
\textit{Convergence of the Allen-Cahn equation to Brakke's motion by mean curvature}.
J. Differe. Geom. {\bf 38}(2), 417--461(1993)
\bibitem{ilmanen1994}
Ilmanen, T.,
\textit{Elliptic regularization and partial regularity for motion by mean curvature}. 
Mem. Am. Math. Soc. {\bf 108}(520) (1994)
\bibitem{kasaitonegawa}
Kasai, K. and Tonegawa, Y.,
\textit{A general regularity theory for weak mean curvature flow},
Calc. Var. PDE {\bf 50}(1), 1--68 (2014)
\bibitem{kim}
Kim, Y., Lai, M.-C., Peskin, C.,
\textit{Numerical simulations of two-dimensional foam by the immersed boundary method}.
J. Comput. Phys. {\bf 229}(13), 5194--5207 (2010) 
\bibitem{ladyzhenskaja}
Lady{\v{z}}enskaja, O. A., Solonnikov, V. A., Ural'ceva, N. N.,
\textit{Linear and quasi-linear equations of parabolic type}.
Amer. Math. Soc. (Translated from Russian) (1968) 
\bibitem{tonegawa2010}
Liu, C., Sato, N., Tonegawa, Y.,
\textit{On the existence of mean curvature flow with transport term}.
Interfaces Free Bound. {\bf 12}(2), 251--277 (2010)
\bibitem{luckhaus}
Luckhaus, S., Sturzenhecker, T., 
\textit{Implicit time discretization for the mean curvature flow equation}. 
Calc. Var. PDE {\bf 3}(2), 253--271 (1995)
\bibitem{Mante}
Mantegazza, C.,
\textit{Lecture notes on mean curvature flow}.
Progress in Mathematics, {\bf 290}. Birkh\"{a}user/Springer Basel AG, Basel, (2011)
\bibitem{meyers1977}
Meyers, N. G., Ziemmer, W. P.,
\textit{Integral inequalities of Poincar\'{e} and Wirtinger type for BV functions}.
Amer. J. Math. {\bf 99}(6), 1345--1360 (1977)
\bibitem{mugnai}
Mugnai, L., R\"{o}ger, M.,
\textit{Convergence of perturbed Allen-Cahn equations to forced mean curvature
flow}. Indiana Univ. Math. J. {\bf 60}(1), 41--75 (2011)
\bibitem{roger-schazle}
R{\"o}ger, M., Sch{\"a}tzle, R.,
\textit{On a modified conjecture of De Giorgi}. 
Math. Z. {\bf 254}(4), 675--714 (2006)
\bibitem{rub}
Rubinstein, J., Sternberg, P., Keller, J. B.,
\textit{Fast reaction, slow diffusion and curve shortening}.
SIAM J. Appl. Math. {\bf 49}(1), 116--133 (1989)
\bibitem{sato}
Sato, N., 
\textit{A simple proof of convergence of the Allen-Cahn equation to Brakke's
motion by mean curvature}. Indiana Univ. Math. J. {\bf 57}(4), 1743--1751 (2008)
\bibitem{simon}
Simon, L.,
\textit{Lectures on Geometric Measure Theory}.
Proc. Centre Math. Anal. Austral. Nat. Univ. {\bf 3} (1983)
\bibitem{soner1}
Soner, H. M., 
\textit{Convergence of the phase-field equations to the Mullins-Sekerka
problem with kinetic undercooling}. Arch. Rational Mech. Anal. {\bf 131}(2),
139--197 (1995)
\bibitem{soner2}
Soner, H. M.,
\textit{Ginzburg-Landau equation and motion by mean curvature. I. Convergence}.
J. Geom. Anal. {\bf 7}(3), 437--475 (1997)
\bibitem{takasao}
Takasao, K.,
\textit{Gradient estimates and existence of mean curvature flow with transport term}.
Diff. Integral Eq. {\bf 26}(1-2), 141--154 (2013)
\bibitem{tonegawa2003}
Tonegawa, Y.,
\textit{Integrality of varifolds in the singular limit of reaction-diffusion equations}.
Hiroshima Math. J. {\bf 33}(3), 323--341 (2003)
\bibitem{tonegawa2012}
Tonegawa, Y., 
\textit{A second derivative H\"{o}lder estimate for weak mean curvature flow}. 
Adv. Calc. Var. {\bf 7}(1), 91--138 (2014)
\bibitem{white}
White, B.,
\textit{Stratification of minimal surfaces, mean curvature flows, and harmonic maps}.
J. Reine Angew. Math. {\bf 488}, 1--35 (1997)
\bibitem{White4} White, B.,
\textit{Evolution of curves and surfaces by mean curvature}. Proceedings of the ICM, {\bf 1} (Beijing, 2002), 525--538, Higher Ed. Press, Beijing, (2002)
\bibitem{ziemer1989}
Ziemer, W. P.,
\textit{Weakly differentiable functions}. 
Springer-Verlag (1989).
\end{thebibliography}
\end{document}